\newif\ifSingleColumn
\SingleColumnfalse
\documentclass[9pt,shortpaper,twoside,web]{ieeecolor} 

\makeatletter
\newcounter{ALC@unique}
\makeatother
\let\labelindent\relax
\usepackage{amsmath,amssymb,amsfonts}
\usepackage[draft=true, clevethm=true, linkcolor=NavyBlue, mparhack=false, microtype=false,builddir=.]{defaultpackages}

\usepackage{mynotation}
\usepackage{scenariotree}

\pgfplotsset{compat=1.16}
\renewcommand{\child}[1]{\operatorname{ch}\left(#1\right)}			
\renewcommand{\anc}[1]{\operatorname{anc}\left(#1\right)}				
			
\renewcommand{\nodes}[1]{\operatorname{nod}\left(#1\right)}			
\renewcommand{\hat}{\widehat}
\newcommand{\Kinf}{\mathcal{K}_{\infty}}
\newcommand{\fa}{\tilde{f}}

\newcommand{\confdyn}{C}
\newcommand{\RAOCP}{DR-OCP}
\newcommand{\nconst}{n_g}
\newcommand{\rv}{\xi}
\newcommand{\rAmb}{r}
\providecommand{\DR}[1]{\hat{#1}}
\newcommand{\fc}{\fa^{\DR{\law}_\hor}}

\newcommand{\conf}{\beta}
\newcommand{\confb}{\boldsymbol{\conf}}
\newcommand{\nbeta}{n_{\conf}}
\newcommand{\augset}{\mathcal{Z}}
\newcommand{\Y}{\mathcal{Y}}

\newcommandx{\Vht}[2][2=\hor]{\DR{\cost}_{#2}^{(#1)}}  
\newcommandx{\probt}[2][1=t,2=\hor]{\gamma_{#2}^{(#1)}}  
\newcommandx{\evtCost}[2][1=t,2=\hor]{B_{#2}^{(#1)}}  
\newcommandx{\eps}[2][1=t,2=\hor]{\varepsilon_{#2}^{(#1)}}  
\newcommandx{\evtCov}[1][1=t]{A^{(#1)}}  
\newcommand{\Uht}[1]{\DR{\Ufeas}^{(#1)}}  
\newcommandx{\Vtt}[2][2=\hor]{Q_{#2}^{(#1)}} 
\newcommandx{\Vtil}[1][1=\hor]{Q_{#1}}
\newcommand{\Vtrue}{{\cost}_{\hor}^{\star}}  
\newcommand{\useq}{\bm{u}}
\newcommand{\xseq}{\bm{x}}

\newcommand{\Qf}{Q_{\mathrm{f}}}

\newcommandx{\D}[2][1={},2={}]{
    \mathcal{D}\ifstrempty#1{}{
        _{\scriptscriptstyle\mathrm{#1}
            \ifstrempty{#2}{}{
                ,#2
            }
        }
    }
}
\newcommand{\DKL}{\D[KL]}  
\newcommand{\ambdia}{\delta}  

\newcommand{\Wker}{K}

\newif\ifDraft
\newif\ifExtraSteps
\newif\ifJournal
\newif\ifArxiv

\ExtraStepsfalse
\Journalfalse
\Draftfalse
\Arxivfalse

\newcommand{\extrastep}[1]{
\ifExtraSteps
{\color{ForestGreen!60!black} #1}
\else\fi
} 
\Journaltrue

\newcommand{\revision}[2]{
{#2}
}
\newcommand{\rerevision}[2]{
{#2}
}

\usepackage{generic}
\usepackage{textcomp}

\usepackage{graphicx}
\newcommand\thetitle{A General Framework for Learning-Based Distributionally Robust MPC of Markov Jump Systems}
\begin{document}
\title{\thetitle
}
\author{Mathijs~Schuurmans and Panagiotis~Patrinos
  \thanks{M. Schuurmans and P. Patrinos are with the Department 
  of Electrical Engineering (\textsc{esat-stadius}), KU Leuven, 
  Kasteelpark Arenberg 10, 3001 Leuven, Belgium.
  Email: \texttt{\{mathijs.schuurmans, panos.patrinos\}@esat.kuleuven.be}}
  \thanks{This work was supported by: FWO projects: No. G086318N; No. G086518N; Fonds de la Recherche Scientifique -- FNRS, the Fonds Wetenschappelijk Onderzoek--Vlaanderen under EOS Project No. 30468160 (SeLMA), Research Council KU Leuven C1 project No. C14/18/068 and the 
  Ford--KU Leuven Research Alliance project No. KUL0023.
  }
  \thanks{A preliminary version of this work has been presented at the 59th IEEE Conference on Decision and Control \cite{schuurmans_LearningBasedDistributionallyRobust_2020}.}
}
\maketitle

\begin{abstract}                          
We present a \revision{data-driven}{learning} model predictive control (MPC) scheme for chance-constrained Markov jump systems with unknown switching probabilities.
Using samples of the underlying Markov chain, ambiguity sets of transition probabilities are estimated which include the true conditional probability distributions with high probability. These sets are updated online and used to formulate a time-varying, risk-averse optimal control problem. We prove recursive feasibility of the resulting MPC scheme and show that the original chance constraints remain satisfied at every time step. Furthermore, we show that under sufficient decrease of the confidence levels, the resulting MPC scheme renders the closed-loop system mean-square stable with respect to the true-but-unknown distributions, while remaining less conservative than a fully robust approach.
Finally, we show that the \revision{data-driven}{} value function of the learning MPC converges from above to its nominal counterpart as the sample size grows to infinity. We illustrate 
our approach on a numerical example.  \end{abstract}

\section{Introduction}
\subsection{Background, motivation and related work}
Due to the ubiquitous nature of stochastic uncertainty 
in processes arising in virtually all branches of science and engineering, 
control of dynamical systems perturbed by stochastic processes 
is a long-standing topic of research.
Model predictive control (\acs{MPC}\acused{MPC}) -- stochastic \ac{MPC} in particular -- has been a popular and successful tool
in this endeavor, due to its ability to naturally include probabilistic 
information directly into the control design via the cost, the dynamics and the constraints \cite{kouvaritakis_ModelPredictiveControl_2016, mesbah_StochasticModelPredictive_2016,rawlings_ModelPredictiveControl_2017}.
In classical stochastic \ac{MPC}, however, it is typically assumed that 
the distribution of the underlying stochastic process is known, although in
practice, this is usually not the case. If the disturbance takes values on 
a bounded set, the absence of full distributional knowledge can be
taken into account by designing the controller under 
the worst-case realization of the stochastic disturbance. This 
approach is referred to as robust MPC \cite{kouvaritakis_ModelPredictiveControl_2016,rawlings_ModelPredictiveControl_2017}.

An obvious drawback of robust approaches is that the complete
disregard of the probabilistic nature of the disturbance can be rather 
crude, resulting in a tendency for overly conservative decisions. 
As an alternative approach, one may simply compute an empirical 
estimate of the disturbance distribution and replace 
the true value by this estimate in the optimal control 
problem. Although this is a reasonable approach 
given a sufficient amount of data, for more
moderate sample sizes, there may be a significant misestimation
of the underlying distributions---often referred to as \textit{ambiguity}.  
It is well known that this is likely to cause degradation of the 
resulting performance when evaluated on new samples from the 
true distribution. This phenomenon is known as the \emph{optimizer's curse} \cite{mohajerinesfahani_DatadrivenDistributionallyRobust_2018}.
To account for this ambiguity, one could, instead of a point estimate, construct a set of all distributions (an \emph{ambiguity set}) that are in some specific sense consistent with the data.
By accounting for the worst-case distribution within this set, 
the decision maker is protected against the limitations of the finite sample size.

This approach, known as \ac{DR} optimization \cite{dupacova_MinimaxApproachStochastic_1987}, addresses 
the drawbacks of the above approaches by utilizing available data, but 
only to the extent that it is statistically meaningful. 
As more data is gathered online and ambiguity sets get updated 
accordingly, it is expected that these sets will shrink,
so that the optimal decisions gradually become less conservative. 
This, among other desirable properties, has caused an increasing popularity of \ac{DR}
methods in recent years, initially mostly in stochastic programming and operations research communities \cite{mohajerinesfahani_DatadrivenDistributionallyRobust_2018,parys_DataDecisionsDistributionally_2020,
gao_DistributionallyRobustStochastic_2016,
wiesemann_DistributionallyRobustConvex_2014,
bertsimas_DatadrivenRobustOptimization_2018a} and 
more recently in (optimal) control \revision{}{\cite{schuurmans_SafeLearningBasedControl_2019,coppens_DatadrivenDistributionallyRobust_2020,
yang_WassersteinDistributionallyRobust_2018,hakobyan_WassersteinDistributionallyRobust_2020,
hakobyan_DistributionallyRobustRisk_2021,
coulson_RegularizedDistributionallyRobust_2019}} as well. See also \cite{rahimian_distributionally_2019} for a comprehensive review.
Much of the earlier work focuses on the study of particular classes of 
ambiguity sets, each modelling certain structural assumptions on the 
underlying distribution. \revision{Our analysis, however,}{Although most of our analysis} does not require a particular family of ambiguity sets, 
\revision{We illustrate this in \Cref{sec:learning}, by reviewing some commonly used ambiguity set classes and 
showing how they fit into our proposed framework.}{we will, for concreteness, put particular emphasis on ambiguity sets that are written as a divergence ball around an empirical estimate, as this family of sets is a natural choice in the setting at hand. This is described in \Cref{sec:learning}, 
where a table containing several choices for the divergence is provided.}

As the focus of research in data-driven and learning-based control is gradually shifting towards real-life, safety-critical applications, there has been an increasing concern for safety guarantees of data-driven methods, which are valid in a finite data regime.
This has led to a variety of different approaches besides distributionally robust methodologies,
each valid under different assumptions on the data-generating process and the controlled systems.
For instance, this has led to data-driven variants of tube-based \ac{MPC} \cite{aswani_ProvablySafeRobust_2013,hewing_ScenarioBasedProbabilisticReachable_2020}, Gaussian-process based estimation with reachability-based safe set constraints\cite{fisac_GeneralSafetyFramework_2019}, Data-enabled predictive control (``DeePC'') \cite{coulson_DataEnabledPredictiveControl_2019} combining Willems' fundamental lemma with \ac{MPC} for linear systems, \revision{}{or techniques based on Koopman operators \cite{zhang_RobustLearningBasedPredictive_2022}}.
We refer to \cite{hewing_learning-based_2020} for a recent survey.

In this work, we allow for general (possibly nonlinear) dynamics under stochastic disturbances with unknown distribution, and subject to 
chance constraints. 
However, we restrict our attention to finitely-supported stochastic disturbances.
One of the advantages of this construction is that the predicted evolution
of the system can be represented on a \emph{scenario tree}, which 
allows us to explicitly (and without approximation) optimize over closed-loop control policies, rather 
than open-loop sequences. This property helps combat excessive conservatism due to accumulation of uncertainty over the prediction horizon \cite{bernardini_StabilizingModelPredictive_2012,bernardini_ScenariobasedModelPredictive_2009,lucia_MultistageNonlinearModel_2013}.
Motivated by similar considerations, \cite{leidereiter_QuadraturebasedScenarioTree_2014a} and \cite{bonzanini2020safe} utilize scenario trees to approximate the realizations of continuous disturbances. 
\cite{bonzanini2020safe} then considers safety separately by projecting the computed control action onto a set of control actions that keep the state within safe \ac{RCI} set, similarly to \cite{fisac_GeneralSafetyFramework_2019}. This projection requires the additional solution of a \ac{MIQP}, whenever the used \ac{RCI} set is polyhedral. 
In our setting, however, we consider the switching behavior inherent to 
the system, allowing us to provide safety guarantees directly through the 
application of \ac{MPC} theory on the joint controller-learner system.

We will in particular assume that the underlying disturbance process is a Markov 
chain, leading to a system class commonly referred to as \emph{Markov jump systems}. 
Control of this class of systems has been widely studied and has been used to model systems stemming from a wide range of applications \cite{costa_discrete-time_2005, patrinos_StochasticModelPredictive_2014,
bernardini_ScenariobasedModelPredictive_2009}. 
In the \emph{known distribution} case, stability analysis of nonlinear stochastic \ac{MPC} for this system class has been performed from a worst-case perspective\cite{lucia_StabilityPropertiesMultistage_2020}, in mean-square sense\cite{patrinos_StochasticModelPredictive_2014} and in the more general risk-square sense \cite{sopasakis_RiskaverseModelPredictive_2019,singh_FrameworkTimeConsistentRiskSensitive_2018}.
\revision{}{We emphasize here the distinction between risk-averse and \ac{DR} approaches, where the former optimizes a given \emph{coherent risk measure} with respect to the true distribution, whereas the latter constructs a data-driven ambiguity set with respect to which the stochastic cost is robustified. By the dual risk representation \cite[Thm. 6.4]{shapiro2009lectures}, every ambiguity set induces some coherent risk measure and vice-versa, leading both 
approaches to solve the same class of optimization problems. However, the statistical interpretation and thus, the corresponding guarantees differ significantly.}

\revision{By the dual interpretation of risk measures \cite[Thm. 6.4]{shapiro2009lectures}, }{Indeed, by the mentioned equivalence, 
the notion of risk-square stability in \cite{sopasakis_RiskaverseModelPredictive_2019} guarantees mean-square stability (MSS) with respect to all the distributions within the `ambiguity set' induced by the used risk measure. In practice, however, this is insufficient to guarantee MSS with respect to the true-but-unknown distribution, as it is impossible to construct a nontrivial ambiguity set that contains the true distribution with certainty.
However, 
}
\revision{We show}{we will show} that by careful design of a data-driven sequence of ambiguity sets
-- which only contain the true distributions with high probability --
this concept can be extended to show MSS\revision{}{, as well as recursive constraint satisfaction} with respect to the true distribution, under some additional assumptions.

Other data-driven methods have been proposed to design controllers 
for \emph{unknown transition probabilities} \cite{beirigo_online_2018,he_reinforcement_2019}. However, these works are restricted to a simpler, \textit{unconstrained} setting involving only linear state-feedback policies. \revision{}{Furthermore, related risk-averse and \ac{DR} techniques have been proposed for Markov decision processes (MDPs) 
\cite{
   derman_DistributionalRobustnessRegularization_2020a,
   xu_DistributionallyRobustMarkov_2010,
   mannor_RobustMDPsKRectangular_2016,
   ahmadi_ConstrainedRiskAverseMarkov_2020
}, although these consider discrete states and actions, allowing one to solve directly the Bellman equation over all admissible policies. Unfortunately, these 
techniques become intractable in the present setting involving continuous states and actions.}

We finally study the convergence of the optimal value function 
of our \revision{data-driven}{learning} controller to the nominal counterpart. 
This property, known as asymptotic consistency, has recently been studied in the stochastic optimization literature 
for (static) distributionally robust optimization problems under Wasserstein ambiguity \cite{cherukuri_ConsistencyDistributionallyRobust_2020,mohajerinesfahani_DatadrivenDistributionallyRobust_2018}.
A common assumption in this line of work is Lipschitz continuity of the cost/constraint functions with respect to the 
random variable.
This assumption is not suitable for our purposes, since we consider discrete random variables $\md \in \W$ for which
a suitable norm may not exist. Instead, we will 
in some cases need to resort to a uniform boundedness assumption,
which serves a similar purpose. 
In the non-convex case, the authors of \cite{cherukuri_ConsistencyDistributionallyRobust_2020} base their analysis on \cite{guo_ConvergenceAnalysisMathematical_2017}, 
in which the ambiguity sets are not assumed to be random. 
An additional assumption is added that the constraint 
boundary has probability zero, such that almost everywhere, the constraint is continuous. This assumption helps in dealing with the discontinuity 
of the step-function at 0 which is inherent to chance constraints. 
Alternatively, the chance constraints can be replaced \revision{}{by}
risk constraints involving the average value-at-risk \cite{nemirovski2012safe}, which also 
circumvents this issue. 
Besides the mentioned differences in assumptions, additional care is required to handle the multistage nature of the stochastic 
optimization problems considered here. \revision{}{Specifically, both the optimal cost and the feasible set are defined recursively through the Bellman operator (see \Cref{sec:theory}), causing more complex 
characterizations of the optimal value function as well as 
reduced freedom in selecting the problem parameters to ensure its required properties as compared to a static two-stage stochastic program.}

\subsection{Contributions}
\revision{}{Summarizing the previous discussion, we highlight the following contributions of our work.}
\begin{inlinelist} 
   \item We present a general \revision{data-driven,}{online learning} \ac{DR}-\ac{MPC} framework for Markov switching systems with unknown transition probabilities. The resulting closed-loop system satisfies the (chance) constraints of the original stochastic problem and allows for online improvement of performance based on observed data. 
   Thus, we extend the recently developed framework of risk-averse \ac{MPC} \cite{sopasakis_RiskaverseModelPredictive_2019,sopasakis_risk-averse_2019c,singh_FrameworkTimeConsistentRiskSensitive_2018} to \revision{a data-driven}{an online learning}
   setting, in which the involved risk measures are selected and calibrated automatically based on their dual (\ac{DR}) interpretation.
   \revision{}{To this end, we formalize the procedure for estimating and updating the corresponding ambiguity sets as a dynamical system, which 
   we refer to as the \textit{learning system}. We present conditions on 
   this learning system to ensure its convergence and to obtain meaningful statistical guarantees on the resulting controllers with respect to the unknown underlying distributions}.
   \item We provide sufficient conditions for recursive feasibility and mean-square stability of the 
   \ac{DR}-\ac{MPC} law, with respect to the true-but-unknown distribution. To this end, we state the problem in terms of an augmented state vector \revision{}{including the state of the previously mentioned \textit{learning system}}.\revision{of constant dimension, which summarizes the available information at every time.}{}
   The dynamics of this so-called \emph{learner state} can be 
   easily expressed for common choices for the ambiguity set. 
   This \ifJournal idea, which is closely related to that of 
   \emph{information states}
   \cite[Ch. 5]{bertsekas_DynamicProgrammingOptimal_2005} 
   \fi allows us to formulate the otherwise time-varying optimal control problem as a dynamic programming recursion, facilitating stability 
   analysis of the original control system and the learning system jointly. 
   \item We provide sufficient conditions under which the value of the \ac{DR} problem converges from above to that of the nominal optimal control problem, extending existing results in stochastic optimization to the \revision{}{constrained,} multi-stage, dynamical setting.
\end{inlinelist}
\subsection{Notation}
Let $\N$ denote the set of natural numbers and $\N_{>0} \dfn \N \setminus \{0\}$. For two naturals $a,b \in \N$ with $a \leq b$, we denote $\natseq{a}{b} \dfn \{ n \in \N \mid a \leq n \leq b \}$ and similarly, we introduce the shorthand $\seq{\md}{a}{b} \dfn (\md_{t})_{t=a}^{b}$ to denote a sequence of variables indexed from $a$ to $b$. We denote the extended real line by $\barre \dfn \Re \cup \{\pm \infty\}$ and the set of \emph{nonnegative} (extended) real numbers by $\Re_+$ (and $\barre_+$). The cardinality of a (finite) set $\W$ is denoted by $|\W|$. We write $f: X \rightrightarrows Y$ to denote that $f$ 
is a \emph{set-valued mapping} from $X$ to $Y$. A function is \ac{lsc} if 
its epigraph is closed. 
Given a matrix $P \in \Re^{n\times m}$, we denote its $(i,j)$'th element by $\elem{P}{i}{j}$ and its $i$'th row as $\row{P}{i} \in \Re^m$. The $i$'th element of a vector $x$ is denoted $\idx{x}{i}$. 
$\vect(M)$ denotes the vertical concatenation of the columns of a matrix $M$. We denote the vector in $\Re^k$ with all elements one as $\1_k \dfn (1)_{i=1}^{k}$ and the probability simplex of dimension $k$ as $\simplex_k \dfn \{ p \in \Re_+^{k} \mid \trans{p} \1_k = 1\}$. We define the function $\1_{x=y} = 1$ if $x=y$ and $0$ otherwise. 
The indicator function $\delta_{X}: \Re^{n} \rightarrow \barre$ of a set $X \subseteq \Re^{n}$ is defined by $\delta_{X}(x) = 0$ if $x \in X$ and $\infty$ otherwise. 
\ifJournal The level set of a function $V: \Re^n \to \barre$ is denoted $\lev_{\leq \varepsilon} V \dfn \{x \in \Re^n \mid V(x) \leq \varepsilon \}$. The interior of a set $X$ is denoted $\itr X$. \fi 
\ifJournal We denote the positive part of a quantity $x$ as $[x]_{+} \dfn \max\{0, x\}$, where $\max$ is taken element-wise.
We say that a function $\phi: \Re_+ \to \Re_+$ belongs to the class of $\Kinf$ functions if it is continuous, strictly increasing, unbounded, and zero at zero \cite{rawlings_ModelPredictiveControl_2017}.
\revision{Given}{Finally, given} a nonempty, proper cone $\cone$, the generalized inequality $a \leqc{\cone} b$ is equivalent to $b - a \in \cone$. $\cone^* \dfn \{y \mid \<x,y\> \geq 0,\, \forall x \in \cone \}$ denotes the dual cone of $\cone$.
\fi
\section{Problem statement and structural assumptions}

Let $\w \dfn (\md_{t})_{t \in \N}$ denote a discrete-time, time-homogeneous Markov chain defined on some probability space\footnote{\rerevision{}{For an explicit construction of $(\Omega, \F, \prob)$, we refer to {\cite[Thm. 8.1]{billingsley_ProbabilityMeasure_1995}}}.} $(\Omega, \F, \prob)$ and taking values on $\W \dfn \natseq{1}{\nModes}$.
The \emph{transition kernel} governing the Markov chain is denoted by $\transmat = (\elem{\transmat}{i}{j})_{i,j \in \W}$, where $\elem{\transmat}{i}{j} = \prob[\md_{t} = j \mid  \md_{t-1} = i]$.
We refer to $w_t$ as the \emph{mode} of the chain at time $t$.
For simplicity, we will assume that the initial mode is known to be $i$, so $p_0=(1_{\md=i})_{\md \in \W}$. 
Therefore, the Markov chain is fully characterized by its transition kernel.
Finally, we will assume that the Markov chain is ergodic.

\begin{assumption}[Ergodicity] \label{assum:ergodicity}
    The Markov chain $(\md_t)_{t\in\N}$ is ergodic, i.e., there 
    exists a value $k \in \N_{>0}$, such that  $\transmat^k > 0$ element-wise\revision{, for some $k \geq 1$}{}.
\end{assumption}
This assumption, stating that every mode is reachable from any 
other mode in $k$ steps, ensures that every mode of the chain gets visited infinitely often \cite[Ex. 8.7]{billingsley_ProbabilityMeasure_1995}. 
This will allow us to guarantee convergence of the proposed \revision{data-driven}{learning} \ac{MPC} scheme to its nominal counterpart.
(See \Cref{sec:theory}.)

We will consider discrete-time systems with dynamics of the form 
\begin{equation} \label{eq:system-dynamics}
    x_{t+1} = f(x_t, u_t, \md_{t+1}), 
\end{equation}
where $x_t \in \Re^{\ns}, u_t \in \Re^{\na}$ are the state and control action at time $t$, respectively. We will assume that the state $x_t$ and mode $\md_t$ are observable at time $t$. This is equivalent to 
the more common notation $x_{t+1}=f(x_t,u_t,w_t)$, assuming $w_{t-1}$ is observable. However, as we will consider $w_{t}$ to be part of 
the system state at time $t$, the notation of \eqref{eq:system-dynamics} 
will be more convenient.

Since $w_t$ is drawn from a 
Markov chain, such systems are commonly referred to as Markov jump systems. Whenever $f(\argdot, \argdot, \md)$ is a linear function, 
\eqref{eq:system-dynamics} describes a Markov jump linear system \cite{costa_discrete-time_2005}. Since the state $x_t$ and mode $ \md_t$ are observable at time $t$, the distribution of $x_{t+1}$ depends solely on the conditional switching distribution $\row{\transmat}{\md_t}$, for a given control action $u_t$.

\revision{
For a given state-mode pair $(x, \md) \in \Re^{\ns} \times \W$, we impose $\nconst$ 
chance constraints of the form
\begin{equation} \label{eq:chance-constraint-original}
\prob[g_{i}(x,u,\md,\mdnxt) > 0 \mid x,\md] \leq \alpha_i, \forall i \in \natseq{1}{\nconst},
\end{equation}
where $\mdnxt \sim \row{\transmat}{\md}$ is randomly drawn from the 
Markov chain $\w$ in mode $w$, and  
$g_i: \Re^{\ns} \times \Re^{\na} \times \W^2 \rightarrow \Re$ are constraint functions with corresponding constraint violation rates $\alpha_i$. By appropriate choices of $\alpha_i$ and $g_i$, constraint \eqref{eq:chance-constraint-original} can be used to encode robust constraints ($\alpha_i = 0$) or chance constraints ($0<\alpha_i<1$) on the state, the control action, or both. 
Note that the formulation \eqref{eq:chance-constraint-original} additionally covers chance constraints on the successor state $f(x,u,\mdnxt)$ under input $u$, conditioned on the current values $x$ and $\md$.
To ease notation, we will without loss of generality assume that $\nconst=1$.
In their standard form, chance constraints lead to nonconvex, nonsmooth (even discontinuous) constraints. For this reason, they are commonly approximated using \emph{risk measures} \cite{nemirovski2012safe}.
Particularly, the (conditional) \emph{average value-at-risk} (at level $\alpha \in [0, 1]$ and with reference distribution $p \in \simplex_\nModes$) of the random variable $\xi: \W^2 \to \Re$ is defined as
\begin{equation} 
    \begin{aligned}
        \label{eq:def-AVAR-old}
            \AVAR_{\alpha_i}^{p} &[\rv(\md,\mdnxt) \mid \md]  
        \ifSingleColumn\else\\\fi
        =&\begin{cases}
            \min\limits_{t\in\Re} 
                t+\nicefrac{1}{\alpha}\E_{p}\left\{\plus{\rv(\md,\mdnxt)-t} \mid \md\right\},&\alpha\neq 0\\
            \max_{\mdnxt \in \W} \left \{ \rv(\md,\mdnxt) \right \},
                                & \alpha=0.
        \end{cases}
    \end{aligned}
\end{equation} 
}
{
For a given state-mode pair $(x, \md)$, we will impose probabilistic 
constraints of the form 
\begin{equation}
    \label{eq:avar-constraint-original}
        \AVAR_{\alpha}^{\row{\transmat}{\md}}
         \big[ g_i(x,u,\md,\mdnxt) \mid x,\md \big] \leq 0,\;\rerevision{}{i \in \natseq{1}{\nconst}} 
\end{equation}
where $\mdnxt \sim \row{\transmat}{\md}$ is randomly drawn from the 
Markov chain $\w$ in mode $w$,  
$g_i: \Re^{\ns} \times \Re^{\na} \times \W^2 \rightarrow \Re$ are constraint functions with corresponding constraint violation rates $\alpha_i$, and $\AVAR$ denotes the (conditional) \emph{average value-at-risk}. The conditional $\AVAR$ (at level $\alpha \in [0, 1]$ and with reference distribution $p \in \simplex_\nModes$) of the random variable $\xi: \W^2 \to \Re$ is defined as
\begin{equation} 
    \begin{aligned}
        \label{eq:def-AVAR}
            \AVAR_{\alpha}^{p} &[\rv(\md,\mdnxt) \mid \md]  
        \ifSingleColumn\else\\\fi
        =&\begin{cases}
            \min\limits_{t\in\Re} 
                t+\nicefrac{1}{\alpha}\E_{p}\left\{\plus{\rv(\md,\mdnxt)-t} \mid \md\right\},&\alpha\neq 0\\
            \max_{\mdnxt \in \W} \left \{ \rv(\md,\mdnxt) \right \},
                                & \alpha=0,
        \end{cases}
    \end{aligned}
\end{equation} 
and it has the useful property that 
}
if $p = \row{\transmat}{\md}$, then the following implication holds tightly \cite[sec. 6.2.4]{shapiro2009lectures}
\begin{equation} \label{eq:risk-constraint-implication} 
    \begin{aligned}
    \AVAR^p_{\alpha}[\rv(\md,\mdnxt) \mid \md] \leq 0
    \Rightarrow \prob[\rv(\md,\mdnxt) \leq 0 \mid \md] \geq 1-\alpha.
    \end{aligned}
\end{equation}
By exploiting the dual risk representation \cite[Thm 6.5]{shapiro2009lectures}, the left-hand inequality in \eqref{eq:risk-constraint-implication} can be formulated in terms of only linear constraints \cite{sopasakis_risk-averse_2019c}. As such, it can be used as a tractable surrogate for chance constraints \revision{}{which would lead to nonconvex, non-smooth constraints \cite{nemirovski2012safe}}.
By appropriate choices of $\alpha_i$ and $g_i$, constraint \eqref{eq:avar-constraint-original} can be used to encode robust constraints ($\alpha_i = 0$) or probabilistic constraints ($0<\alpha_i<1$) on the state, the control action, or both. 
Note that it additionally covers chance constraints on the successor state $f(x,u,\mdnxt)$ under input $u$, conditioned on the current values $x$ and $\md$.
To ease notation, we will without loss of generality assume that $\nconst=1$.
\revision{Consequently, }{To summarize,}
the set of feasible control actions as a function of $x$ and $w$ can be written as 
\begin{equation}
    \label{eq:chance-constraint}
    \Ufeas(x, \md) \dfn \left \{ u \in U:
        \AVAR_{\alpha}^{\row{\transmat}{\md}} \big[ g(x,u,\md,\mdnxt) \mid x,\md \big] \leq 0 \right\},
\end{equation}
where $U \subseteq \Re^{\na}$ is a nonempty, closed set.

Ideally, our goal is to synthesize -- by means of a stochastic \ac{MPC} scheme -- a stabilizing control law $\law_{\hor}: \Re^{\ns} \times \W \rightarrow \Re^{\na}$, such that for the closed loop 
system \(x_{t+1} = f(x_t, \law_{\hor}(x_t,  \md_t), \md_{t+1})\), 
it holds almost surely that $\law_{\hor}(x_t, \md_t) \in \Ufeas(x_t, \md_t)$, for all $t \in \N$. Consider a sequence of $\hor$ control laws $\pol = (\pol_{k})_{k=0}^{\hor-1}$, referred to as a \emph{policy} of length $\hor$.
Given a stage cost $\ell:\Re^{\ns} \times \Re^{\na} \times \W \rightarrow \Re_{+}$,
and a terminal cost $\Vf: \Re^{\ns} \times \W \rightarrow \Re_{+}$ and corresponding terminal set $\Xf$: $\bar{\Vf}(x,w) \dfn \Vf(x, \md) + \delta_{\Xf}(x, \md)$, we can assign to each such policy $\pol$, a cost
\begin{equation} \label{eq:cost-function-stochastic}
    \cost_{\hor}^{\pol}(x, \md) \dfn \textstyle\E\big[\sum_{k=0}^{\hor-1} \ell(x_k, u_k,\md_k) + \bar{\Vf}(x_\hor,\md_\hor) \big], 
\end{equation}
where $x_{k+1} = f(x_k, u_k, \md_{k+1})$, $u_k = \pol_k(x_k, \md_k)$ and $(x_0, \md_0) = (x,\md)$, for $k \in \natseq{0}{\hor-1}$.
This defines the following stochastic \ac{OCP}. 

\begin{definition}[Stochastic \ac{OCP}] \label{def:stochastic-ocp}
For a given state-mode pair $(x, \md)$, the optimal cost of the stochastic \ac{OCP} is 
\begin{subequations} \label{eq:stochastic-OCP}
\begin{align}
    \cost_{\hor}(x, \md) = \min_{\pol} \cost_{\hor}^{\pol}(x, \md)
\end{align}
subject to 
\begin{align} \label{eq:stochastic-constraints}
    x_0&=x, \md_0=\md, \pol = (\pol_{k})_{k=0}^{\hor-1}, \\
    x_{k+1} &= f(x_k, \pol_{k}(x_k,\md_k), \md_{k+1}),\\
    \pol_{k}(x_k, \md_k) &\in \Ufeas(x_k ,\md_k), 
    \; \forall k \in \natseq{0}{\hor-1}.
\end{align}
\end{subequations}
We denote by $\Pi_{\hor}(x, \md)$ the corresponding set of minimizers.
\end{definition}

To ensure existence of a solution to \eqref{eq:stochastic-OCP} (and its DR counterpart,
defined in \Cref{sec:data-driven-MPC}),
we will impose the following (standard) regularity conditions\cite{rawlings_ModelPredictiveControl_2017, sopasakis_RiskaverseModelPredictive_2019}.
\begin{assumption}[Problem regularity] \label{assum:regularity}
    The following are satisfied for all $\md, \mdnxt \in \W$: 
    \begin{conditions}
        \item Functions
        $\ell(\argdot, \argdot, \md): \Re^{\ns} \times \Re^{\na} \to \Re_+$,
        $\Vf(\argdot, \md): \Re^{\ns} \to \Re_+$, 
        $f(\argdot, \argdot, \md)$, and 
        $g(\argdot, \argdot, \md, \mdnxt)$, $i \in \natseq{1}{\nconst}$
        are continuous;
        \item $U$ and $\Xf$ are closed; 
        \item $f(0,0, \md) = 0$, $\ell(0, 0, \md) = 0$, $0 \in \Ufeas(0, \md)$, and $\bar{\Vf}(0,\md)=0$;
        \item One of the following is satisfied: 
        \begin{enumerate}
            \item $U$ is compact; or 
            \item $\ell(x,u,\md) \geq c(\nrm{u})$ with $c \in \Kinf$, for all $(x,u) \in \Re^{\ns} \times U$.
        \end{enumerate} 
    \end{conditions}
    
\end{assumption}
Let $(\pol^{\star}_{k}(x, \md))_{k=0}^{\hor-1} \in \Pi_{\hor}(x, \md)$, so that the stochastic \ac{MPC} control law is given by $\law_{\hor}(x, \md) = \pol^{\star}_0(x, \md)$. 
Sufficient conditions on the terminal cost $\bar{\Vf}$ and its effective 
domain $\dom \bar{\Vf} = \Xf$ to ensure mean-square stability of the 
closed-loop system, have been studied for a similar problem set-up in 
\cite{patrinos_StochasticModelPredictive_2014}, among others.

Both designing and computing such a stochastic \ac{MPC} law requires knowledge of the probability distribution governing the state dynamics \eqref{eq:system-dynamics},
or equivalently, of the transition kernel $\transmat$. 
In the absence of this knowledge, these probabilities are to be estimated 
from a finitely-sized data set and therefore subject to some 
level of ambiguity. Our goal is to devise an \ac{MPC} scheme which 
uses the available data in a principled manner, while explicitly 
taking this ambiguity into account. 

To this end, we will introduce the notion of a \textit{learner state}, 
which is very similar in spirit to the concept of an \emph{information state}, commonly used in control of partially observed Markov decision processes \cite{krishnamurthy_PartiallyObservedMarkov_2016}, where -- in contrast to our approach -- it is typically adopted in a Bayesian setting. In both cases, however, it can be regarded as an internal state of the controller that stores all the information required to build (a set of) conditional distributions over the next state, given the observed data.
\revision{}{
We will make this more precise in the next section.
Equipped with such a learning system,
our aim is to find a data-driven approximation to the stochastic \ac{OCP} defined by \eqref{eq:stochastic-OCP}, which asymptotically attains the optimal cost while preserving stability and constraint satisfaction during closed-loop operation.
}

\revision{
}{
The remainder of this work is organized as follows.
\Cref{sec:learning} formalizes the assumed learning system and presents and several classes of ambiguity sets found in the literature that fit the framework.
In \Cref{sec:data-driven-MPC}, we use this learning system to construct a distributionally robust counterpart to the optimal control problem in terms of the ingredients introduced above.
\Cref{sec:theory} contains a theoretical analysis of the proposed scheme;
and in \Cref{sec:numerical}, we illustrate the approach on some numerical examples.
}
\revision{
We formalize this in the following assumption.

\begin{assumption}[Learning system] \label{assum:learner-old}
Given a sequence $\seq{\md}{0}{t}$ sampled from the Markov chain $\w$, we can compute
\begin{enumerate}
    \item a statistic $\lrn_t: \W^{t+1} \rightarrow \lrnset \subseteq \Re^{n_{s}}$\revision{, with $\lrnset$ compact}{}, accompanied by a vector of confidence parameters $\confb_t = (\idx{\conf_t}{i})_{i=1}^{\nbeta} \in \I \dfn [0,1]^{\nbeta}$, \revision{which admit recursive update rules}{for which there exist some Markovian dynamics $\learner$ and $\confdyn$ such that} $\lrn_{t+1} = \learner(\lrn_t, \confb_t, \md_t, \md_{t+1})$ and $\confb_{t+1} = \confdyn(\confb_t)$, $t \in \N$; 
    \item an \emph{ambiguity set} $\amb: \lrnset \times \W \times [0,1] \rightrightarrows \simplex_{\nModes}: (\lrn, \md, \conf) \mapsto \amb_{\conf}(\lrn, \md)$, mapping $\lrn_t$, $ \md_t$ and the component $\idx{\conf_{t}}{i}$ to a convex subset of the $\nModes$-dimensional probability simplex $\simplex_\nModes$, such that for all $t \in \N$, and for all $i \in \natseq{1}{\nbeta}$,
    \begin{equation} \label{eq:high-confidence-old}
        \prob[\row{\transmat}{\md_{t}} \in \amb_{\idx{\conf_{t}}{i}}(\lrn_t,  \md_t)] \geq 1 - \idx{\conf_{t}}{i}.
    \end{equation}
\end{enumerate}
    We will refer to $\lrn_t$ and $\confb_t$ as the state of the learner and the confidence vector at time $t$, respectively.
\end{assumption}
\revision{}{

}
\sidenote{Somehow, I should make more clear what $\learner$ and $\confdyn$ are at this point. Maybe move this assumption to III?}

\begin{remark}[confidence levels]
    Two points of clarification are in order. 
    First, we consider a vector of confidence levels, rather 
    than a single value. This is motivated by the fact that one would often wish to assign separate confidence levels to ambiguity sets corresponding to the cost function on the one hand; and to those corresponding to the $\nconst$ chance constraints on the other hand (See \Cref{def:raocp}). Accordingly, we will assume that $\nbeta = \nconst + 1$. 
    
    Second, the confidence levels are completely exogenous to the system dynamics and can in principle be chosen to be any time-varying sequence satisfying the technical conditions discussed further (see \Cref{prop:conditions-confidences} and \Cref{assum:confidences}).
    The requirement that the sequence $(\confb_{t})_{t\in\N}$ can be written as the trajectory of a time-invariant dynamical system
    serves to facilitate theoretical analysis of the proposed 
    scheme through dynamic programming. 
\end{remark}

We will furthermore require the following restrictions 
on the choice of the learning dynamics the confidence levels.
\begin{assumption} \label{assum:stationary-old}
    There exists a stationary learner state  
    $\lrn^{\sstar} = \learner(\lrn^{\sstar}, \confb, \md, \mdnxt)$, for all $(\confb, \md, \mdnxt) \in \I \times \W^2$, such that  
    from any initial state $s_0$,
    $\lim_{t \to \infty} s_t = s^{\sstar}$, \as.
\end{assumption}
\begin{assumption} \label{assum:confidences-old}
    The confidence dynamics $\confb_{t+1} = \confdyn(\confb_t)$ is chosen such that 
    \(
        \sum_{t=0}^{\infty} \confb_{t} < \infty,
    \)    
    element-wise.
\end{assumption}
\Cref{assum:stationary} imposes that asymptotically, the 
learner settles down to some value which is no longer 
modified by additional data. It is natural to assume that in such a state, the learner unambiguously models the 
underlying distribution, as demonstrated, for instance, in \Cref{ex:learner-dynamics-TV}. However, without further 
assumptions, \revision{one could also consider}{this also includes} the trivial case where 
$\lrnset = \{s^{\sstar}\}$ and e.g., $\amb_{\beta}(\lrn, \md) = \simplex_{\nModes}$, in which case, no learning occurs and, in fact, a robust \ac{MPC} scheme is recovered.
In \Cref{sec:consistency}, we will pose an additional 
constraint on the learning system, which excludes this 
case, but allows us to show consistency of the \revision{data-driven}{learning} controller. 
\Cref{assum:confidences} states that the probability of obtaining an ambiguity set that contains the true conditional distribution (expressed by \eqref{eq:high-confidence}) increases sufficiently fast. This assumption will be of crucial importance in showing stability (see \Cref{sec:stability}).
To fix ideas, we keep the following example in mind as a suitable choice for the confidence dynamics throughout the article.
\begin{example}[Confidence dynamics] \label{ex:confidence-dynamics-old}
    A suitable family of sequences for the confidence levels satisfying \Cref{assum:confidences} (assuming $n_\conf = 1$ for simplicity)
    is obtained as 
    \[
        \confb_t = b (1+t)^{-q},\, t \in \N,
    \]
    with parameters $0 \leq b \leq 1$, $q > 1$. 
        This sequence can be described by the recursion 
    \(
        \confb_{t+1} = \confdyn(\confb_t) = b \confb_t {(\confb_t^{\nicefrac{1}{q}} + b^{\nicefrac{1}{q}})^{-q}},\, \confb_0 = b.
    \)
    Thus, it additionally satisfies the requirements of \Cref{assum:learner}.

\end{example}

The learner state $\lrn_t$ will in most practical cases 
be composed of a \revision{sufficient statistic}{statistical estimator} for the transition 
kernel and some parameter calibrating the size of the 
ambiguity set, based on \revision{statistical information}{concentration inequalities}. See \Cref{sec:divergence-based-ambiguity-sets} for more details and some concrete examples.

Equipped with a generic learning system of this form,
our aim is to find a data-driven approximation to the stochastic \ac{OCP} defined by \eqref{eq:stochastic-OCP}, which asymptotically attains the optimal cost while preserving stability and constraint satisfaction during closed-loop operation.

The remainder of this work is organized as follows.
\Cref{sec:learning} presents and compares several classes of ambiguity sets found in the literature,
and discusses how they fit in the framework of \Cref{assum:learner}.
In \Cref{sec:data-driven-MPC}, we construct a distributionally robust counterpart to the optimal control problem in terms of the ingredients introduced above.
\Cref{sec:theory} contains a theoretical analysis of the proposed scheme;
and in \Cref{sec:numerical}, we illustrate the approach on a numerical example.
}{}\section{Data-driven ambiguity sets} \label{sec:learning}

\revision{}{
\subsection{Abstract learning system}
As mentioned in the previous section, we model the procedure that 
maps the observed data into a set of transition probabilities 
as a generic Markovian system, which we refer to as the \emph{learning 
system}. We first state the required structure in a compact, abstract notation and 
later provide a concrete example, which will suffice in many practical cases. 
}
\begin{assumption}[Learning system] \label{assum:learner}
  Given a sequence $\seq{\md}{0}{t}$ sampled from the Markov chain $\w$, we can compute
  \begin{enumerate}
      \item a statistic 
      $\lrn_t: \W^{t+1} \rightarrow \lrnset \subseteq \Re^{n_{s}}$, with $\lrnset$ compact,
      accompanied by a vector of confidence parameters
      $\confb_t = (\idx{\conf_t}{i})_{i=1}^{\nbeta} \in \I \dfn [0,1]^{\nbeta}$,
      \revision{which admit recursive update rules $\lrn_{t+1} = \learner(\lrn_t, \confb_t, \md_t, \md_{t+1})$}{for which there exist some Markovian dynamics $\learner$ and $\confdyn$ such that
      $\lrn_{t+1} = \learner(\lrn_t, \confb_t, \md_t, \md_{t+1})$}
      and
      $\confb_{t+1} = \confdyn(\confb_t)$, $t \in \N$; 
      \item an \emph{ambiguity set}
      $\amb: \lrnset \times \W \times [0,1] \rightrightarrows \simplex_{\nModes}: (\lrn, \md, \conf) \mapsto \amb_{\conf}(\lrn, \md)$,
      mapping
      $\lrn_t$, $ \md_t$ and the component
      $\idx{\conf_{t}}{i}$ to a convex subset of the
      $\nModes$-dimensional probability simplex 
      $\simplex_\nModes$, such that for all
      $t \in \N, \md \in \W$, and for all $i \in \natseq{1}{\nbeta}$,
      \begin{equation} \label{eq:high-confidence}
          \prob[\row{\transmat}{\md} \in \amb_{\idx{\conf_{t}}{i}}(\lrn_t,  \md)] \geq 1 - \idx{\conf_{t}}{i}.
      \end{equation}
  \end{enumerate}
      We will refer to $\lrn_t$ and $\confb_t$ as the learner state and the confidence vector at time $t$, respectively.
  \end{assumption}
  \revision{}{
  \begin{remark}[Learner dynamics $\learner$, $\confdyn$] \label{rem:learner-sys}
      The existence of the dynamics $\learner$ and $\confdyn$ implies that 
      the system with the augmented state consisting of both the original system state-mode pair $(x_t, w_t)$ and the learner-confidence pair
      $(\lrn_t, \confb_t)$, is Markovian.
      This assumption aids the theoretical analysis in \Cref{sec:theory} and 
      is not restrictive in practice, as it essentially only requires that 
      finite memory is needed for the method, which is the case
      for all implementable methods. 
      For concreteness, typical examples for $\learner$ and $\confdyn$, which are valid for many practical use cases, are presented in \Cref{ex:learner-dynamics} and \Cref{ex:confidence-dynamics}, respectively. 
  \end{remark}
  }
  
  \begin{remark}[confidence levels]
    \revision{
      Two points of clarification are in order. 
      First, }{}
    We consider a vector of confidence levels, rather 
      than a single value.
      This is motivated by the fact that one would often wish to assign separate confidence levels to
      ambiguity sets corresponding to the cost function, and to those corresponding to the $\nconst$ chance constraints (See \Cref{def:raocp}). Accordingly, we will assume that $\nbeta = \nconst + 1$. 
    \revision{  
      Second, the confidence levels are completely exogenous to the system dynamics and can in principle be chosen to be any time-varying sequence satisfying the technical conditions discussed further (see \Cref{prop:conditions-confidences} and \Cref{assum:confidences}).
      The requirement that the sequence $(\confb_{t})_{t\in\N}$ can be written as the trajectory of a time-invariant dynamical system
      serves to facilitate theoretical analysis of the proposed 
      scheme through dynamic programming.
    }{}
  \end{remark}
  In order to ensure reasonable behavior of the learning system, we 
  impose the following restrictions 
  on the choice of the learning dynamics and the confidence levels.
  \begin{assumption} \label{assum:stationary}
      There exists a stationary learner state \revision{}{$\lrn^{\sstar}$}
      such that $\lrn^{\sstar} = \learner(\lrn^{\sstar}, \confb, \md, \mdnxt)$, for all $(\confb, \md, \mdnxt) \in \I \times \W^2$, and that from any initial state $s_0$,
      $\lim_{t \to \infty} s_t = s^{\sstar}$, \as.
  \end{assumption}
  \begin{assumption} \label{assum:confidences}
      The confidence dynamics $\confb_{t+1} = \confdyn(\confb_t)$ is chosen such that 
      \revision{
      \(
          \sum_{t=0}^{\infty} \confb_{t} < \infty,
      \)    
      element-wise.
      }{
      \begin{inlinelist*}
        \item \label{cond:conf-summable} \(\sum_{t=0}^{\infty} \confb_{t} < \infty\)
        \item \label{cond:conf-slow}\(\lim_{t\to \infty} \tfrac{\log \confb_t}{t}=0\),
      \end{inlinelist*}
      element-wise.
      }
  \end{assumption}
  \Cref{assum:stationary} imposes that asymptotically, the 
  learner settles down to some value which is no longer 
  modified by additional data.
  It is natural to expect that in such a state, 
  \revision{
  the learner unambiguously models the underlying distribution, as demonstrated, for instance, in \Cref{ex:learner-dynamics-TV}}
  {
  the learner has acquired perfect knowledge of the underlying transition 
  kernel and the ambiguity sets $\amb_{\conf_{t,i}}(\lrn^\star, \md)$, $i \in \natseq{1}{\nbeta}$ have 
  all converged to a singleton. 
  }
  However, \revision{without further 
  assumptions,}{
    this is not necessarily the case. For instance,
  } \revision{one could also consider}{} the trivial case where 
  $\lrnset = \{s^{\sstar}\}$ and
  $\amb_{\beta}(\lrn, \md) = \simplex_{\nModes},\, \forall (\lrn, \md) \in \lrnset \times \W$ \revision{}{satisfies \Cref{assum:stationary}, but under these conditions}\revision{, in which case}{}, no learning occurs and, in fact, a robust \ac{MPC} scheme is recovered.
  In \Cref{sec:consistency}, we will pose an additional 
  assumption on the learning system, which excludes this 
  case, but allows us to show consistency of the \revision{data-driven}{learning} controller.
  
  \Cref{assum:confidences} states that the probability of obtaining an ambiguity set that contains the true conditional distribution (expressed by \eqref{eq:high-confidence}) increases sufficiently fast \revision{}{(condition \ref{cond:conf-summable})}. This assumption will be of crucial importance in showing stability (see \Cref{sec:stability}).
  \revision{}{In addition, it places a lower bound on the convergence rate 
  of the confidence levels (condition \ref{cond:conf-slow}), which 
  is crucial in establishing asymptotic consistency of the scheme (see \Cref{sec:consistency}), 
  since it will allow convergence of the ambiguity sets, as we discuss in \Cref{rmrk:radius-state}.
  }
  To fix ideas, we keep the following example in mind as a suitable choice for the confidence dynamics throughout the article.
  \begin{example}[Confidence dynamics] \label{ex:confidence-dynamics}
      A suitable family of sequences for the confidence levels satisfying \Cref{assum:confidences} (assuming $n_\conf = 1$ for simplicity\footnote{For $n_\conf > 1$, the same construction can be repeated element-wise.})
      is obtained as 
      \begin{equation} \label{eq:conf-time-dependent}
          \confb_t = b (1+t)^{-q},\, t \in \N,
      \end{equation}
      with parameters $0 < b \leq 1$, $q > 1$. \revision{}{Indeed, as \rerevision{$q>0$}{$q>1$}, this sequence is summable, and furthermore 
      \[ 
        \lim_{t \to \infty} \frac{- q \log (b (1+t))}{t} = 0. 
      \]
      Using $\eqref{eq:conf-time-dependent}$, a straightforward calculation reveals that $\confb_t$ can be updated recursively: 
          \[ 
              \begin{aligned}
                \confb_{t+1}^{-\nicefrac{1}{q}} &= b^{-\nicefrac{1}{q}} (2 + t)
                = b^{-\nicefrac{1}{q}} + \confb_{t}^{-\nicefrac{1}{q}}\\
              \iff \confb_{t+1} &= (b^{-\nicefrac{1}{q}} + \confb_{t}^{-\nicefrac{1}{q}})^{-q}\\
              &= b \confb_t (b^{\nicefrac{1}{q}} + \confb_{t}^{\nicefrac{1}{q}})^{-q} \nfd C(\confb_t),
              \end{aligned}
          \]
      Thus, it additionally satisfies the requirements of \Cref{assum:learner}.}
  \end{example}

  The learner state $\lrn_t$ will in most practical cases 
  be composed of a \revision{sufficient statistic}{data-driven estimator} for the transition 
  kernel and some parameter calibrating the size of the 
  ambiguity set, based on statistical information.
  \revision{}{
  Indeed, ambiguity sets are very often defined as the set of distributions that lie within some radius from an empirical estimate using a particular distance metric or divergence.
  We will refer to such ambiguity sets as 
  \emph{divergence-based} ambiguity sets.
  For the current setting concerning finitely supported distributions, two notable examples of such divergences are the \ac{TV} metric 
  \cite{schuurmans_SafeLearningBasedControl_2019,jiang_RiskAverseTwoStageStochastic_2018,rahimian_EffectiveScenariosMultistage_2021} and the \ac{KL} divergence
   \cite{parys_DataDecisionsDistributionally_2020}. 
  In the following section, we show that these divergences can be used 
  to design a learning system satisfying our assumptions, and illustrate that from these two cases, several other divergence-based ambiguity sets 
  can be constructed.    
  }
  \revision{For general, continuous distributions, popular choices for the distance metric/divergence include the Wasserstein distance \cite{mohajerinesfahani_DatadrivenDistributionallyRobust_2018,yang_WassersteinDistributionallyRobust_2018} or moment-based ambiguity sets \cite{delage_DistributionallyRobustOptimization_2010b,coppens_DatadrivenDistributionallyRobust_2020}, where the first two moments of the distributions are confined to a ball around the empirical estimate.}{}
  
\revision{

\sidenote[
For the setting involving finitely supported distributions, \cite{wang_LikelihoodRobustOptimization_2016a} 
proposes \emph{likelihood regions}:
ambiguity sets containing 
all distributions with respect to which the likelihood of observed data is larger than some threshold $\alpha$.
\cite{wang_LikelihoodRobustOptimization_2016a} provides a data-driven estimate for $\alpha$ to satisfy a condition similar to \eqref{eq:high-confidence} using asymptotic results.
However, a modification 
to provide finite sample guarantees is straightforward. 
Closely related to this family of ambiguity sets are defined by considering distributions that are close to the empirical distribution 
as measured by the \ac{KL} divergence. Depending on the ordering 
of the arguments in the KL divergence one either obtains the ambiguity 
set proposed in \cite{parys_DataDecisionsDistributionally_2020} or 
the ambiguity set corresponding to the \emph{entropic value-at-risk} \cite{ahmadi-javid_EntropicValueatRiskNew_2012}.
]{Rewrite/remove: not very useful.}

Furthermore ambiguity sets defined as balls in the \ac{TV} metric are 
quite commonly used \cite{schuurmans_SafeLearningBasedControl_2019,jiang_RiskAverseTwoStageStochastic_2018,sun_ConvergenceAnalysisDistributionally_2016}.
\sidenote[
More generally, \cite{ben-tal_RobustSolutionsOptimization_2012,bayraksan_DataDrivenStochasticProgramming_2015} provide tractable formulations of linear programs under ambiguity, considering 
the broad class of $\phi$-divergences, which include the \ac{KL} divergence 
and \ac{TV} distance as special cases. However, the parameters controlling 
the size of the ambiguity sets are calibrated on data using asymptotic results.]{Rev. 4: unclear.}

\sidenote[In what follows, we will focus on the \ac{KL} and \ac{TV} ambiguity 
sets and show how they satisfy \Cref{assum:learner}.]{And others}
}{}

\subsection{Divergence-based ambiguity sets} \label{sec:divergence-based-ambiguity-sets}

\begin{table*}[t]
    \centering
    \caption{Overview of properties of common divergence-based ambiguity sets.}
    \label{tab:overview}
    \begin{tabular}{cccc}
        \toprule\\
      Divergence & $\mathcal{D}(\hat{p}, p)$  & radius $r(\nsample,\conf)$ & Conic representation\\
      \midrule
    Total variation (TV) 
        & $\nrm{p - \hat{p}}_1$ 
            &  $2 \sqrt{\rTV(\nsample, \conf)}$
                & Linear \\  
    Kullback-Leibler (KL)    
        &$\DKL(\hat{p},p)$
            & $\rKL(\nsample, \conf)$
                & Exponential    \\ 
    Jensen-Shannon  (JS)
        & $\nicefrac{1}{2}\left(  
            \DKL\left(
                \hat{p}, \tfrac{p+\hat{p}}{2}
                \right) 
            + \DKL\left( 
                \hat{p}, \tfrac{p+\hat{p}}{2}
            \right)
            \right)$
                & $\tfrac{1}{2} \rKL(\nsample, \conf)$  
                    & Exponential       
    \\  
    (Squared) Hellinger (H)
        & $\displaystyle \sum_{i\in \W} \left(
            \sqrt{p}_i - \sqrt{\hat{p}_i}
            \right)^2$ 
            & $\rKL(\nsample, \conf)$
                & Quadratic \\
    Wasserstein$^{\hyperlink{wasserstein}{\star}}$ (W)
        & $\displaystyle 
        \min_{\Pi \in \Re^{\nModes \times \nModes}_+} \Big\{
            \sum_{i,j\in \W} \Pi_{ij} \elem{\Wker}{i}{j} \mid \Pi \1_{\nModes} = p, \Pi^\top \1_{\nModes} = \hat{p} 
        \Big \}$ 
            & $\max_{i,j \in \W} \elem{\Wker}{i}{j} \sqrt{\rTV(m, \conf)}$
            & Linear \\
    \bottomrule\\
    \multicolumn{4}{l}{
        \footnotesize
        \smash{
            \raisebox{0.7\normalbaselineskip}{
                \hypertarget{wasserstein}{
                    \rerevision{}{\textit{$^\star$Assumes $\W$ is a metric space. $K \in \Re^{\nModes \times \nModes}$ is a symmetric distance kernel with $\elem{K}{i}{j} = \dist(i,j),\; \forall i,j \in \W$.}}
                }
            }
        }
    }
    \end{tabular}

  \end{table*}
Our goal is to obtain for each mode $\md$ of the Markov chain, a data-driven subset of the probability simplex, containing the $\md$th row of the transition kernel $\transmat$ with high probability. 
Given a sequence $\seq{\DR{w}}{1}{t} \in \W^{t}$ of $t \in \N$ samples 
drawn from the Markov chain, $\nModes$ individual datasets
$\hat{W}_{t,i} \dfn \{
    \DR{\md}_{k+1} \mid 
        \DR{\md}_{k} = i,
        k \in \natseq{1}{t}
    \}$,
$i \in \W$ 
can be obtained by partitioning the set of observed transitions by the mode they originated in. 
As such, each $\hat{W}_{t,i}$ contains $t_i$ i.i.d. draws from the distribution $\row{\transmat}{i}$. Ambiguity sets can now be constructed for each individual row $i$, using concentration inequalities based 
on the data in $\hat{W}_{t,i}$.\revision{See, for instance, \cite{wainwright_HighdimensionalStatisticsNonasymptotic_2019,boucheron_ConcentrationInequalitiesNonasymptotic_2013} for more details on related techniques.}{}

With this set-up, we now consider the following \revision{broad class of ambiguity sets. In the remainder of this section, we will for ease of notation consider a scalar confidence level $\conf_t \in (0,1]$.}{instance of a learning system}.
\revision{}{
\begin{definition}[Empirical learner]
  Let the learner state be composed as $\lrn_t = (\vect\hat{\transmat}_t,  \revision{\rAmb_t}{ \gamma_t}) \in \revision{}{\lrnset = } \simplex_{\nModes}^{\nModes} \times \revision{\Re_+^\nModes}{[0,1]^{\nModes}}$, where $\hat{\transmat}_t$ denotes the empirical transition probability matrix at time $t$, that is, 
  \[ 
    \elem{\hat{\transmat}_t}{i}{j} = 
    \begin{cases}
      \tfrac{1}{t_i} \sum_{w \in \idx{\hat{W}_t}{i}} \1_{w=j} & \text{if } t_i > 0 \\
      \nicefrac{1}{\nModes} & \text{Otherwise,}
    \end{cases}
  \]
  and $\gamma_{t} = (\tfrac{1}{t_i+1})_{i \in \W}$ is a vector containing the inverse of the mode-specific sample sizes.\footnote{
    The inversion results in simpler updates and renders $\lrnset$ robustly positive invariant, i.e, $\lrn_t \in \lrnset \implies \lrn_{t+1} \in \lrnset$.
  }
\end{definition}
}
For this instance of a learning system, we can now easily derive an explicit characterization of $\learner$.
\revision{}{
  \begin{example}[Dynamics of the empirical learner] \label{ex:learner-dynamics}
    The learner state $\lrn_t$ is composed of $\lrn_t = (\vect \hat{P}_t, \gamma_t)$.
    \newcommand{\idvec}[1]{\bm{e}_{#1}}
    For the update of the empirical distribution $\hat{P}_t$, note that if $w_{t} \neq i$, then trivially, $\row{\hat{P}_{t+1}}{i} =\row{\hat{P}_{t}}{i}$. Otherwise, 
    we may use the following well-known construction. Let $\idvec{w} \in \Re^{\nModes}$ denote the $w$'th standard basis vector in $\Re^d$, then 
    \begin{align*}
          \nonumber 
          \row{\hat{P}_{t+1}}{i} &= \tfrac{1}{t_i+1} \tsum_{w \in \hat{W}_{t+1,i}}  \idvec{w} \\ 
          &= 
          \nonumber 
          \tfrac{1}{t_i+1} \left( \tsum_{w \in \hat{W}_{t,i}} \idvec{w} + \idvec{w_{t+1}} \right) \\ 
          \nonumber 
          &= 
          \tfrac{1}{t_i+1} \left( t_i \row{\hat{P}_{t}}{i} + \idvec{w_{t+1}}
          \right) \\ 
          &= 
            (1 - \gamma_{t,i}) \row{\hat{P}_t}{i} + \gamma_{t,i} \idvec{w_{t+1}}.
  \end{align*}
  Thus, for all $i \in \W$, we define 
    \begin{equation}\label{eq:deriv-P-update}
      \learner_{1,i}(\hat{P}_t, \gamma_{t}, \md_t) \dfn \begin{cases}
        (1 - \gamma_{t,i}) \row{\hat{P}_t}{i} + \gamma_{t,i} \idvec{w_{t+1}} & \text{if } \md_{t} = i\\
        \row{\hat{P}}{t,i} & \text{otherwise.} 
      \end{cases}
    \end{equation}
  Similarly, $\gamma_{t+1,i} = \gamma_{t, i}$ if $\md_{t+1} \neq i$. Otherwise, it follows from the definition of $\gamma_{t,i}$ that 
    \(
    \gamma_{t+1,i} = \tfrac{1}{t_i + 2} 
    \implies  \gamma_{t+1,i} = \tfrac{\gamma_{t,i}}{1 + \gamma_{t,i}}, 
    \)
  resulting in 
  \begin{equation} \label{eq:deriv-update-gamma}
      \learner_{2,i}(\gamma_{t}, \md_{t+1}) \dfn \begin{cases}
        \tfrac{\gamma_{t,i}}{1 + \gamma_{t,i}} & \text{if }w_{t+1} = i\\
        \gamma_{t,i} & \text{otherwise.}  
      \end{cases}
  \end{equation}
  
  Concatenating \eqref{eq:deriv-P-update}--\eqref{eq:deriv-update-gamma}, 
  we obtain the Markovian update required by \Cref{assum:learner}
  \[ 
    \learner(\lrn_t, \md_t,\md_{t+1}) = \big( (\learner_{1,i}(\hat{P}_t, \gamma_t, \md_t))_{i \in \W}, (\learner_{2,i}(\gamma_{t}, \md_{t+1}))_{i \in \W} \big).
  \] 
  Furthermore, this system satisfies \Cref{assum:stationary}. Indeed, given ergodicity of the Markov chain (\Cref{assum:ergodicity}),
  the Borel-Cantelli lemma \cite[Thm. 4.3]{billingsley_ProbabilityMeasure_1995} in conjunction with \cite[Lem. 6]{wolfer_MinimaxLearningErgodic_2019} guarantees that with probability 1,
  there exists a finite time $T$, such that for all $t > T$ and for all $i \in \W$, it holds that $t_{i} \geq c t$,
  where $c > 0$ is a constant depending on specific properties of the Markov chain, and $t_i$ denotes the number of visits to mode $i$. That is, all modes are visited infinitely often, 
  and as a result, both 
  $\lim_{t \to \infty} \gamma_{t} = 0$ and $\lim_{t \to \infty} \DR{P}_t = \transmat$, which are indeed fixed points of \eqref{eq:deriv-P-update}--\eqref{eq:deriv-update-gamma}.
  \end{example}
}
\revision{}{
We can now associate with the newly defined empirical learner 
the following wide class of ambiguity sets, which take the form of 
a ball around the empirical estimate in some given statistical divergence. 
}
\begin{definition}[Divergence-based ambiguity set] \label{def:divergence-amb}
  Consider the empirical learner with state $\lrn_t = (\vect{\hat{\transmat}_t, \gamma_{t}})$.
  We say that an ambiguity set $\amb_{\conf_t}(\lrn_t, \md)$ is a divergence-based ambiguity set if it can be expressed in the form
  \[
    \amb_{\conf_t}(s_t, \md) 
    \dfn \{ 
      p \in \simplex_{\nModes}
      \mid 
      \D{} (\row{\hat{\transmat}_t}{\md}, p) 
      \leq 
      \revision{
        \idx{\rAmb_t}{\md}
      }{
        \rAmb(\idx{\gamma_t}{\md}^{-1} - 1, \beta_{t})}\}, \, \forall\md \in \W\]
  where $\D{}: \simplex_{\nModes} \times \simplex_{\nModes} \to \Re_+$
  is 
  some statistical divergence \revision{.}{ and $\rAmb: \Re_+ \times [0,1] \to \Re_+$ is a given function that returns a radius, given a sample size and a confidence level.}
\end{definition}

\revision{Statistically meaningful values for the radius $\rAmb_t$ under different choices of divergences can be obtained using the following 
standard results.}{
Statistically meaningful values for the radius $\rAmb$ under different choices of divergences can be obtained using the following 
standard results.
}

\begin{proposition}[Concentration inequalities] \label{thm:concentrations}
    Let $p \in \simplex_{\nModes}$ denote a distribution on the probability simplex and
    $\hat{p} = \tfrac{1}{\nsample} \sum_{t=0}^{\nsample-1} (\1_{w_{t}=i})_{i=1}^{\nModes}$ the empirical distribution based on $\nsample$ \iid{} draws $\md_t \sim p$. Then, 
    $\prob\big[(\tfrac{1}{2} \nrm{p - \hat{p}}_1)^2 > \rTV(\nsample, \conf)\big] \leq \conf$, with 
    \begin{equation} \label{eq:TV-radius}
      \rTV(\nsample, \conf) = 
        \frac{\nModes \log 2 - \log \conf}{2\nsample}.
    \end{equation}
    Similarly, it holds that
    \(
        \prob[ \DKL(\hat{p}, p) > \rKL(\nsample, \conf)] \leq \conf, 
    \) with 
    \begin{equation} \label{eq:KL-radius}
       \rKL(\nsample, \conf) = \frac{d \log \nsample - \log \conf}{\nsample},
    \end{equation}
    where $\DKL(p, q) \dfn \sum_{i=1}^{\nModes} p_i \log \tfrac{p_i}{q_i}$ 
    denotes the \ac{KL} divergence from $q$ to $p$.
  \end{proposition}
  The bound on the \ac{TV} distance \eqref{eq:TV-radius} is known as the Bretagnolle-Huber-Carol inequality \cite[Thm. A.6.6]{vaart_WeakConvergenceEmpirical_2000}.
  \begin{remark}\label{rem:pinsker}
    Expression \eqref{eq:KL-radius} for the \ac{KL} radius is a well-known result from the field of information theory, obtained through the so-called method-of-types \cite{csiszar_MethodTypes_1998,cover_ElementsInformationTheory_2006}. A slight improvement can be obtained by replacing $\nModes \log \nsample$ by $\log \binom{\nsample+\nModes-1}{\nModes - 1}$. Moreover, in \cite{mardia_ConcentrationInequalitiesEmpirical_2019}, an even sharper result for \eqref{eq:KL-radius} is derived. In fact, this improved concentration bound in the \ac{KL} divergence was used in the same work to improve upon the \ac{TV} concentration bound \eqref{eq:TV-radius} for $\frac{\nsample}{\nModes} \ll 1$, using Pinsker's inequality
    \cite{csiszar_InformationTheoryCoding_2011},
    which relates the \ac{TV} distance between distributions $p, q \in \simplex_{\nModes}$ to the \ac{KL} divergence as 
     \(
       \nrm{p-q}_{1}^{2} \leq 2 \DKL(p,q).
     \)   
    \revision{These improvements remain compatible with the framework but would complicate notations in subsequent analysis. For this reason,
    we will define the following divergence-based ambiguity sets using \Cref{thm:concentrations}.}{
    Of course, these improved bounds can be readily used in practice to replace those in \cref{thm:concentrations}.
    However, for the theoretical discussion, these modifications are inconsequential. For this reason, we opt to develop the ideas for the simpler, more commonly used forms.
    }
    \end{remark}

  \revision{}{
  Besides Pinsker's inequality, there exist several other inequalities relating different statistical divergences (see for instance
  \cite{gibbs_ChoosingBoundingProbability_2002}
  for a comprehensive overview). Based on these relations, one can derive from \cref{thm:concentrations} several divergence-based ambiguity sets defined through other 
  statistical divergences; 
  \rerevision{}{
  For instance, since the squared Hellinger divergence is upper bounded by the \ac{KL} divergence, \eqref{eq:KL-radius} can be used as 
  a radius for Hellinger divergence-based ambiguity sets.
  }
  A summary of the resulting radii is provided in \Cref{tab:overview}.
  The rightmost column in this table refers to the 
  conic representation of the induced ambiguity sets (cf. \eqref{eq:conic-representation}), which determines the complexity of the resulting optimal control problems (see Appendix~\ref{sec:conic-reps} for more details).
  \rerevision{}{Other works that have used these divergences (which belong to the class of $\phi$-\textit{divergences})
  for distributionally robust optimization are {\cite{bayraksan_DataDrivenStochasticProgramming_2015,ben-tal_RobustSolutionsOptimization_2012,yanikoglu_SafeApproximationsAmbiguous_2012}}. 
  In these works, however, the radii are either selected as a tuning parameter or calibrated using asymptotic arguments, 
  leading to approximate ambiguity sets, 
  which satisfy the coverage condition \eqref{eq:high-confidence} only as the sample size tends to infinity.
  By contrast, the radii given in \Cref{tab:overview} are valid for any sample size.}
  }

We conclude the section by proposing a useful extension of the learner 
state in the case of divergence-based ambiguity sets. 

\begin{remark}[Radius as part of the learner state] \label{rmrk:radius-state}
  For divergence-based ambiguity sets, it is often convenient to augment the learner state $\lrn_t = (\vect \hat{P}_t, \gamma_{t,i})$ with the computed radii $\rAmb_{t,i} \dfn \rAmb(\gamma_{t,i}^{-1} - 1, \beta_t)_{i \in \W}$, for which the recursive update is obtained simply by composition of $\rAmb$ with the previously designed $\learner$ and $\confdyn$.
  It can be easily verified using \Cref{thm:concentrations} that this quantity also converges to the fixed point $\lim_{t\to \infty} r_{t,i} = 0$, $\forall i$. Indeed, 
  Recall from \Cref{ex:learner-dynamics} that $\gamma_{t,i}^{-1} - 1 \sim t$ as $t \to \infty$. Using a radius function $r$ based on either $\eqref{eq:TV-radius}$ or $\eqref{eq:KL-radius}$, we obtain
  $$
    \lim_{t \to \infty} \rAmb_{t,i} = \lim_{t \to \infty}\rAmb(t, \beta_t) = \lim_{t \to \infty} \tfrac {-\log \beta_t}{t} = 0,
  $$
  where the last equality follows from \Cref{assum:confidences}.
\end{remark}\section{Learning model predictive control} \label{sec:data-driven-MPC}

Given a learning system satisfying \Cref{assum:learner}, we define the augmented state $y_t = (x_t, \lrn_t, \confb_t) \in \Y \dfn \Re^{\ns} \times \lrnset \times \I$, which evolves over time according to the dynamics
\begin{equation} \label{eq:augmented-dynamics}
    y_{t+1} = \fa(y_t, \md_{t}, u_t, \md_{t+1}) \dfn \smallmat{f(x_t, u_t, \md_{t+1})\\ \learner(\lrn_t, \confb_t, \md_t, \md_{t+1})\\\confdyn(\confb_t) },
\end{equation}
with $\md_{t+1} \sim \row{\transmat}{\md_t}$, 
for $t \in \N$. Furthermore, it will be convenient to define the process  $z_t = (y_t, \md_t) \in \augset \dfn \Y \times \W$.
Consequently, the objective is now to obtain a feedback law $\law: \augset \rightarrow \Re^{\na}$. To this end, 
we will formulate a \ac{DR} counterpart to the stochastic \ac{OCP} \eqref{eq:stochastic-OCP}, in which the expectation operator in the cost and the conditional probabilities in the constraint will be replaced by operators that account for ambiguity in the involved distributions. 

\subsection{Ambiguity and risk}
In order to reformulate the cost function \eqref{eq:cost-function-stochastic}, 
we first introduce an ambiguous conditional expectation operator, leading to a formulation akin to the Markovian risk measures utilized in \cite{sopasakis_RiskaverseModelPredictive_2019,ruszczynski_RiskaverseDynamicProgramming_2010b}. 
Consider a function $\rv: \augset \times \W \rightarrow \barre$, defining a stochastic process $(\rv_t)_{t \in \N} = (\rv(z_t, \md_{t+1}))_{t\in\N}$ on $(\Omega, \F, \prob)$, and suppose that 
the augmented state $z_t = z = (x, \lrn, \confb, \md)$ is given. Let $\conf \in [0,1]$ denote an arbitrary component of $\confb$. The \textit{ambiguous} conditional expectation of $\rv(z, \mdnxt)$, given $z$ is then
\begin{equation} \label{eq:dr-expectation}
    \begin{aligned}
       \lrisk{\lrn}{\md}{\conf}[\rv(z,\mdnxt)] &\dfn \max_{p \in \amb_{\conf}(\lrn, \md)} \E_{p}[\rv(z,\mdnxt) | z] \\
        &= \max_{p \in \amb_{\conf}(\lrn, \md)} \tsum_{\mdnxt \in \W} p_\mdnxt \rv(z,\mdnxt).
    \end{aligned}
\end{equation}
Trivially, it holds that if the $w$'th row of the transition matrix lies in the corresponding ambiguity set, i.e., $\row{\transmat}{\md} \in \amb_{\conf}(\lrn,\md)$, then
\ifJournal
\begin{equation}\label{eq:risk-upper-bound}
    \begin{aligned}
    \lrisk{\lrn}{\md}{\conf}[\rv(z,\mdnxt)] &\geq \E_{\row{\transmat}{\md}} [\rv(z,\mdnxt) \mid z] \\
    &= \textstyle \sum_{\mdnxt \in \W} \elem{\transmat}{\md}{\mdnxt} \rv(z,\mdnxt).
    \end{aligned}
\end{equation}
\else 
\begin{equation*} 
    \begin{aligned}
    \lrisk{\lrn}{\md}{\confb}[\rv(z,\mdnxt)] &\geq \E_{\row{\transmat}{\md}} [\rv(z,\mdnxt)|z] 
    = \textstyle \sum_{\mdnxt \in \W} \elem{\transmat}{\md}{\mdnxt} \rv(z,\mdnxt).
    \end{aligned}
\end{equation*}
\fi
Note that the function $\lrisk{\lrn}{\md}{\conf}$ defines a coherent risk measure \cite[Sec. 6.3]{shapiro2009lectures}. We say that $\lrisk{\lrn}{\md}{\conf}$ is the risk measure \emph{induced by} the ambiguity set $\amb_{\conf}(\lrn,\md)$. 

A similar construction can be carried out for the chance constraints \eqref{eq:chance-constraint}. 
We robustify the average value-at-risk with respect to the reference distribution, defining  
\begin{equation} \label{eq:ambiguous-chance-constraint}
    \lriskc{\lrn}{\md}{\conf}{\DR{\alpha}}[\rv(z,\mdnxt)] \dfn \hspace{-6pt} \max_{p \in \amb_{\conf}(\lrn,\md)} \hspace{-6pt} \AVAR^p_{\DR{\alpha}}[\rv(z,\mdnxt) \mid z] \leq 0.
\end{equation}
The function $\lriskc{\lrn}{\md}{\conf}{\DR{\alpha}}$ in turn defines a coherent risk measure.
Note that we have replaced the $\AVAR$ parameter $\alpha$ 
by $\DR{\alpha}$. The reason for this is that the ambiguity set 
only contains the true distribution with high probability. 
Considering this fact, it is natural to expect that $\alpha$ needs 
to be tightened to some extent in order to ensure that the original 
chance constraint remains satisfied. We make this precise in the following result. 

\begin{proposition} \label{prop:conditions-confidences}
    Let $\conf, \alpha \in [0,1]$, be given values with $\conf < \alpha$.
    Consider the random variable $\lrn: \Omega \rightarrow \lrnset$, denoting an (a priori unknown) learner state satisfying $\Cref{assum:learner}$,
    i.e., $\prob[\row{\transmat}{\md} \in \amb_{\conf}(s,\md)] \geq 1 - \conf$.
    If the parameter $\DR{\alpha}$ is chosen to satisfy 
    \(
        0 \leq \hat{\alpha} \leq \frac{\alpha - \conf}{1 - \conf} \leq 1, 
    \)
    then, for an arbitrary function $g : \augset \times \W \rightarrow \Re$, the following implication holds:
    \begin{equation} \label{eq:implication-prop3-1}
        \begin{aligned}
        \lriskc{\lrn}{\md}{\conf}{\hat{\alpha}}[g(z,\mdnxt)] \leq 0, \, \as
        \Rightarrow
        \prob[g(z,\mdnxt) \leq 0 \mid x,\md] \geq 1- \alpha.
        \end{aligned}
    \end{equation}
\end{proposition}

\begin{proof}

    If \(
        \lriskc{\lrn}{\md}{\conf}{\hat{\alpha}}[g(z,\mdnxt) ] \leq 0 
    \), \as, then \eqref{eq:risk-constraint-implication} and \eqref{eq:ambiguous-chance-constraint} imply that 
    \[
        \prob[g(z,\mdnxt) \leq 0 \mid x,\md,\row{\transmat}{\md} \in \amb_{\conf}(\lrn, \md)] \geq 1 - \DR{\alpha}, \as  
    \] Therefore,
    \[
        \begin{aligned}
            \ifSingleColumn\else&\fi
            \prob[g(z,\mdnxt) \leq 0 \mid x,\md] \ifSingleColumn\else\\\fi
            &\geq 
            \prob[g(z,\mdnxt) \leq 0 \mid x,\md,\row{\transmat}{\md} \in \amb_{\conf}(\lrn, \md)]
             \prob[\row{\transmat}{\md} \in \amb_{\conf}(\lrn, \md)] \\ 
            &\geq (1-\DR{\alpha})(1-\conf).
        \end{aligned}
    \]
    Requiring that $(1-\DR{\alpha})(1-\conf) \geq (1 - \alpha)$ then immediately yields the sought condition.
\end{proof}
    Notice that the implication \eqref{eq:implication-prop3-1} in \Cref{prop:conditions-confidences} provides an \textit{a priori} guarantee,
    since the learner state is considered to be random. 
    In other words, the statement is made before the data is revealed. 
    Indeed, for a \emph{given} learner state $\lrn$ and mode $\md$, the ambiguity set $\amb_{\conf}(\lrn,\md)$ is fixed and therefore, the outcome of the event $E = \{\row{\transmat}{\md} \in \amb_{\conf}(\lrn,\md)\}$ is determined. Whether \eqref{eq:implication-prop3-1} then holds for these fixed values, depends on the outcome of $E$.
    This is naturally reflected through the above condition on $\DR{\alpha}$, which implies that $\DR{\alpha} \leq \alpha$, and thus tightens the chance constraints that are imposed conditioned on a \emph{fixed} $\lrn$. Hence, the possibility that for this particular $\lrn$, the ambiguity set may not include the conditional distribution, is accounted for. This tightening can be mitigated by decreasing $\conf$, at the cost of a larger ambiguity set. A more detailed study of this trade-off is left for future work.

\subsection{Distributionally robust model predictive control}

We are now ready to describe the \ac{DR} counterpart to the OCP \eqref{eq:stochastic-OCP}, which, when solved in receding horizon fashion, yields the proposed \revision{data-driven}{learning} MPC scheme. 

Consider a given augmented state $z = (x, \lrn, \confb, \md) \in \augset$. 
Hereafter, we will assume that $\confb = (\conf, \bar \conf)$, where 
component $\conf$ is related to the cost function and $\bar \conf$ 
is reserved for the constraints. 

We use \eqref{eq:ambiguous-chance-constraint} to define the \ac{DR} set of feasible inputs $\DR{\Ufeas}(z)$ in correspondence to \eqref{eq:chance-constraint}, as 
\begin{equation} \label{eq:DR-constraints}
 \DR{\Ufeas}(z) {=} \left \{ u \in U \sep 
    \lriskc{\lrn}{\md}{\bar \conf}{\DR{\alpha}}[g (x,u,\md,\mdnxt)] \leq 0 \right \}.
\end{equation}
\begin{remark}
The parameter $\DR{\alpha}$ remains to be chosen in relation to the confidence levels $\confb$ and the original violation rates $\alpha$. In light of \Cref{prop:conditions-confidences}, 
$\DR{\alpha}= \tfrac{\alpha - \bar \conf}{1- \bar \conf}$ yields the least conservative choice. This choice is valid as long as it is ensured that $\bar \conf < \alpha$.
\end{remark}

Using \eqref{eq:dr-expectation}, we express the \ac{DR} cost of a policy $\pol = (\pol_{k})_{k=0}^{\hor-1}$ as
\begin{multline} \label{eq:cost-function-risk}
    \DR{\cost}_\hor^{\pol}(z) \dfn \ell(x_0, u_0, \md_0) +
    \lrisk{\lrn_0}{\md_0}{\conf_{0}}
    \big[
        \ell(x_1, u_1,\md_1)\ifSingleColumn\else\\\fi +
         \lrisk{\lrn_1}{\md_1}{\conf_{1}}
        \big[
            \dots + \lrisk{\lrn_{\hor-2}}{\md_{\hor-2}}{\conf_{\hor-2}}\big[
             \ell(x_{\hor-1},u_{\hor-1}, \md_{\hor-1})\\
             + \lrisk{\lrn_{\hor-1}}{\md_{\hor-1}}{\conf_{\hor-1}}[ \DR{\Vf}(x_\hor, \lrn_\hor, \confb_{\hor}, \md_\hor) ] \big]
            \dots 
        \big]
    \big],
\end{multline}
where $z_0=z$, $z_{k+1} =\fa(z_k, u_k, \md_{k+1})$ and $u_{k} = \pol_{k}(z_{k})$, for all $k \in \natseq{0}{\hor-1}$. 
In \Cref{sec:theory}, conditions on the terminal cost $\DR{\Vf}: \augset \rightarrow \barre_+: (x, \lrn, \confb, \md) \mapsto \Vf(x,\md) + \delta_{\DR{\Xf}}(x,\lrn, \confb, \md)$ and its domain are provided in order to guarantee recursive feasibility and stability of the \ac{MPC} scheme defined by the following \ac{OCP}.

\begin{definition}[\acs{\RAOCP}] \label{def:raocp}
Given an augmented state $z \in \augset$, the optimal cost of the \ac{\RAOCP} is 
\begin{subequations} \label{eq:risk-averse-OCP}
\begin{align}
    \DR{\cost}_{\hor}(z) = \min_{\pol} \DR{\cost}_{\hor}^{\pol}(z)
\end{align}
subject to 
\begin{align} \label{eq:risk-averse-constraints}
    (x_0, \lrn_0, \confb_0, \md_0) &= z,\, \pol = (\pol_{k})_{k=0}^{\hor-1}, \\
    z_{k+1} &= (\fa(z_k, \pol_{k}(z_k), \md_{k+1}), \md_{k+1}),\\
    \pol_{k}(z_k) &\in \DR{\Ufeas}(z_k), \label{eq:risk-constraints-OCP}
    \; \forall \seq{\md}{0}{k} \in \W^{k},
\end{align}
for all $k \in \natseq{0}{\hor-1}$.
\end{subequations}
We denote by $\DR{\Pi}_{\hor}(z)$ the corresponding set of minimizers.
\end{definition}
\begin{remark}
    Note that the definition of $\DR{\Vf}$ implicitly imposes the terminal constraint $z_\hor \in \DR{\Xf}$, \as. 
\end{remark}

We now define the \revision{data-driven}{learning} \ac{MPC} law analogously to the stochastic case as 
\begin{equation} \label{eq:MPC-law-DR}
    \DR{\law}_{\hor}(z) = \DR{\pol}_{0}^{\star}(z),
\end{equation}
where $(\DR{\pol}_{k}^{\star}(z))_{k=0}^{\hor-1} \in \DR{\Pi}_{\hor}(z)$.
At every time $t$, the \revision{data-driven}{learning} MPC scheme thus consists of repeatedly
\begin{inlinelist*}
    \item solving \eqref{eq:risk-averse-OCP} to obtain a control action $u_t = \DR{\law}_\hor(z_t)$ and applying it to the system \eqref{eq:system-dynamics}
    \item observing the outcome of $\md_{t+1} \in \W$ and the corresponding next state $x_{t+1}=f(x_t,u_t,\md_{t+1})$
    \item updating the learner state $\lrn_{t+1} = \learner(\lrn_t, \md_t,\md_{t+1})$ and the confidence levels $\confb_{t+1} = \confdyn(\confb_t)$, gradually decreasing the size of the ambiguity sets.
\end{inlinelist*}

\rerevision{
\revision{
\section{Tractable reformulation}
\subsection{Conic risk measures} 
}{
\subsection{Conic reformulations}\label{sec:conic}
}
Since ambiguity sets inducing coherent risk measures are convex by construction, 
many classes of ambiguity sets can be represented 
using conic inequalities.
}{
Note that in its general form, \eqref{eq:risk-averse-OCP} is a non-smooth,
infinite-dimensional optimization problem.
However, provided that the involved risk measures are
\textit{conic risk measures} (as defined by \Cref{def:conic-risk-measure}), 
problem \eqref{eq:risk-averse-OCP} can be reformulated as 
a finite-dimensional, smooth nonlinear program.
}
\begin{definition}[Conic risk measure \cite{sopasakis_risk-averse_2019c}] \label{def:conic-risk-measure}
    We say that an ambiguity set $\amb \subseteq \simplex_{\nModes}$ is conic representable if it can be written in the form
    \begin{equation}\label{eq:conic-representation}
        \amb = \{ p \in \simplex_{\nModes} \mid \exists \nu: Ep + F \nu \leqc{\cone} b\}, 
    \end{equation}
    with matrices $E, F$ and vector $b$ of suitable dimensions, 
    and a proper cone $\cone$. The coherent risk measure induced by a conic representable ambiguity set is called a conic risk measure.
\end{definition}
\rerevision{
These risk measures, referred to as conic risk measures, are of great use for reformulating \acp{\RAOCP} of the form \eqref{eq:risk-averse-OCP}.
}{
Since ambiguity sets inducing coherent risk measures are convex by construction, 
many classes of ambiguity sets can be represented using conic inequalities.
For completeness, we state the conic representations for the ambiguity sets summarized in \Cref{tab:overview}, 
as well as the reformulation of \eqref{eq:risk-averse-OCP}
in Appendix~\ref{sec:conic-reps} and \ref{sec:tractability}, 
respectively.
}
\rerevision{
By definition, a conic risk measure $\rho$ is given as the optimal value of a standard \ac{CP}. Under strong duality, which holds if the \ac{CP} is strictly feasible \cite[Prop. 2.1]{shapiro_DualityTheoryConic_2001}, its epigraph $\epi \rho \dfn \{ (G,\gamma) \in \Re^{\nModes + 1} \mid \gamma \geq \rho[G]\}$ can be characterized as \cite{sopasakis_risk-averse_2019c}
\begin{equation} \label{eq:risk-epigraph-old}
    \ifSingleColumn
    \epi \rho = \{(G,\gamma) \in \Re^{\nModes + 1} \mid \exists y: \trans{E}y=G, \trans{F}y=0, y \geqc{\cone^*} 0, \gamma \geq \trans{b}y \}.\else 
    \epi \rho = \left\{
        (G,\gamma) \in \Re^{\nModes + 1} \sep \begin{matrix}\exists y: \trans{E}y=G, \trans{F}y=0,\\
         y \in \cone^*, \gamma \geq \trans{b}y \end{matrix} 
        \right\}.
    \fi
\end{equation}
\revision{Since the \ac{TV} ambiguity set defined in \Cref{sec:learning} as well as the ambiguity set inducing the average value-at-risk are polyhedra, they are conic representable (taking the nonnegative orthant as the cone $\cone$).
Similarly, the \ac{KL} ambiguity set, and similarly the entropic value-at-risk \cite{ahmadi-javid_EntropicValueatRiskNew_2012} are known 
to be conic representable \cite{sopasakis_risk-averse_2019c,parys_DataDecisionsDistributionally_2020}.}{
All divergence-based ambiguity sets collected in \Cref{tab:overview} admit a conic formulation (see Appendix~\ref{sec:conic-reps}). 
}
Additionally, it is not difficult to show that the worst-case average value-at-risk \eqref{eq:ambiguous-chance-constraint} over a conic representable ambiguity set also 
defines a conic risk measure:
\begin{proposition} \label{prop:conic-robust-avar-old}
    Let $\amb = \{ p \in \simplex_{\nModes} \mid \exists \nu: \bar E p + \bar F \nu \leqc{\cone} \bar b \}$ be a conic-representable ambiguity set. Then, the risk measure
    $\bar \rho = \max_{p \in \amb} \AVAR_{\alpha}^{p}$ is a conic risk measure.
\end{proposition}
\begin{proof}
    For any reference distribution $p\in \simplex_\nModes$,
    the ambiguity set $\amb_{\AVAR}$ inducing $\AVAR_{\alpha}^{p}$ can be written in the form 
    \eqref{eq:conic-representation} with 
    $E=\trans{\smallmat{\1_{\nModes}&-\1_{\nModes}&\alpha I&-I}}$,
    $F=0$, $\cone=\Re^{2(\nModes + 1)}_{+}$ the nonnegative orthant,
    and
    $b = \trans{\smallmat{1&-1&\trans{p}&0}}$ 
    (which is of the form $b=b' + Bp$) \cite{sopasakis_risk-averse_2019c}.
    Writing out the definition of $\max_{p \in \amb} \AVAR_{\alpha}^{p}$ and rearranging terms yields 
    \begin{align*}
      &\max_{p \in \amb} \AVAR_{\alpha}^{p}[z] \\
      =& \max_{\mu} \left\{ \trans{\mu}z \sep \exists \nu: \smallmat{E\\0}\mu + \smallmat{-B & 0\\\bar E & \bar F} \nu \leqc{\Re^{2(\nModes + 1)}_{+} \times \cone}
      \smallmat{b' \\ \bar b} \right\},
    \end{align*}
    which is exactly of the form \eqref{eq:conic-representation}.
\end{proof}

Thus, if for all $(\lrn, \md, \conf) \in \lrnset \times \W \times [0,1]$, $\amb_{\conf}(\lrn,\md)$ is conic representable, then $\lrisk{\lrn}{\md}{\conf}$ and $\lriskc{\lrn}{\md}{\conf}{\alpha}$ are conic risk measures.

\revision{
This fact will allow us to leverage \eqref{eq:risk-epigraph} to 
obtain an efficiently solvable reformulation of \eqref{eq:risk-averse-OCP}.
We refer to Appendix~\ref{sec:tractability} for more details on the 
final formulation of the optimal control problem.
}{}
}{}\section{Theoretical analysis} \label{sec:theory}
\subsection{Dynamic programming}
    To facilitate theoretical analysis of the proposed \ac{MPC} scheme, we
    follow an approach similar to \cite{sopasakis_RiskaverseModelPredictive_2019} and represent \eqref{eq:risk-averse-OCP} as a dynamic programming recursion.
    We define the Bellman operator $\T$ as 
    \(
        \T(\hat{\cost})(z) \dfn \hspace{-1pt} \min_{u \in \DR{\Ufeas}(z)} \hspace{-3pt} \ell(x,u,\md) + \lrisk{\lrn}{\md}{\conf} [\hat{\cost}(\fa(z,u,\mdnxt),\mdnxt)],
    \)
    where $z = (x, \lrn, \confb, \md) \in \augset$, with $\confb = (\conf, \bar{\conf})$ as before, are fixed quantities and $\mdnxt \sim \row{\transmat}{\md}$.
    We denote by $\DParg (\DR{\cost})(z)$ the corresponding set of minimizers.
    The optimal cost $\hat{V}_N$ of \eqref{eq:risk-averse-OCP} is obtained through the iteration,
    \begin{equation} \label{eq:definition-DP}
        \begin{aligned}
        \hat{\cost}_k &= \T \hat{\cost}_{k-1}, \; \hat{\cost}_0 = \DR{\Vf}, \;k \in \natseq{1}{\hor}.
        \end{aligned}
    \end{equation}
    Similarly, $\DR{\augset}_{k} \dfn \dom \hat{\cost}_{k}$ is given recursively by 
    \begin{equation*}
        \DR{\augset}_k = \left\{z \, \middle|\, \exists u \in \DR{\Ufeas}(z) : 
            (\fa(z,u,\mdnxt), \mdnxt) \in \DR{\augset}_{k-1}, \,
            \forall \mdnxt \in \W \right\}.
    \end{equation*}

Now consider the stochastic closed-loop system
\begin{equation} \label{eq:closed-loop}
    \begin{aligned} 
    y_{t+1} = \fc(z_t, \md_{t+1}) \dfn \fa(z_t, \DR{\law}_\hor(z_t), \md_{t+1}),
    \end{aligned}
\end{equation}
where $\DR{\law}_{\hor}(z_t) \in \DParg(\hat{\cost}_{N-1})(z_t)$ is an optimal control law obtained by solving the \rerevision{\revision{data-driven}{learning-based}}{} \ac{\RAOCP} of horizon $\hor$ in receding horizon.

\subsection{Constraint satisfaction and recursive feasibility}

In order to show existence of $\DR{\law}_\hor \in \DParg \DR{\cost}_{\hor-1}$ at every time step, \Cref{prop:recursive-feasibility} will
require that $\DR{\Xf}$ is a robust control invariant set. We define robust control invariance for the augmented control system under consideration as follows.

\begin{definition}[Robust control invariance] \label{def:rob-inv}
    A set $\rinv \subseteq \augset$ is \iac{RCI} set for the system \eqref{eq:augmented-dynamics} if
    for all $z \in \rinv$, $\exists u \in \DR{\Ufeas}(z)$ such that $(\fa(z,u,\mdnxt), \mdnxt)  \in \rinv, \forall \mdnxt \in \W$.
    Similarly, $\rinv$ is \iac{RPI} set for the closed-loop system \eqref{eq:closed-loop} if for all $z \in \rinv$, $(\fc(z,\mdnxt), \mdnxt) \in \rinv,\, \forall \mdnxt \in \W$.
\end{definition}
Since $\DR{\Ufeas}$ consists of conditional risk constraints, our definition of robust invariance provides a distributionally robust counterpart to the notion of \emph{stochastic} robust invariance in \cite{korda_StronglyFeasibleStochastic_2011}. This notion is less 
conservative than the following, more classical notation of robust invariance. 

\begin{definition}[Classical robust control invariance] \label{def:classical-RCI} 
    A set $\rinv_x \subseteq \Re^{\ns} \times \W$ is \ac{RCI} for 
    system \eqref{eq:system-dynamics} in the classical sense if 
    for all $x \in \rinv_x$,
    \begin{equation} \label{eq:example-xf}
        \exists u: g(x,u,\md, \mdnxt) \leq 0,\, f(x,u,\mdnxt) \in \Xf(\mdnxt) , \;\forall \mdnxt \in \W. 
    \end{equation}
\end{definition}

In fact, for any set $\rinv_{x}$ as in \Cref{def:classical-RCI}, the set $\rinv_{x} \times \lrnset \times \I \times \W$ is covered by \Cref{def:rob-inv}, as illustrated in \Cref{ex:traditional-robust}. 
On the other hand, our notion of robust control invariance is stricter than 
that of \emph{uniform control invariance} considered in \cite{sopasakis_RiskaverseModelPredictive_2019}, which only requires successor states to remain in the invariant set for modes $\mdnxt$ in the \emph{cover}
of the given mode $\md$, i.e., the set of modes $\mdnxt$ for which $\elem{\transmat}{\md}{\mdnxt}>0$. This flexibility is not available in the 
current setting, as the transition kernel is assumed
to be unknown, so the cover of a mode cannot be determined with certainty.

\begin{example}[Classical robust invariant set] \label{ex:traditional-robust}
    Suppose that the terminal constraint set $\Xf$ of the 
    nominal problem is a robust control invariant set in the 
    classical sense and define for convenience $\Xf(\md) \dfn \{x \mid (x,w) \in \Xf \}$.
    Then, if $\DR{\Xf}$ is chosen such that $\DR{\Xf}(\md) \dfn \{y \mid (y, \md) \in \DR{\Xf} \} = \Xf(\md) \times \lrnset \times \I$,  $\DR{\Xf}$ is \ac{RCI} for the augmented system \eqref{eq:augmented-dynamics} according to \Cref{def:rob-inv}. Indeed, 
    since $\AVAR_{\alpha}^{p}[g(x,u,\md, \mdnxt)] \leq \max_{\mdnxt} g(x, u, \md, \mdnxt)$ for all $\alpha \in [0,1]$ and $p \in \simplex_\nModes$, \eqref{eq:example-xf} implies that for all $z \in \DR{\Xf}$, there exists $u \in \DR{\Ufeas}(z)$, such that $\fa(z,u,\mdnxt)\in \DR{\Xf}(\mdnxt)$.
\end{example}

\begin{proposition}[Recursive feasibility] \label{prop:recursive-feasibility}
    If $\DR{\Xf}$ is \iac{RCI} set for $\eqref{eq:augmented-dynamics}$, then \eqref{eq:risk-averse-OCP} is recursively feasible. That is,     
    feasibility of \ac{\RAOCP} \eqref{eq:risk-averse-OCP} for some $z \in \augset$, implies feasibility for $z^+ = (\fc(z,\mdnxt), \mdnxt)$, for all $\mdnxt \in \W, \hor \in \N_{>0}$.
\end{proposition}
\newcommand\recfeasproof{
    \begin{proof}
        The proof follows from a straightforward inductive argument on 
        the prediction horizon $\hor$. 
        We first show that if $\DR{\Xf}$ is \ac{RCI}, then so is $\DR{\augset}_\hor$.
        This is done by induction on the horizon $\hor$ of the \ac{OCP}. 
        
        \textbf{Base case ($\hor = 0$).} Trivial, since $\DR{\augset}_0 = \DR{\Xf}$.
        
        \textbf{Induction step ($\hor \Rightarrow \hor +1$).} Suppose that for some $\hor \in \N$, $\DR{\augset}_{\hor}$ is \ac{RCI} for \eqref{eq:augmented-dynamics}. Then, by definition of $\DR{\augset}_{\hor+1}$, there exists for each $z \in \DR{\augset}_{\hor+1}$, a nonempty set $\DR{\Ufeas}_{\hor}^{\sstar}(z) \subseteq \DR{\Ufeas}(z)$ such that for every $u \in \DR{\Ufeas}_{\hor}^{{\sstar}}(z)$ and for all $\mdnxt\in\W$, it holds that $z^+ \in \DR{\augset}_{\hor}$, where $z^+ = \fa(z,u,\mdnxt)$. Furthermore, the induction hypothesis ($\DR{\augset}_{\hor}$ is RCI), implies that there also exists a $u^+ \in \DR{\Ufeas}(z^+)$ such that $\fa(z^+, u^+, \mdnxt^+) \in \DR{\augset}_{\hor}(\mdnxt^+), \forall \mdnxt^+ \in \W$. Therefore, $z^+$ satisfies the conditions defining $\DR{\augset}_{\hor+1}$. In other words, $\DR{\augset}_{\hor+1}$ is RCI. 
        
        The claim follows from the fact that for any $\hor > 0$ and $z \in \DR{\augset}_{\hor}$, $u = \DR{\law}_{\hor}(z) \in \DParg(\DR{\cost}_{\hor-1})(z) \subseteq \DR{\Ufeas}^{{\sstar}}_{\hor-1}(z)$, as any other choice of $u$ would yield infinite cost in the definition of the Bellman operator.  
    \end{proof}
}
\ifArxiv
    \recfeasproof
\else
    \ifJournal
        \recfeasproof    
    \else
        \Cref{prop:recursive-feasibility} follows from a standard inductive 
        argument. We refer to \cite{schuurmans_LearningBasedDistributionallyRobust_2020a} for the detailed proof.
    \fi
\fi

    \begin{cor}[Chance constraint satisfaction] \label{cor:constraint-satisfaction}
        If the conditions for \Cref{prop:recursive-feasibility} hold, then by \Cref{prop:conditions-confidences}, the stochastic process $(z_t)_{t \in \N} = (x_t, \lrn_t, \confb_t, \md_t)_{t\in\N}$ satisfying dynamics \eqref{eq:closed-loop} satisfies the nominal chance constraints 
        \[ \prob[g(x_t, \DR{\law}_{\hor}(z_t), \md_{t+1})>0 \mid x_t, \md_t] < \alpha, \]
        \as, for all $t \in \N$.
    \end{cor}
    We conclude this section by emphasizing that although 
    the MPC scheme guarantees closed-loop constraint satisfaction, it does so while being less conservative than a fully robust approach,
    which is recovered by taking 
    $\amb_{\conf}(s,\md) = \simplex_{\nModes}$ for all $(s,\md,\conf) \in \lrnset \times \W \times [0,1]$. It is apparent from \Cref{eq:ambiguous-chance-constraint,eq:DR-constraints}, that for all other choices of the ambiguity set, the set of feasible control actions will be larger (in the sense of set inclusion).

\subsection{Stability} \label{sec:stability}
In this section, we will provide sufficient conditions on the control setup under which the origin is \ac{MSS} for \eqref{eq:closed-loop}, i.e., $\lim_{t \to \infty} \E[\nrm{x_{t}}^2] = 0$ for all $x_0$ in some specified compact set containing the origin.

Our main stability result, stated in \Cref{thm:MPC-stability},
hinges in large on the following \namecref{lem:DR-MSS}, which relates \emph{risk-square stability} {\cite[{\revision{Lem. 5}{Thm. 6}}]{sopasakis_RiskaverseModelPredictive_2019}} of the origin for the autonomous system \eqref{eq:closed-loop} (with respect to a statistically determined ambiguity set) to stability in the mean-square sense (with respect to the true distribution).

\begin{lem}[Distributionally robust \ac{MSS} condition] \label{lem:DR-MSS}
    Suppose that \Cref{assum:confidences} holds and that there exists a nonnegative, proper function $\cost: \augset \rightarrow \barre_{+}$, such that 
    \begin{conditions*}
        \item \label{cond:RPI} $\dom \cost$ is \revision{compact}{} \ac{RPI} for \eqref{eq:closed-loop} 
        \revision{}{and $\dom \cost(\argdot, \lrn, \confb, \md)$ is compact and contains the origin for all $(\lrn, \confb, \md): \dom \cost(\argdot, \lrn, \confb, \md) \neq \emptyset$}
        \item \label{cond:lyap-decrease} $\lrisk{\lrn}{\md}{\conf}[\cost(\fc(z,\mdnxt), \mdnxt)]-\cost(z) \leq - c \nrm{x}^2$, for some $c>0$,
        for all $z \in \dom \cost$;
        \item \label{cond:boundedness} $\cost$ is uniformly bounded on its domain.
    \end{conditions*}
    Then, $\lim_{t \to \infty}\E[ \nrm{x_t}^2] = 0$ 
    for all $z_0 \in \dom \cost$, where $(z_t)_{t\in\N} = (x_t, \lrn_t, \confb_t, \md_t)_{t\in\N}$ is the stochastic process governed by dynamics \eqref{eq:closed-loop}.   
\end{lem}
\begin{proof}     
    See \Cref{proof:lem:DR-MSS}.
\end{proof}

\begin{thm}[\acs{MPC} stability] \label{thm:MPC-stability}
    Suppose that \Cref{assum:confidences,assum:regularity} are satisfied and the following statements hold.
    \begin{conditions*}
        \item \label{cond:TVfleqVf} $\T \DR{\Vf} \leq \DR{\Vf}$;
        \item \label{cond:stage-cost-bound}
        $c \nrm{x}^2 \leq \ell(x,u,\md)$
              for some $c > 0$,
              for all
              $z = (x,\lrn,\confb,\md) \in \dom \DR{\cost}_\hor$ and all 
              $u \in \DR{\Ufeas}(z)$;
        \item \label{cond:locally-bounded} $\DR{\cost}_{\hor}$ is locally bounded on its domain.
    \end{conditions*}
    Then, the origin is \ac{MSS} for the \ac{MPC}-controlled system \eqref{eq:closed-loop}, over all \revision{compact}{} $\ac{RPI}$ sets $\bar{\augset} \subseteq \dom \DR{\cost}_{\hor}$ \revision{}{
        such that for all $(\lrn, \confb, \md): (x,\lrn,\confb,\md) \in \bar{\augset}$,
        the projection
        $\{x \mid (x, \lrn,\confb, \md) \in \bar{\augset} \}$ 
        is compact and contains the origin.
    }
\end{thm}
\begin{proof}
    The proof is along the lines of that of \cite[thm. 6]{sopasakis_RiskaverseModelPredictive_2019} and shows that $\hat{\cost}_\hor$ satisfies the conditions of \Cref{lem:DR-MSS}. Details are in the \Cref{proof:thm:MPC-stability}.
\end{proof}

\ifJournal
The results in this section indicate that after an appropriate choice 
of the learning system, the thusly defined risk measures 
can be used to design \iac{MPC} controller using existing techniques (e.g., those presented in \cite{sopasakis_RiskaverseModelPredictive_2019}).
Corresponding stability guarantees (assuming known transition probabilities) then translate directly into stability guarantees under an ambiguously estimated transition kernel.

\subsection{Out-of-sample bounds and consistency} \label{sec:consistency}
\rerevision{}{
    We now turn our attention to analyzing the value function of the 
    \ac{\RAOCP} in relation to the nominal (stochastic) \ac{OCP}. 
    We will show that under quite general assumptions, the former 
    provides an upper bound to the latter with high probability (\Cref{thm:out-of-sample}). 
    Furthermore, 
}
under appropriate constraint qualifications, we will show that the
optimal value of the \ac{\RAOCP} converges to that of the nominal problem as the sample size increases, see \Cref{thm:consistency}.
In the particular case where the constraints do not depend on the 
distribution, we can relax the constraint qualification to obtain 
a similar result. We include this as a separate statement, as it 
permits a more direct and illustrative proof using dynamic programming.

Given an arbitrary state-mode pair $(x,w)$, initial
value of the learning state $\lrn_0$ and confidence $\confb_0$, 
the stochastic process defined by the optimal value of the \ac{\RAOCP} \eqref{eq:risk-averse-OCP},
i.e., $\Vht{t}(x,\md) \dfn \DR{V}_{\hor}(x, \lrn_t, \confb_t, \md)$,
$t \in \N$
serves as a sequential approximation of the optimal value $\cost_{\hor}(x,\md)$ of the horizon-$\hor$ nominal \ac{OCP} \eqref{eq:stochastic-OCP}.
This section will establish sufficient conditions under which \rerevision{}{$\Vht{t}$ bounds $\cost_\hor$ from above, and}
for which it converges to $\Vtrue$ almost surely---a property which we refer to as \textit{asymptotic consistency}.
\rerevision{}{
The former guarantee will provide a performance certificate in the sense that
the true optimal cost
(under full knowledge of the distribution) will be no worse than the cost predicted 
by solving its \ac{DR} counterpart.
Of course, this guarantee is also provided by a robust (minimax) scheme 
(obtained by taking $\amb_{\conf_t} \equiv \simplex_{\nModes}, \forall t$).
However, such an approach is non-adaptive and therefore lacks \textit{consistency}. 
On the other hand, a sample-average approximation 
(in which the ambiguity set is replaced by a singleton containing only the empirical distribution) 
may under similar conditions be consistent, but it provides no safety guarantees/performance bounds. 
}
\rerevision{
To this end, we make the following assumption on the learner state and the corresponding ambiguity set. 
\begin{assumption}[Ambiguity decrease] 
    There exists a sequence $\{\ambdia_t\}_{t \in \N}$ with $\lim_{t \to \infty} \ambdia_t = 0$, such that
    \[
        \sup_{p, q \in \amb_{\idx{\confb_{t}}{i}}(s_t, \md)} \nrm{p-q}_1 \leq \ambdia_{t}\quad \as, \qquad \forall \md \in \W,\,\forall i \in \natseq{1}{\nbeta}, 
    \]
\end{assumption}
\Cref{assum:amb-decrease} states that the ambiguity sets ``shrink''
to a singleton with probability one. Since 
the ambiguity is expected to decrease as more information is observed,
this is a rather natural assumption, which is satisfied by most classes 
of ambiguity sets, such as the ones discussed in \Cref{sec:learning} (cf. \Cref{rmrk:radius-state}).
}{}
\revision{
\begin{example}[Divergence-based ambiguity sets] \label{ex:ambiguity-decrease-confidence}
    Consider again the divergence-based ambiguity sets introduced in 
    \Cref{sec:learning}.
    \Cref{thm:concentrations} provides an expression for the radius $r_t(\md)$ of two commonly used ambiguity sets, which asymptotically behave as 
    $
        r_t(\md) \sim -t_{\md}^{-1} \log(\confb_t).
    $
    Recall that $t_{\md}$ denotes the number of visits to mode $\md$ at 
    time $t$. In this case, the requirement of \Cref{assum:amb-decrease} results in a lower bound on the rate at which $\confb_t$ may decrease with $t$, posing a trade-off 
    with \Cref{assum:confidences}, which requires summability of the sequence $(\confb_t)_{t \in \N}$. 
    Given ergodicity of the Markov chain (\Cref{assum:ergodicity}),
    it is straightforward to verify --
    using the Borel-Cantelli 
    lemma \cite[Thm. 4.3]{billingsley_ProbabilityMeasure_1995} in conjunction with \cite[Lem. 6]{wolfer_MinimaxLearningErgodic_2019} -- that with 
    probability 1,
    there exists a finite time $T$, such that for all $t > T$ and for all $\md \in \W$, it holds that $t_{\md} \geq c t$,
    where $c > 0$ is a constant depending on specific properties of the Markov chain. Hence, so long as $(\confb_t)_{t\in\N}$ is chosen to satisfy 
    $\lim_{t\to \infty} \tfrac{-\log \confb_t}{t} = 0$ element-wise,
    then \Cref{assum:amb-decrease} is satisfied. Note that the choice in \Cref{ex:confidence-dynamics} satisfies both this requirement and that of \Cref{assum:confidences}.
\end{example}
}{}
\rerevision{
    We are now ready to prove consistency of the \ac{\RAOCP} in the absence of 
chance constraints. Below, we denote $\Xf(\md) = \{x \mid (x,\md) \in \Xf\}$ and similarly $\DR{\Xf}(\md) = \{y \mid (y, \md) \in \DR{\Xf} \}$. 
}{}

\rerevision{}{
}
Below, we denote $\Xf(\md) = \{x \mid (x,\md) \in \Xf\}$ and similarly $\DR{\Xf}(\md) = \{y \mid (y, \md) \in \DR{\Xf} \}$.
\rerevision{}{
We will also pose the following assumptions in the remainder of the section.
\begin{assumption} \label{asm:assumption-consistency}
    \begin{conditions}
        \item The risk levels $\DR{\alpha}_t$ are chosen according to the upper bound of \Cref{prop:conditions-confidences}, i.e.,
        $\DR{\alpha}_t = \frac{\alpha - \bar \conf_t}{1 - \bar{\conf}_t}$ and $\bar \conf_t < \alpha \leq 1$.
        \item $\hat\Xf$ is constructed in relation to the original 
        problem such that for all $\md \in \W$, $\hat \Xf(\md) = \Xf(\md) \times \lrnset \times \I$,
        and $\Xf$ is \ac{RCI} for system \eqref{eq:system-dynamics} in the sense of \Cref{def:classical-RCI}.
    \end{conditions}
\end{assumption}

\begin{thm}[Performance guarantee] \label{thm:out-of-sample}
    Suppose that \cref{asm:assumption-consistency} holds. 
    Then, for 
    any initial learner state $s_0 = s \in \lrnset$ and 
    any initial confidence level $\confb_0 = \confb \in \I$, 
    \begin{statements}
    \item \label{state:general}
    the value function of the \ac{\RAOCP} of horizon $\hor \geq 0$ 
    asymptotically upper bounds the true value function. That is, 
    \begin{equation} 
        \prob[\Vht{t}(x,\md) \geq \cost_{\hor}(x,\md), \forall (x, \md) \in \dom \cost_\hor] \geq 1 - \probt,
        \label{eq:upper-approximation-hp}
    \end{equation}
    for all $ t \in \N$, with $\probt=\nModes \sum_{k=t}^{t+\hor} \|\confb_k\|_1$. 
    \item \label{state:concentric}
    If, furthermore, $\amb_{\conf}$ is
    selected such that 
    \begin{equation}\label{eq:concentric-ambiguity}
        \conf' \leq \conf \implies \amb_{\conf}(\lrn, \md) \subseteq \amb_{\conf'}(\lrn, \md), \quad \forall (\lrn, \md) \in \lrnset \times \W, 
    \end{equation}
    then, 
    \eqref{eq:upper-approximation-hp} holds with 
    \(
        \probt=\nModes \sum_{k=t}^{t+\hor} \nrm{\confb_k}_\infty.
    \)
    \end{statements}
\end{thm}

\begin{proof}
    See \Cref{proof:thm:out-of-sample}.
\end{proof}

    \Cref{thm:out-of-sample} guarantees that with high probability, 
    the \ac{DR} value function provides an upper bound for the 
    value function under full knowledge of the distribution. 
    The corresponding violation rate $\probt$ can be tuned using
    the user-specified confidence levels $\confb_t$.
    
    Note that the violation rate $\probt$ increases with the prediction horizon. 
    This is to be expected, since we essentially require the ambiguity set to contain the true switching distribution
    for all predicted time steps,
    which becomes increasingly difficult as the horizon length increases.
    However, due to the summability of the confidence levels $\confb_t$ (cf. \Cref{assum:confidences}), 
    the violation rate $\probt$ will converge to a finite value as $\hor \to \infty$.
    Similarly, as $t \to \infty$ for fixed $\hor$,
    $\probt$ converges to zero at a summable rate.
    We will use this fact in \cref{cor:asymptotic-performance-bound} 
    to obtain a stronger guarantee asymptotically.

    Before stating the asymptotic extension of \Cref{thm:out-of-sample},
    we briefly highlight the sharper bound for $\probt$ stated
    in \Cref{thm:out-of-sample}-\ref{state:concentric}.
    This result requires that for a given learner state $\lrn$,
    the size of the ambiguity set scales monotonically 
    with the required confidence level. 
    This is satisfied for the described divergence-based ambiguity sets in \Cref{tab:overview}.
    Indeed, the center of the divergence balls are given by the empirical distribution and 
    therefore independent of the confidence level $\conf$. The radii,
    by \cref{thm:concentrations}, are monotone decreasing functions of $\conf$.
    Thus, the intersection of a collection of such ambiguity sets is equal to the 
    ambiguity set with the largest value of $\beta$ (and thus, the smallest radius).

}
\rerevision{}{
    \begin{cor} \label{cor:asymptotic-performance-bound}
        Under the same conditions as \cref{thm:out-of-sample},
        we have with probability one that, 
        \begin{equation}
            \Vht{t}(x,\md) \geq \cost_{\hor}(x,\md)
            \text{ for all sufficiently large }
            t, \label{eq:upper-approximation-general}\\
        \end{equation}
        for all $(x, \md) \in \dom \cost_\hor$.
    \end{cor}
    \begin{proof}
    For fixed $(x, \md) \in \dom \cost_{\hor}$, 
    \cref{thm:out-of-sample} guarantees that 
    \( 
        \prob[\Vht{t}(x,\md) < \cost_{\hor}(x,\md)] \leq \probt, 
    \)
    where due to \cref{assum:confidences}, $\sum_{t=0}^{\infty} \probt = \nModes \sum_{k=0}^{\hor} \sum_{t=0}^{\infty} \nrm{\confb_{t+k}}_1 < \infty$.
    The claim then follows from the Borel-Cantelli lemma \cite[Thm. 4.3]{billingsley_ProbabilityMeasure_1995}.
    \end{proof}
    Having established a performance bound on the true cost,
}
we will now demonstrate consistency of the method, starting with the special 
case where the constraints are independent of the learner state (\Cref{thm:consistency-hard-constraints}),
before tackling the general case in \Cref{thm:consistency}.
To this end, we make the following assumption on the learner state and the corresponding ambiguity set. 
\begin{assumption}[Ambiguity decrease] \label{assum:amb-decrease}
    There exists a sequence $\{\ambdia_t\}_{t \in \N}$ with $\lim_{t \to \infty} \ambdia_t = 0$, such that
    \[
        \sup_{p, q \in \amb_{\idx{\confb_{t}}{i}}(s_t, \md)} \nrm{p-q}_1 \leq \ambdia_{t}\quad \as, \qquad \forall \md \in \W,\,\forall i \in \natseq{1}{\nbeta}, 
    \]
\end{assumption}
\Cref{assum:amb-decrease} states that the ambiguity sets ``shrink''
to a singleton with probability one. Since 
the ambiguity is expected to decrease as more information is observed,
this is a rather natural assumption, which is satisfied by most classes 
of ambiguity sets, such as the ones discussed in \Cref{sec:learning} (cf. \Cref{rmrk:radius-state}).

\begin{thm}[Asymptotic consistency with hard constraints]\label{thm:consistency-hard-constraints}
    Suppose that all constraints are hard constraints, 
    i.e., $\alpha = 0$, so that $\DR{\Ufeas}(z) = \Ufeas(x, \md)$ for all $z = (x,\lrn, \confb,\md)$.
    Then,
    for any state-mode pair $(x,\md) \in \dom \cost_\hor$, any initial 
    learner state $s_0 = s \in \lrnset$ and any initial confidence level $\confb_0 = \confb \in \I$, 
    the optimal cost of the \ac{\RAOCP} of horizon $\hor \geq 0$ almost surely converges from above to the true optimal cost. That is, with 
    probability one,
    \begin{align}
        &\lim_{t \to \infty} \Vht{t}(x,\md) = \cost_{\hor}(x,\md), \label{eq:convergence}
    \end{align}
    for all $(x, \md) \in \dom \cost_{\hor}$.
\end{thm}

\begin{proof}
    See \Cref{proof:thm:consistency-hard-constraints}.
\end{proof}

In the more general case, where aside from the cost, also the constraints 
are probabilistic and therefore dependent on the learner state, some 
additional assumptions on the problem ingredients 
are required. 

\begin{thm}[Asymptotic consistency under chance constraints] \label{thm:consistency}
    Let $\lrn^{\sstar} \in \lrnset$ denote a stationary learner 
    state (cf. \Cref{assum:stationary}) and 
    suppose that for a given state-mode pair $(x, \md) \in \dom \cost_{\hor}$, the following hold:
    \begin{conditions}
        \item \label{cond:blanket-assumption} \Cref{asm:assumption-consistency} holds, and $\Xf(\md)$ is closed and convex; 
        \item \label{cond:consistency-continuity} the costs $\ell(\argdot,\argdot, \md), \Vf(\argdot,\md)$,  constraints $g(\argdot, \argdot, \md, \mdnxt)$ and dynamics \revision{$\fa(\argdot, \argdot, \md, \mdnxt)$}{$f(\argdot, \argdot, \md, \mdnxt)$} are continuously differentiable; 
        \item \label{cond:conic} the ambiguity set $\amb_{\conf}(\lrn, \md)$ is conic representable with convex cone $\cone$ and parameters $E_{\md}(\lrn,\conf)$, $F_{\md}(\lrn,\conf)$ and $b_{\md}(\lrn, \conf)$ that depend smoothly on $\lrn$ and $\conf$;
        \item \label{cond:constraint-qualification} 
        Robinson's constraint qualification \cite[Def. 2.86]{bonnans_PerturbationAnalysisOptimization_2000} holds for \eqref{eq:scenario-tree-reformulation}, for initial states $(x,w)$ $\lrn^\iota = \lrn^\star, \beta^\iota = 0$.
    \end{conditions}
    Then, $\lim_{t \to \infty} \Vht{t}(x, \md) = \cost_\hor(x,\md), \as.$
\end{thm}
\begin{proof}
    \newcommand{\Nd}{\mathcal{N}}
    \newcommand{\Vp}{\tilde \cost_{x,\md}}
    \revision{}{
    Let $x$ and $\md$ be fixed and 
    consider the scenario tree formulation of the \ac{DR}-\ac{OCP} \eqref{eq:scenario-tree-reformulation}, 
    with parameters \( 
        \theta \dfn (
            \lrn^{\iota}, \beta^{\iota}
         )_{\iota \in \Nd}
    \), where $\Nd \dfn \nodes{[0,\hor-1]}$ 
    denotes the set of non-leaf nodes of the tree\footnote{See Appendix~\ref{sec:tractability} for scenario tree related notation.}.
    Problem \eqref{eq:scenario-tree-reformulation} can then be compactly written in the form (as we specify below)
    \begin{equation} \label{eq:proof-parametric-form}
        \begin{aligned} 
          \Vp(\theta) &\dfn \min_{\zeta} \Psi(\zeta) && \stt & \Gamma(\zeta, \theta) \in K.
        \end{aligned}
    \end{equation}
    By \Cref{prop:tractable}, we have that if the parameter vector 
    $\theta = 
    \theta_t 
    \dfn (
        \nodevar{\lrn_{t}}{\iota},
        \nodevar{\confb_t}{\iota}
    )_{\iota \in \Nd}$ is dynamically feasible, i.e., it satisfies \eqref{eq:params-dynamics} with values 
    \( 
        (
            \nodevar{\lrn_t}{0},\nodevar{\confb_t}{0}
        )
     = (
        \lrn_t, 
        \conf_t)
    \) 
    at the root node,
    then 
    \begin{equation} \label{eq:proof-parametric-equality}
        \Vp(\theta_t) = \Vht{t}(x,\md),  \quad \forall t \in \N.
    \end{equation}
    Our proof now consists of two main steps.
    First, we observe that in the limit point $\theta^\star \dfn \lim_{t \to \infty}\theta_t$ (which exists by \Cref{assum:stationary,assum:confidences}), $\Vp(\theta^\star)$ coincides with $\cost_{\hor}(x,w)$. Secondly, we show that the mapping 
    $\Vp(\argdot)$ is continuous at $\theta^\star$.
    }
    \begin{proofsteps}
       \item \label{proof:part1-equality}
       \Cref{assum:confidences} ensures that $\lim_{t \to \infty} \confb_t = \confb^{\sstar} = 0$ and consequently,
       by \Cref{cond:blanket-assumption}, $\lim_{t\to \infty} \DR \alpha_t = \alpha$.
       By \Cref{assum:amb-decrease} and the requirement \eqref{eq:high-confidence}, the Borel-Cantelli lemma \cite[Thm. 4.3]{billingsley_ProbabilityMeasure_1995} implies that 
       for every sequence $(p_t \in \amb_{\confb_t}(\lrn_t, \md))_{t \in \N}$, $\lim_{t\to \infty}p_t = \row{\transmat}{\md}$, a.s.
       Furthermore, as $\lrn_t \to \lrn^{\sstar}$, it follows by \Cref{cond:conic} that
       the mapping $(\lrn, \conf) \mapsto \amb_{\conf}(\lrn,\md)$ is continuous for all $\md \in \W$ and therefore $\amb_{0}(\lrn^{\sstar}, \md) = \{ \row{\transmat}{\md}\}, \forall \md \in \W$.
       \revision{}{
       Thus, by their definitions \eqref{eq:dr-expectation}--\eqref{eq:ambiguous-chance-constraint}, 
       \( \lrisk{\lrn^\star}{\md}{\conf^\star} \) and \( \lriskc{\lrn^\star}{\md}{\bar{\conf}^\star}{\alpha} \) reduce to 
       \( \E_{\row{\transmat}{\md}}\) and 
       \( 
        \AVAR^{\row{\transmat}{\md}}_{\alpha}
        \), respectively. 
       Finally, by \Cref{cond:blanket-assumption},
        \( 
            (x, \lrn, \confb, \md) \in \DR{\Xf} \iff (x,\md) \in \Xf
        \). 
       Therefore, the \ac{DR}-OCP \eqref{eq:risk-averse-OCP} reduces to 
       the nominal counterpart \eqref{eq:stochastic-OCP}, or equivalently 
       \( \Vp(\theta^\star) = \cost_\hor(x,\md) \). 
       }
     \item \label{proof:part2-bonnans}
     \revision{}{
     In order to show that $\Vp$ is continuous at $\theta^\star$, we will show that $\Psi$ and $\Gamma$ are continuously differentiable and $K$ is a closed convex set. Invoking furthermore \Cref{cond:constraint-qualification}, continuity of $\Vp$ then follows from \cite[Prop. 4.4]{bonnans_PerturbationAnalysisOptimization_2000}.
     }
    By inspection of \eqref{eq:scentree-cost} it is clear that $\Psi$ is a linear function, satisfying the requirements.
    We now proceed to demonstrate that the 
    constraints \eqref{eq:scentree-dynamics}--\eqref{eq:scentree-terminal-constraints} admit the desired 
    representation as well.  
    \begin{enumerate}
        \item \label{part:regular-constraints}The constraints \eqref{eq:scentree-dynamics}--\eqref{eq:scentree-term-cost-bounds}, and \eqref{eq:scentree-terminal-constraints} can be directly combined into the form $\Gamma_1(\revision{}{\zeta}, \theta) \in K_1 \dfn \{0\} \times \Re^{n_1}_{+} \times \DR{\Xf}$, where $\Gamma_1$ is a concatenation of the functions $\ell(\argdot,\argdot, \md), \Vf(\argdot,\md)$, and $f(\argdot, \argdot, \md, \mdnxt)$ and therefore continuously differentiable by \Cref{cond:consistency-continuity}.
        $K_1$ is convex due to \Cref{cond:blanket-assumption}.
        \item \label{part:epigraph-constraints}
        Finally, we consider the remaining constraints \eqref{eq:scentree-epigraph-cost} and \eqref{eq:scentree-epigraph-constraints}. Using \eqref{eq:risk-epigraph},
        a conic risk epigraph constraint $(\xi, \gamma) \in \epi \tilde{\rho}^\iota$ with parameters
        $\tilde{E}(\theta^\iota), \tilde{F}(\theta^\iota)$, $\tilde{b}(\theta^\iota)$ and cone $\tilde{\cone}$
        can be written in the desired form 
        \begin{equation} \label{eq:proof-epigraph-functional}
            \tilde{\Gamma}_2(\xi, \chi, \theta^\iota) \in \tilde{K}_2 \dfn \{0\} \times \tilde \cone^* \times \Re^{n_2}_+
        \end{equation}
        with $\chi$ an auxiliary variable and 
        \begin{equation*}
            \tilde{\Gamma}_2(\xi, \gamma, y, \theta^\iota) \dfn \trans{\smallmat{\tilde{E}(\theta^\iota)&
             \tilde{F}(\theta^\iota)&
             I&
             -\tilde{b}(\theta^\iota)}} y + \smallmat{0 \\ -1}\gamma + \smallmat{-I \\ 0}\xi,
        \end{equation*}
        which is differentiable provided that $\tilde{E}(\theta^\iota), \tilde{F}(\theta^\iota)$ and $\tilde{b}(\theta^\iota)$ are differentiable.
        This is ensured exactly by \Cref{cond:conic}, 
        for the cost risk measure $\lrisk{\lrn}{\md}{\conf}$, and 
        thus \eqref{eq:scentree-epigraph-cost} is of the form \eqref{eq:proof-epigraph-functional}. 
        
        Invoking 
        \Cref{prop:conic-robust-avar}, $\lriskc{\lrn}{\md}{\bar \conf}{\DR{\alpha}}$ is conic 
        representable with parameters 
        \begin{equation} \label{eq:cone-parameters-constraints}
        \begin{aligned}
            \bar{E}_{\md}(\lrn, \bar{\conf})
            {}&={}
            \smallmat{E_{\DR{\alpha}}\\0}, \;
            \bar{F}_{\md}(\lrn, \bar{\conf})
            {}={}
            \smallmat{-B & 0 \\
                          E_{\md}(\lrn, \bar{\conf}) & F_ \md(\lrn, \bar{\conf})},\\
            \bar{b}_\md(\lrn, \bar{\conf})
            {}&={}
            \smallmat{b' \\ b_\md(\lrn, \bar{\conf})},
            \;
            \bar{\cone} = \Re^{2(\nModes + 1)}_{+} \times \cone,
        \end{aligned}
        \end{equation}
        with $E_{\DR{\alpha}} = \trans{\smallmat{\1_{\nModes}&-\1_{\nModes}&\DR{\alpha} I & -I}}$, and $B$, $b'$ constant.
        \Cref{cond:blanket-assumption} requires that $\DR{\alpha} = \tfrac{\alpha - \bar{\conf}}{1 - \bar{\conf}}$ is 
        continuously differentiable in $\bar{\conf}$ for all ${\bar{\conf} < 1}$. The case $\bar \beta = 1$ is excluded by design and furthermore inconsequential as $\bar \beta \to 0$.
        As a result, \eqref{eq:scentree-epigraph-constraints}, i.e., 
        constraints $(g(x,u,\md, \mdnxt), 0) \in \epi \lriskc{\lrn}{\md}{\bar \conf}{\DR{\alpha}}$ can be written in the form \eqref{eq:proof-epigraph-functional}, 
        replacing $\xi$ with $g(x,u,\md,\mdnxt)$ -- which 
        preserves continuous differentiability, due to \Cref{cond:consistency-continuity} -- and 
        replacing the risk parameters $\tilde{E}(\theta^\iota), \tilde{F}(\theta^\iota)$ and $\tilde{b}(\theta^\iota)$ and $\tilde{\cone}$
        with those in \eqref{eq:cone-parameters-constraints}.
    \end{enumerate}
    
    \revision{}{
    Given the established differentiability of $\Gamma$, the final requirement of \cite[Prop. 4.4]{bonnans_PerturbationAnalysisOptimization_2000} is equivalent to 
    \Cref{cond:constraint-qualification}, and thus, the result applies.
    }
    \end{proofsteps}
    \revision{}{
    Combining \ref{proof:part1-equality} and \ref{proof:part2-bonnans}, we conclude that $\lim_{t \to \infty} \Vp(\theta_t) = \cost_{\hor}(x,\md)$, and the claim follows from \eqref{eq:proof-parametric-equality}.
    }
\end{proof}
\revision{
\begin{remark}
    The required differentiability of the learner dynamics in \Cref{cond:consistency-continuity} can in principle be relaxed
    \revision{be relaxed}{} as these dynamics are exogenous and can be precomputed 
    over the scenario tree. In this case, the parameter vector in 
    the proof of \Cref{thm:consistency} can be taken to be 
    $\{\nodevar{\lrn}{\iota}, \nodevar{\confb}{\iota}\}_{\iota \in \nodes{[0,\hor-1]}}$, leaving the remainder of the argument mostly intact. 
\end{remark}
}{
We conclude this section with a few brief remarks regarding the conditions of \Cref{thm:consistency}. First, we note that using the learning system 
described in \Cref{sec:learning} (including the ambiguity radius as part of the learner state as suggested in \Cref{rem:learner-sys}), \Cref{cond:conic} is satisfied for all 
divergence-based ambiguity sets considered in \Cref{tab:overview}. Indeed, in the conic formulations provided in Appendix~\ref{sec:conic-reps}, we find that in all cases, the empirical distribution and the ambiguity radius enter linearly in the final conic form of the constraints. 
Second, we remark that Robinson's constraints qualification (\Cref{cond:constraint-qualification}) can be regarded as a generalization of the more well-known Mangasarian-Fromowitz constraint qualification \cite[eq. 2.191]{bonnans_PerturbationAnalysisOptimization_2000} (see also \cite[Prop. 3.3.8]{bertsekas_NonlinearProgramming_1999} or \cite[4.10]{royset_OptimizationPrimer_2021}), which is in turn a generalization of the linear independence constraint qualification (LICQ). It is a very 
common regularity assumption, ensuring several useful properties such as boundedness of Lagrange multipliers. 
Of main importance for the purpose of showing consistency under probabilistic constraints, however, is that it provides metric 
regularity of the (now parametric) feasible set, which implies that 
the distance from the feasible set can upper bounded by a multiple of the constraint violation.
}
\section{Illustrative examples} \label{sec:numerical}

\begin{figure*}[ht!]
    {\footnotesize 
    \centering
    \begin{minipage}[t]{0.19\textwidth}
       \includegraphics{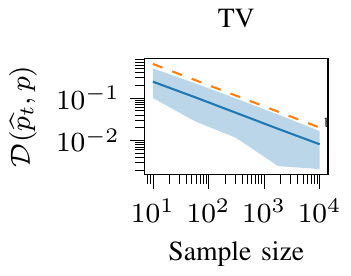}
    \end{minipage}\hfill
    \begin{minipage}[t]{0.19\textwidth}
        \includegraphics{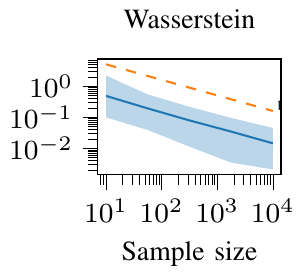}
    \end{minipage}\hfill
    \begin{minipage}[t]{0.19\textwidth}
        \includegraphics{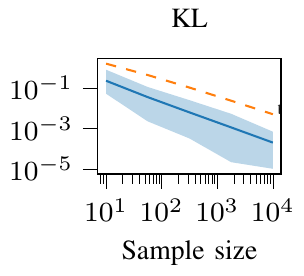}
    \end{minipage}\hfill
    \begin{minipage}[t]{0.19\textwidth}
        \includegraphics{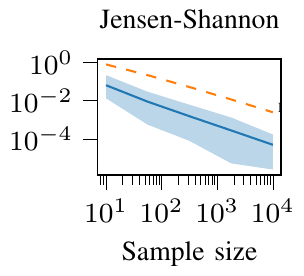}
    \end{minipage}\hfill
    \begin{minipage}[t]{0.19\textwidth}
        \includegraphics{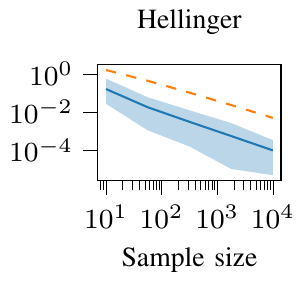}
    \end{minipage}
    }
    \vspace{-0.3cm}
    \caption{Radii of the ambiguity sets versus sample size $t$. The shaded area delineate the empirical $\beta_t$ upper and lower quantiles of $\D{}(\hat{p}_t, p)$ for different divergences $\D{}$, computed over 200 Monte-Carlo runs. The dashed lines represent the theoretical upper bounds given in \Cref{tab:overview}. $p_i = \nicefrac{1}{\nModes}, \; \forall i \in \natseq{1}{\nModes}$, $\nModes = 5$ and $\beta_t = e^{-\nModes} (t+1)^{-2}$. 
    }
    \label{fig:concentration}
\end{figure*}

\subsection{Ambiguity sets} \label{sec:experiments-ambiguity}
To illustrate the concentration inequalities provided in \Cref{thm:concentrations} and \Cref{tab:overview}, 
we select a sequence of confidence levels $\beta_t = e^{-\nModes} (t+1)^{2}$ satisfying summability (\Cref{assum:confidences}) and we plot the radii corresponding to the considered divergences as a function of the sample size $t$ (see \cref{fig:concentration}). For comparison, we recursively estimate
compute the empirical estimate $\hat{p}_t$ of a fixed probability vector $p \in \simplex_{\nModes}$ and plot the empirical
upper and lower $\beta_t$-quantile of $\D{}(\hat{p}_t, p)$ over 200 Monte-Carlo runs. 
For the Wasserstein distance, a quadratic kernel $\elem{\Wker}{\md}{\mdnxt} = (\md - \mdnxt)^2, \; \forall\md,\mdnxt \in \natseq{1}{\nModes}$ was used. 
For all divergences, the given bounds provide reasonable approximations, but in particular, we note that the total variation bound is almost tight. Furthermore, it only requires linear constraints in its conic representation \eqref{eq:conic-representation}, making it an attractive choice in terms of both statistical and computational complexity.
\subsection{Distributionally robust MPC}

{\color{black}}
We consider a Markov jump linear system\footnote{
    For more extensive simulation results, involving nonlinear dynamics and nonconvex constraints, see \cite{schuurmans_SafeLearningBasedMPC_2022_arxiv}
}
$x_{t+1} = A(\md_{t+1}) x_t + B(\md_{t+1}) u_t$, 
with 
\begin{equation} \label{eq:cooling-dynamics}
    A(\md)
    {}={}
    \smallmat{
        1+\frac{\md-1}{\nModes} & 0.01 \\
        0.01                & 1 + 2.5 \frac{\md-1}{\nModes}
    },
    \; 
    B(\md)
    {}={} I, \md \in \natseq{1}{\nModes}
\end{equation}
The state $x_t \in \Re^2$ of this system, inspired by \cite{recht_TourReinforcementLearning_2019},
models the deviation of temperatures from some nominal value of two adjacent servers in a data center. The actuators $u_t \in \Re^2$ correspond to the amount of 
heating ($u_t \geq 0$) or cooling ($u_t < 0$) applied to the corresponding machines.
The mode $w$ models the load on the servers. If $w=1$, the system is idle 
and no heat is generated. If $w = \nModes$, then the processors are fully occupied and a maximum amount of heat is added to the system.
Note that the second server generates more heat under increasing loads.
    The true-but-unknown transition probabilities are computed as
\[
    \elem{\transmat}{i}{j} = \tfrac{
        e^{
            -(j - \nicefrac{i}{2})^2
        }
    }{
        \sum_{w=1}^{\nModes} e^{
            -(w - \nicefrac{i}{2})^2
        }
    }, \; \forall i,j \in \natseq{1}{\nModes}.
\]

As in \cite{recht_TourReinforcementLearning_2019}, we will use a mode-independent quadratic cost $\ell(x,u,\md) = \nrm{x}_2^2 + 10^3 \nrm{u}_2^2$.

We impose hard constraints $-1.5 \leq u \leq 1.5$ on the actuation and 
(nominally) impose \revision{}{robust} chance constraints
\[ 
   \AVAR_{\alpha}^{\row{\transmat}{\md_t}}[
     \row{H}{i} x_{t+1} - h_i \mid x_{t} 
    ] \leq \alpha
   \text{ with } H = \smallmat{I_{\ns} \\ \trans{\1_{\ns}}}, h = \smallmat{1_{\ns} \\ 0.5},
\]
for all $t \in \natseq{0}{\hor-1}$, and $\alpha = 0.19$. Hence, in this example, we have $g_i(x,u,\md,\mdnxt) = \row{H}{i} (A(\mdnxt) x + B(\mdnxt) u) - h_i$.

We compute stabilizing terminal ingredients offline using standard techniques from robust control.
We compute a robust quadratic Lyapunov function $\Vf(x) = \trans{x} \Qf x$ along with a local linear 
control gain $K$, such that $\Vf\big( (A(\md) + B(\md) K) x \big) \leq -\ell(x, Kx), \forall \md \in \W$
by solving \iac{LMI} as in \cite{kothare_RobustConstrainedModel_1996}. The \ac{RCI} terminal set $\Xf$ is computed as the level set
$\Xf = \lev_{\leq \varepsilon} \Vf$, where $\varepsilon = \min_i\{ \nicefrac{h_i}{\nrm{\Qf^{-\nicefrac{1}{2}} \row{H}{i}}_2^2} \}$ is the largest value such that 
$\lev_{\leq \varepsilon} \Vf$ lies inside the polyhedral set $\{x \in \Re^{\ns} \mid H (A(\md) + B(\md)K) \leq h,\, \forall \md \in \W \}$.

For the \ac{DR} controllers below
we choose 
confidence levels $\confb_t = (\conf_t, \bar{\conf}_t)$ with $\conf_t = \bar{\conf}_t = 0.19 t^{-2} < \alpha$ for the cost and  
the constraints, respectively, ensuring that 
\Cref{assum:confidences} is satisfied. 
For simplicity, we use identical confidence levels $\bar{\conf}_t$ for all the constraints.

We compare the proposed \acs{DR}-\acs{MPC} controller with 
\begin{inlinelist*}
    \item the (nominal) stochastic \ac{MPC} controller (see \eqref{eq:stochastic-OCP}), which we call \textbf{omniscient} as it has access to the true transition matrix $\transmat$
    \item the \textbf{robust} \ac{MPC} controller, obtained by solving \eqref{eq:scenario-tree-reformulation}, taking the ambiguity set $\amb_{\conf}(\lrn, \md) = \amb_{\bar{\conf}}(\lrn, \md) = \simplex_{\nModes}$ to be the entire probability simplex, regardless of the mode or learner state.
\end{inlinelist*}
Both the \acp{LMI} involved in the offline computation of the terminal ingredients as the online risk-averse optimal control problem \eqref{eq:scenario-tree-reformulation} are solved using \textsc{mosek} \cite{mosek} through the \textsc{cvxpy} \cite{diamond2016cvxpy} interface.

We fix the number of modes to $\nModes = 3$, and take $\hor = 5$. All computations were performed on an Intel Core i7-7700K CPU at 4.20GHz.

\subsubsection{Timings}
To obtain an indication of the comparative computational 
burden of the different divergence-based ambiguity sets under consideration, we solve the described \ac{DR}-\ac{OCP}
using the considered divergences 10 times each, for random initial states. \Cref{tab:solver-times} reports the average  and maximum observed solver time. As expected, the \ac{TV} and Wasserstein divergence result require the least amount of time, as they introduce only linear constraint. The Hellinger divergence, which introduces second-order cone constraints 
results in slightly longer run times.
The \ac{KL} and JS divergence both introduce exponential cone constraints, resulting in the most computationally demanding \acp{OCP}.

\begin{table}[h]
\centering
\caption{Solver times [\SI{}{\milli\second}] for \eqref{eq:risk-averse-OCP} using different divergence-based ambiguity sets}
\label{tab:solver-times}
\begin{tabular}{cccccc}
    &TV   & Wasserstein&KL     &JS       & Hellinger\\
\midrule
avg. &50.14 & 50.63     &225.6  & 112.00  & 61.31\\
max. &52.20 & 52.02     &235.02 & 118.79  & 62.03 
\end{tabular}
\end{table}

\subsubsection{Closed-loop simulation}

Motivated by previous experiments, we now select the \ac{TV} ambiguity set, and perform a more extensive closed-loop simulation.
Fixing the initial state at $x = \trans{\smallmat{0.5& 0.5}}$, we perform 50 Monte-Carlo simulations 
of the described \ac{MPC} problems for 30 steps. As the simulation 
time is rather short, we initialize the \ac{DR} controller 
with 10 and 100 offline observations of the Markov chain to obtain 
more interesting comparisons. Hence, the simulation below essentially 
compares the controller responses after a sudden disturbance after 10 and 100 time steps.
All considered controllers are recursively feasible and mean-square stabilizing by construction.
By the nature of the problem set-up, the optimal behavior is to 
just barely stabilize the system with minimal control effort. However, 
the larger the uncertainty on the state evolution, the more the controller 
is forced to drive the states further away from the constraint boundary, 
leading to larger control actions and consequently, larger costs.  
\begin{figure}[h]
    \centering
    \includegraphics[]{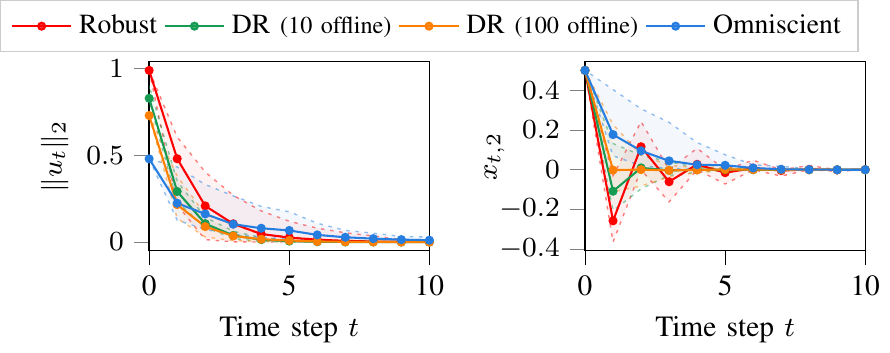}
    \vspace{-0.6cm}
    \caption{Control effort and second component of the state vector over 50 monte-carlo simulations. Full lines depict the means over the realizations and the shaded areas are delineated by the 0.05 and 0.95 quantiles.}
    \label{fig:control-actions}    
\end{figure}
    
\begin{figure}[h]
    \centering
    \includegraphics{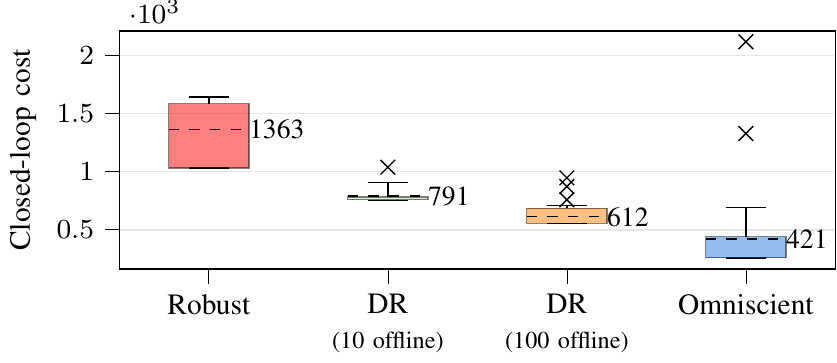}
    \label{subfig:costs-cooling}
    \caption{Box plot of the closed-loop cost over 50 monte-carlo simulations. The annotated lines show the mean. The whiskers depict the 0.05 and 0.95 quantiles.}
    \label{fig:boxplot-cost}
\end{figure}

This behavior can be observed in \cref{fig:control-actions,fig:boxplot-cost}. \cref{fig:control-actions} shows the controls and states over time and \cref{fig:boxplot-cost} presents the distribution of the closed-loop costs (sum of the stage costs over the simulation time). In the first time step, the robust controller 
takes the largest step, driving the state the furthest from the 
constraint boundary. As illustrated in \cref{fig:control-actions} (right), this is particularly pronounced for the second 
component of the state vector, as it is more sensitive to the mode (cf. \eqref{eq:cooling-dynamics}). The \emph{omniscient} stochastic \ac{MPC}, by contrast, has perfect knowledge of the transition probabilities, and by consequence is able to more slowly drive the state to the origin, reducing the control effort considerably.
The \ac{DR} controller naturally `interpolates' between these behaviors.
Initially, it performs only marginally better than the robust controller (due to the very limited number of online learning steps).
As it gets access to increasing sample sizes, however, it gradually
approximates the behavior of the omniscient controller, while guaranteeing satisfaction of the constraints throughout.

\subsubsection{Asymptotic consistency}
\begin{figure}[hb!]
    \centering
    \includegraphics{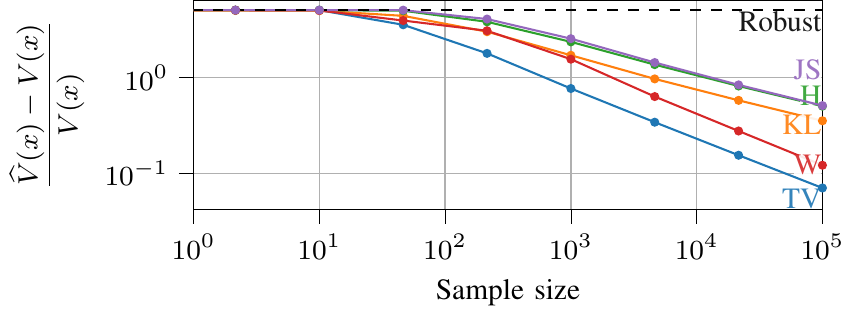}
    \vspace{-0.5cm}
    \caption{Relative suboptimality versus sample size for the example system \eqref{eq:cooling-dynamics}. The dashed line depicts the 
    relative suboptimality of the robust controller: $\nicefrac{(\cost_{\mathrm{r}}-\cost^{\sstar})}{\cost^{\sstar}}$.
    }
    \label{fig:convergence-cost}
\end{figure}
To illustrate the consistency results from \Cref{sec:consistency}, we fix the initial state-mode 
pair $x_0= \trans{\smallmat{0.25 & 0.25}}, w_0=1$ 
and recompute the solution to 
problem \eqref{eq:scenario-tree-reformulation} to obtain $\DR{\cost}^{(t)} \dfn \DR{\cost}^{(t)}_\hor (x_0, \md_0)$ for increasing sample sizes $t$.
For comparison, we compute 
\begin{inlinelist*}
    \item the true value $\cost^{\sstar} \dfn \cost_{\hor}(x_0, \md_0)$ by solving the stochastic \ac{MPC} problem \eqref{eq:stochastic-OCP}, using the true transition probabilities
    \item the robust value function $\cost_{\textrm{r}}$, obtained by 
    solving \eqref{eq:scenario-tree-reformulation}, taking the ambiguity set $\amb_{\conf}(\lrn, \md) = \simplex_{\nModes}$ to be 
    the entire probability simplex, regardless of the mode or learner 
    state.
\end{inlinelist*}

\Cref{fig:convergence-cost} shows the relative difference 
between the \ac{DR} value $\DR{\cost}^{(t)}$ and the true value $\cost^{\sstar}$ 
\revision{}{for the different statistical divergences}. At very low sample sizes, the \ac{DR} controllers achieve the same cost
as the robust controller.
However, as more data is gathered and the ambiguity set is updated, $\DR{\cost}^{(t)}$ approaches $\cost^{\sstar}$ 
from above.

\section{Conclusion}
{\color{black}}
We presented a distributionally robust \ac{MPC} strategy for Markov jump systems with unknown transition probabilities subject to general probabilistic constraints. 
We proved closed-loop constraint satisfaction, mean-square stability,
and consistency of the resulting controller for a broad range of data-driven ambiguity sets.

\bibliographystyle{ieeetr}
\bibliography{references}

\begin{thebibliography}{10}

\bibitem{schuurmans_LearningBasedDistributionallyRobust_2020}
M.~Schuurmans and P.~Patrinos, ``Learning-{{Based Distributionally Robust Model
  Predictive Control}} of {{Markovian Switching Systems}} with {{Guaranteed
  Stability}} and {{Recursive Feasibility}},'' {\em arXiv:2009.04422}, Sept.
  2020.

\bibitem{kouvaritakis_ModelPredictiveControl_2016}
B.~Kouvaritakis and M.~Cannon, {\em Model {{Predictive Control}}}.
\newblock Advanced {{Textbooks}} in {{Control}} and {{Signal Processing}},
  {Cham}: {Springer International Publishing}, 2016.

\bibitem{mesbah_StochasticModelPredictive_2016}
A.~Mesbah, ``Stochastic {{Model Predictive Control}}: {{An Overview}} and
  {{Perspectives}} for {{Future Research}},'' {\em IEEE Control Systems
  Magazine}, vol.~36, pp.~30--44, Dec. 2016.

\bibitem{rawlings_ModelPredictiveControl_2017}
J.~B. Rawlings, D.~Q. Mayne, and M.~M. Diehl, {\em Model Predictive Control:
  Theory, Computation, and Design}.
\newblock {Madison, Wisconsin}: {Nob Hill Publishing}, second~ed., 2017.

\bibitem{mohajerinesfahani_DatadrivenDistributionallyRobust_2018}
P.~Mohajerin~Esfahani and D.~Kuhn, ``Data-driven distributionally robust
  optimization using the {{Wasserstein}} metric: Performance guarantees and
  tractable reformulations,'' {\em Mathematical Programming}, vol.~171,
  pp.~115--166, Sept. 2018.

\bibitem{dupacova_MinimaxApproachStochastic_1987}
J.~Dupa{\v c}ov{\'a}, ``The minimax approach to stochastic programming and an
  illustrative application,'' {\em Stochastics}, vol.~20, pp.~73--88, Jan.
  1987.

\bibitem{parys_DataDecisionsDistributionally_2020}
B.~P.~G. Van~Parys, P.~M. Esfahani, and D.~Kuhn, ``From {{Data}} to
  {{Decisions}}: {{Distributionally Robust Optimization Is Optimal}},'' {\em
  Management Science}, Nov. 2020.

\bibitem{gao_DistributionallyRobustStochastic_2016}
R.~Gao and A.~J. Kleywegt, ``Distributionally {{Robust Stochastic
  Optimization}} with {{Wasserstein Distance}},'' {\em arXiv:1604.02199}, Apr.
  2016.

\bibitem{wiesemann_DistributionallyRobustConvex_2014}
W.~Wiesemann, D.~Kuhn, and M.~Sim, ``Distributionally {{Robust Convex
  Optimization}},'' {\em Operations Research}, vol.~62, pp.~1358--1376, Dec.
  2014.

\bibitem{bertsimas_DatadrivenRobustOptimization_2018a}
D.~Bertsimas, V.~Gupta, and N.~Kallus, ``Data-driven robust optimization,''
  {\em Mathematical Programming}, vol.~167, pp.~235--292, Feb. 2018.

\bibitem{schuurmans_SafeLearningBasedControl_2019}
M.~Schuurmans, P.~Sopasakis, and P.~Patrinos, ``Safe {{Learning}}-{{Based
  Control}} of {{Stochastic Jump Linear Systems}}: A {{Distributionally Robust
  Approach}},'' in {\em 58th {IEEE} {{Conference}} on {{Decision}} and
  {{Control}} ({{CDC}})}, pp.~6498--6503, Dec. 2019.

\bibitem{coppens_DatadrivenDistributionallyRobust_2020}
P.~Coppens, M.~Schuurmans, and P.~Patrinos, ``Data-driven distributionally
  robust {{LQR}} with multiplicative noise,'' in {\em Learning for {{Dynamics}}
  and {{Control}}}, pp.~521--530, {PMLR}, July 2020.

\bibitem{yang_WassersteinDistributionallyRobust_2018}
I.~Yang, ``Wasserstein {{Distributionally Robust Stochastic Control}}: {{A
  Data}}-{{Driven Approach}},'' {\em arXiv:1812.09808}, Dec. 2018.

\bibitem{hakobyan_WassersteinDistributionallyRobust_2020}
A.~Hakobyan and I.~Yang, ``Wasserstein {{Distributionally Robust Motion
  Control}} for {{Collision Avoidance Using Conditional Value}}-at-{{Risk}},''
  {\em arXiv:2001.04727}, Jan. 2020.

\bibitem{hakobyan_DistributionallyRobustRisk_2021}
A.~Hakobyan and I.~Yang, ``Distributionally robust risk map for learning-based
  motion planning and control: {{A}} semidefinite programming approach,'' {\em
  arXiv:2105.00657}, May 2021.

\bibitem{coulson_RegularizedDistributionallyRobust_2019}
J.~Coulson, J.~Lygeros, and F.~D{\"o}rfler, ``Regularized and
  {{Distributionally Robust Data}}-{{Enabled Predictive Control}},'' in {\em
  2019 {{IEEE}} 58th {{Conference}} on {{Decision}} and {{Control}} ({{CDC}})},
  pp.~2696--2701, Dec. 2019.

\bibitem{rahimian_distributionally_2019}
H.~Rahimian and S.~Mehrotra, ``Distributionally {Robust} {Optimization}: {A}
  {Review},'' {\em arXiv:1908.05659}, Aug. 2019.

\bibitem{aswani_ProvablySafeRobust_2013}
A.~Aswani, H.~Gonzalez, S.~S. Sastry, and C.~Tomlin, ``Provably safe and robust
  learning-based model predictive control,'' {\em Automatica}, vol.~49,
  pp.~1216--1226, May 2013.

\bibitem{hewing_ScenarioBasedProbabilisticReachable_2020}
L.~Hewing and M.~N. Zeilinger, ``Scenario-{{Based Probabilistic Reachable
  Sets}} for {{Recursively Feasible Stochastic Model Predictive Control}},''
  {\em IEEE Control Systems Letters}, vol.~4, pp.~450--455, Apr. 2020.

\bibitem{fisac_GeneralSafetyFramework_2019}
J.~F. Fisac, A.~K. Akametalu, M.~N. Zeilinger, S.~Kaynama, J.~Gillula, and
  C.~J. Tomlin, ``A {{General Safety Framework}} for {{Learning}}-{{Based
  Control}} in {{Uncertain Robotic Systems}},'' {\em IEEE Transactions on
  Automatic Control}, vol.~64, pp.~2737--2752, July 2019.

\bibitem{coulson_DataEnabledPredictiveControl_2019}
J.~Coulson, J.~Lygeros, and F.~D{\"o}rfler, ``Data-{{Enabled Predictive
  Control}}: {{In}} the {{Shallows}} of the {{DeePC}},'' {\em
  arXiv:1811.05890}, Mar. 2019.

\bibitem{zhang_RobustLearningBasedPredictive_2022}
X.~Zhang, J.~Liu, X.~Xu, S.~Yu, and H.~Chen, ``Robust {{Learning-Based
  Predictive Control}} for {{Discrete-Time Nonlinear Systems With Unknown
  Dynamics}} and {{State Constraints}},'' {\em IEEE Transactions on Systems,
  Man, and Cybernetics: Systems}, pp.~1--14, 2022.

\bibitem{hewing_learning-based_2020}
L.~Hewing, K.~P. Wabersich, M.~Menner, and M.~N. Zeilinger, ``Learning-{Based}
  {Model} {Predictive} {Control}: {Toward} {Safe} {Learning} in {Control},''
  {\em Annual Review of Control, Robotics, and Autonomous Systems}, vol.~3,
  no.~1, 2020.

\bibitem{bernardini_StabilizingModelPredictive_2012}
D.~Bernardini and A.~Bemporad, ``Stabilizing {{Model Predictive Control}} of
  {{Stochastic Constrained Linear Systems}},'' {\em IEEE Transactions on
  Automatic Control}, vol.~57, pp.~1468--1480, June 2012.

\bibitem{bernardini_ScenariobasedModelPredictive_2009}
D.~Bernardini and A.~Bemporad, ``Scenario-based model predictive control of
  stochastic constrained linear systems,'' in {\em 48th {{IEEE Conference}} on
  {{Decision}} and {{Control}} ({{CDC}}) Held Jointly with 2009 28th {{Chinese
  Control Conference}}}, pp.~6333--6338, {IEEE}, Dec. 2009.

\bibitem{lucia_MultistageNonlinearModel_2013}
S.~Lucia, T.~Finkler, and S.~Engell, ``Multi-stage nonlinear model predictive
  control applied to a semi-batch polymerization reactor under uncertainty,''
  {\em Journal of Process Control}, vol.~23, pp.~1306--1319, Oct. 2013.

\bibitem{leidereiter_QuadraturebasedScenarioTree_2014a}
C.~Leidereiter, A.~Potschka, and H.~G. Bock, ``Quadrature-based scenario tree
  generation for {{Nonlinear Model Predictive Control}},'' {\em IFAC
  Proceedings Volumes}, vol.~47, no.~3, pp.~11087--11092, 2014.

\bibitem{bonzanini2020safe}
A.~D. Bonzanini, J.~A. Paulson, and A.~Mesbah, ``Safe learning-based model
  predictive control under state-and input-dependent uncertainty using scenario
  trees,'' in {\em Proceedings of the IEEE Conference on Decision and Control.
  Jeju Island, Republic of Korea. Submitted}, 2020.

\bibitem{costa_discrete-time_2005}
O.~L. d.~V. Costa, M.~D. Fragoso, and R.~P. Marques, {\em Discrete-time
  {Markov} jump linear systems}.
\newblock Probability and its applications, London: Springer, 2005.

\bibitem{patrinos_StochasticModelPredictive_2014}
P.~Patrinos, P.~Sopasakis, H.~Sarimveis, and A.~Bemporad, ``Stochastic model
  predictive control for constrained discrete-time {{Markovian}} switching
  systems,'' {\em Automatica}, vol.~50, pp.~2504--2514, Oct. 2014.

\bibitem{lucia_StabilityPropertiesMultistage_2020}
S.~Lucia, S.~Subramanian, D.~Limon, and S.~Engell, ``Stability properties of
  multi-stage nonlinear model predictive control,'' {\em Systems \& Control
  Letters}, vol.~143, p.~104743, Sept. 2020.

\bibitem{sopasakis_RiskaverseModelPredictive_2019}
P.~Sopasakis, D.~Herceg, A.~Bemporad, and P.~Patrinos, ``Risk-averse model
  predictive control,'' {\em Automatica}, vol.~100, pp.~281--288, Feb. 2019.

\bibitem{singh_FrameworkTimeConsistentRiskSensitive_2018}
S.~Singh, Y.-L. Chow, A.~Majumdar, and M.~Pavone, ``A {{Framework}} for
  {{Time}}-{{Consistent}}, {{Risk}}-{{Sensitive Model Predictive Control}}:
  {{Theory}} and {{Algorithms}},'' {\em arXiv:1703.01029}, Apr. 2018.

\bibitem{shapiro2009lectures}
A.~Shapiro, D.~Dentcheva, and A.~Ruszczy{\'n}ski, {\em Lectures on stochastic
  programming: modeling and theory}.
\newblock SIAM, 2009.

\bibitem{beirigo_online_2018}
R.~L. Beirigo, M.~G. Todorov, and A.~M. S.~Barreto, ``Online {TD}($\lambda$)
  for discrete-time {Markov} jump linear systems,'' in {\em 57th {IEEE}
  {Conference} on {Decision} and {Control} ({CDC})}, pp.~2229--2234, Dec. 2018.

\bibitem{he_reinforcement_2019}
S.~He, M.~Zhang, H.~Fang, F.~Liu, X.~Luan, and Z.~Ding, ``Reinforcement
  learning and adaptive optimization of a class of {Markov} jump systems with
  completely unknown dynamic information,'' {\em Neural Computing and
  Applications}, Apr. 2019.

\bibitem{derman_DistributionalRobustnessRegularization_2020a}
E.~Derman and S.~Mannor, ``Distributional {{Robustness}} and {{Regularization}}
  in {{Reinforcement Learning}},'' July 2020.

\bibitem{xu_DistributionallyRobustMarkov_2010}
H.~Xu and S.~Mannor, ``Distributionally {{Robust Markov Decision Processes}},''
  in {\em Advances in {{Neural Information Processing Systems}} 23} (J.~D.
  Lafferty, C.~K.~I. Williams, J.~{Shawe-Taylor}, R.~S. Zemel, and A.~Culotta,
  eds.), pp.~2505--2513, {Curran Associates, Inc.}, 2010.

\bibitem{mannor_RobustMDPsKRectangular_2016}
S.~Mannor, O.~Mebel, and H.~Xu, ``Robust {{MDPs}} with k-{{Rectangular
  Uncertainty}},'' {\em Mathematics of Operations Research}, vol.~41,
  pp.~1484--1509, Nov. 2016.

\bibitem{ahmadi_ConstrainedRiskAverseMarkov_2020}
M.~Ahmadi, U.~Rosolia, M.~D. Ingham, R.~M. Murray, and A.~D. Ames,
  ``Constrained {{Risk-Averse Markov Decision Processes}},'' {\em
  arXiv:2012.02423}, Dec. 2020.

\bibitem{cherukuri_ConsistencyDistributionallyRobust_2020}
A.~Cherukuri and A.~R. Hota, ``Consistency of {{Distributionally Robust Risk}}-
  and {{Chance}}-{{Constrained Optimization}} under {{Wasserstein Ambiguity
  Sets}},'' {\em arXiv:2012.08850}, Dec. 2020.

\bibitem{guo_ConvergenceAnalysisMathematical_2017}
S.~Guo, H.~Xu, and L.~Zhang, ``Convergence {{Analysis}} for {{Mathematical
  Programs}} with {{Distributionally Robust Chance Constraint}},'' {\em SIAM
  Journal on Optimization}, vol.~27, pp.~784--816, Jan. 2017.

\bibitem{nemirovski2012safe}
A.~Nemirovski, ``On safe tractable approximations of chance constraints,'' {\em
  European Journal of Operational Research}, vol.~219, no.~3, pp.~707--718,
  2012.

\bibitem{sopasakis_risk-averse_2019c}
P.~Sopasakis, M.~Schuurmans, and P.~Patrinos, ``Risk-averse risk-constrained
  optimal control,'' in {\em 18th {European} {Control} {Conference} ({ECC})},
  pp.~375--380, June 2019.

\bibitem{bertsekas_DynamicProgrammingOptimal_2005}
D.~P. Bertsekas, {\em Dynamic Programming and Optimal Control. {{Vol}}. 1}.
\newblock Athena Scientific Optimization and Computation Series, {Belmont,
  Mass}: {Athena Scientific}, third~ed., 2005.

\bibitem{billingsley_ProbabilityMeasure_1995}
P.~Billingsley, {\em Probability and Measure}.
\newblock Wiley Series in Probability and Mathematical Statistics, {New York}:
  {Wiley}, third~ed., 1995.

\bibitem{krishnamurthy_PartiallyObservedMarkov_2016}
V.~Krishnamurthy, {\em Partially Observed {{Markov}} Decision Processes: From
  Filtering to Controlled Sensing}.
\newblock {Cambridge}: {Cambridge University Press}, 2016.

\bibitem{jiang_RiskAverseTwoStageStochastic_2018}
R.~Jiang and Y.~Guan, ``Risk-{{Averse Two}}-{{Stage Stochastic Program}} with
  {{Distributional Ambiguity}},'' {\em Operations Research}, vol.~66,
  pp.~1390--1405, Oct. 2018.

\bibitem{rahimian_EffectiveScenariosMultistage_2021}
H.~Rahimian, G.~Bayraksan, and T.~{Homem-de-Mello}, ``Effective {{Scenarios}}
  in {{Multistage Distributionally Robust Optimization}} with a {{Focus}} on
  {{Total Variation Distance}},'' {\em arXiv:2109.06791}, Sept. 2021.

\bibitem{wolfer_MinimaxLearningErgodic_2019}
G.~Wolfer and A.~Kontorovich, ``Minimax {{Learning}} of {{Ergodic Markov
  Chains}},'' in {\em Algorithmic {{Learning Theory}}}, pp.~903--929, Mar.
  2019.

\bibitem{vaart_WeakConvergenceEmpirical_2000}
A.~W. van~der Vaart and J.~A. Wellner, {\em Weak Convergence and Empirical
  Processes: With Applications to Statistics}.
\newblock {New York}: {Springer}, 2000.

\bibitem{csiszar_MethodTypes_1998}
I.~Csiszar, ``The method of types,'' {\em IEEE Transactions on Information
  Theory}, vol.~44, pp.~2505--2523, Oct. 1998.

\bibitem{cover_ElementsInformationTheory_2006}
T.~M. Cover and J.~A. Thomas, {\em Elements of Information Theory}.
\newblock {Hoboken, N.J}: {Wiley-Interscience}, 2nd~ed., 2006.

\bibitem{mardia_ConcentrationInequalitiesEmpirical_2019}
J.~Mardia, J.~Jiao, E.~T{\'a}nczos, R.~D. Nowak, and T.~Weissman,
  ``Concentration inequalities for the empirical distribution of discrete
  distributions: Beyond the method of types,'' {\em Information and Inference:
  A Journal of the IMA}, Nov. 2019.

\bibitem{csiszar_InformationTheoryCoding_2011}
I.~Csisz{\'a}r and J.~K{\"o}rner, {\em Information Theory: Coding Theorems for
  Discrete Memoryless Systems}.
\newblock {Cambridge ; New York}: {Cambridge University Press}, 2nd ed~ed.,
  2011.

\bibitem{gibbs_ChoosingBoundingProbability_2002}
A.~L. Gibbs and F.~E. Su, ``On {{Choosing}} and {{Bounding Probability
  Metrics}},'' {\em International Statistical Review}, vol.~70, no.~3,
  pp.~419--435, 2002.

\bibitem{bayraksan_DataDrivenStochasticProgramming_2015}
G.~Bayraksan and D.~K. Love, ``Data-{{Driven Stochastic Programming Using
  Phi}}-{{Divergences}},'' in {\em The {{Operations Research Revolution}}}
  (D.~Aleman, A.~Thiele, J.~C. Smith, and H.~J. Greenberg, eds.), pp.~1--19,
  {INFORMS}, Sept. 2015.

\bibitem{ben-tal_RobustSolutionsOptimization_2012}
A.~{Ben-Tal}, D.~{den Hertog}, A.~De~Waegenaere, B.~Melenberg, and G.~Rennen,
  ``Robust {{Solutions}} of {{Optimization Problems Affected}} by {{Uncertain
  Probabilities}},'' {\em Management Science}, vol.~59, pp.~341--357, Nov.
  2012.

\bibitem{yanikoglu_SafeApproximationsAmbiguous_2012}
{\.I}.~Yan{\i}ko{\u g}lu and D.~den Hertog, ``Safe {{Approximations}} of
  {{Ambiguous Chance Constraints Using Historical Data}},'' {\em INFORMS
  Journal on Computing}, Nov. 2012.

\bibitem{ruszczynski_RiskaverseDynamicProgramming_2010b}
A.~Ruszczy{\'n}ski, ``Risk-averse dynamic programming for {{Markov}} decision
  processes,'' {\em Mathematical Programming}, vol.~125, pp.~235--261, Oct.
  2010.

\bibitem{korda_StronglyFeasibleStochastic_2011}
M.~Korda, R.~Gondhalekar, J.~Cigler, and F.~Oldewurtel, ``Strongly feasible
  stochastic model predictive control,'' in {\em 50th {{IEEE Conference}} on
  {{Decision}} and {{Control}} and {{European Control Conference}}},
  pp.~1245--1251, Dec. 2011.

\bibitem{bonnans_PerturbationAnalysisOptimization_2000}
J.~F. Bonnans and A.~Shapiro, {\em Perturbation Analysis of Optimization
  Problems}.
\newblock Springer Series in Operations Research, {New York}: {Springer}, 2000.

\bibitem{bertsekas_NonlinearProgramming_1999}
D.~P. Bertsekas, {\em Nonlinear Programming}.
\newblock {Belmont, Mass}: {Athena Scientific}, second~ed., 1999.

\bibitem{royset_OptimizationPrimer_2021}
J.~O. Royset and R.~J.-B. Wets, {\em An Optimization Primer}.
\newblock Springer Series in Operations Research and Financial Engineering,
  {Cham, Switzerland}: {Springer}, 2021.

\bibitem{schuurmans_SafeLearningBasedMPC_2022_arxiv}
M.~Schuurmans, A.~Katriniok, C.~Meissen, H.~E. Tseng, and P.~Patrinos, ``Safe,
  {{Learning-Based MPC}} for {{Highway Driving}} under {{Lane-Change
  Uncertainty}}: {{A Distributionally Robust Approach}},'' June 2022.

\bibitem{recht_TourReinforcementLearning_2019}
B.~Recht, ``A {{Tour}} of {{Reinforcement Learning}}: {{The View}} from
  {{Continuous Control}},'' {\em Annual Review of Control, Robotics, and
  Autonomous Systems}, vol.~2, no.~1, pp.~253--279, 2019.

\bibitem{kothare_RobustConstrainedModel_1996}
M.~V. Kothare, V.~Balakrishnan, and M.~Morari, ``Robust constrained model
  predictive control using linear matrix inequalities,'' {\em Automatica},
  vol.~32, no.~10, pp.~1361--1379, 1996.

\bibitem{mosek}
{MOSEK ApS}, {\em The MOSEK optimization toolbox for MATLAB manual. Version
  8.1.}, 2017.

\bibitem{diamond2016cvxpy}
S.~Diamond and S.~Boyd, ``{CVXPY}: {A} {P}ython-embedded modeling language for
  convex optimization,'' {\em Journal of Machine Learning Research}, vol.~17,
  no.~83, pp.~1--5, 2016.

\bibitem{shapiro_DualityTheoryConic_2001}
A.~Shapiro, ``On {{Duality Theory}} of {{Conic Linear Problems}},'' in {\em
  Semi-{{Infinite Programming}}} (P.~Pardalos, M.~{\'A}. Goberna, and M.~A.
  L{\'o}pez, eds.), vol.~57, pp.~135--165, {Boston, MA}: {Springer US}, 2001.

\bibitem{pflug_MultistageStochasticOptimization_2014}
G.~C. Pflug and A.~Pichler, {\em Multistage {{Stochastic Optimization}}}.
\newblock Springer {{Series}} in {{Operations Research}} and {{Financial
  Engineering}}, {Cham}: {Springer International Publishing}, 2014.

\bibitem{schuurmans_LearningBasedRiskAverseModel_2020}
M.~Schuurmans, A.~Katriniok, H.~E. Tseng, and P.~Patrinos, ``Learning-{{Based
  Risk}}-{{Averse Model Predictive Control}} for {{Adaptive Cruise Control}}
  with {{Stochastic Driver Models}},'' in {\em {{IFAC}} 2020 {{World
  Congress}}}, ({Berlin}), pp.~15337--15342, 2020.

\bibitem{rockafellar_VariationalAnalysis_1998}
R.~T. Rockafellar and R.~J.~B. Wets, {\em Variational {{Analysis}}}, vol.~317
  of {\em Grundlehren Der Mathematischen {{Wissenschaften}}}.
\newblock {Berlin, Heidelberg}: {Springer Berlin Heidelberg}, 1998.

\end{thebibliography}
\begin{IEEEbiography}[{\includegraphics[width=1in,height=1.25in,clip,keepaspectratio]{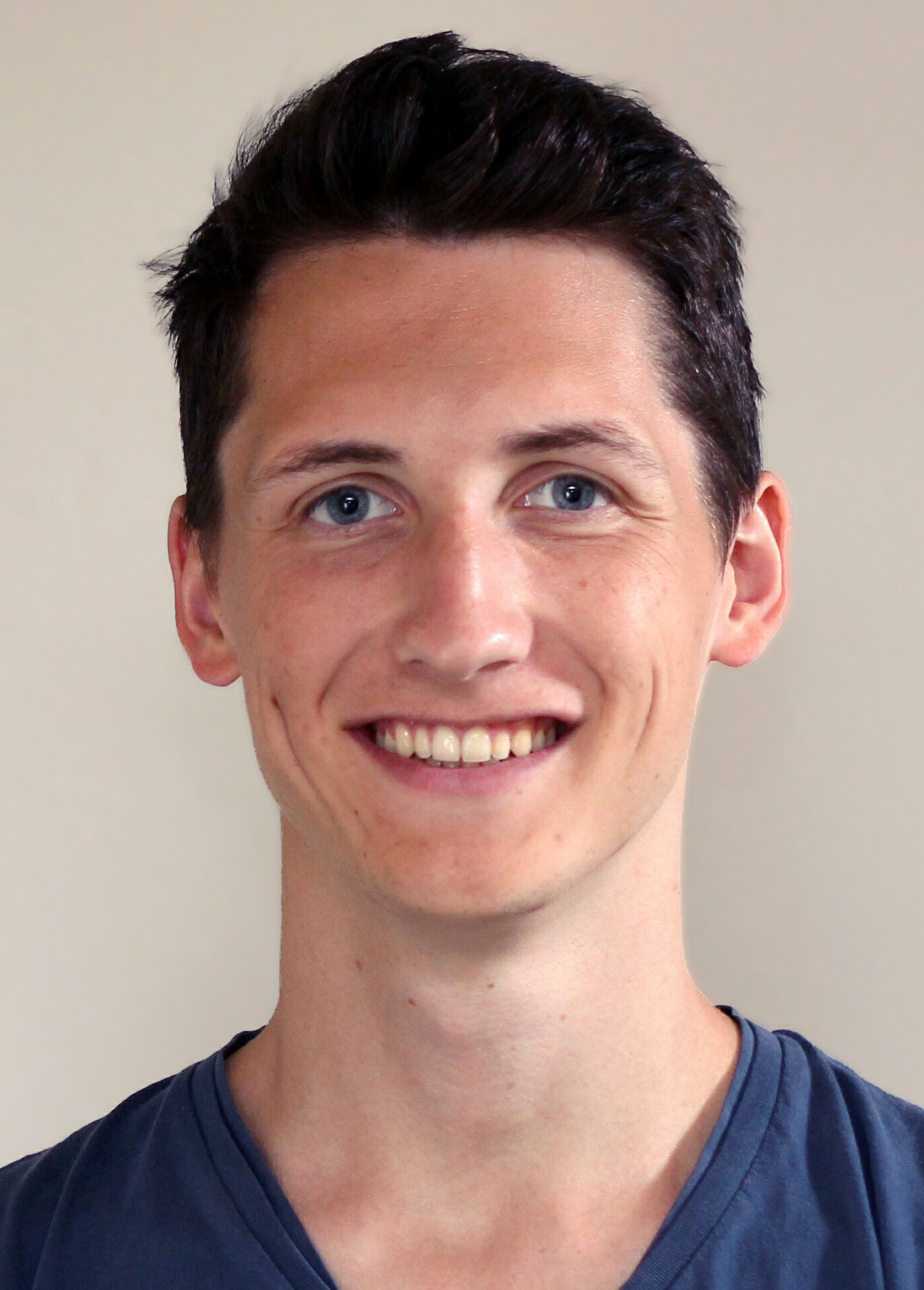}}]{Mathijs Schuurmans}
obtained a Bachelor's degree (BSc) in Electrical and Mechanical Engineering and a Master's (MSc) in Mathematical Engineering from KU Leuven, Leuven, Belgium in 2016 and 2018, respectively. 
He is currently a PhD candidate at the Department of Electrical Engineering (ESAT) of KU Leuven. His research is focused on data-driven model predictive control of stochastic systems, focusing on distributionally robust approaches for safety-critical applications in autonomous driving.    \end{IEEEbiography}
\begin{IEEEbiography}[{\includegraphics[width=1.1in,height=1.25in,clip,keepaspectratio]{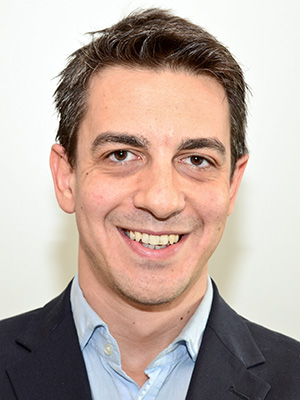}}]{Panagiotis Patrinos}
Panagiotis (Panos) Patrinos is associate professor at the Department of Electrical Engineering (ESAT) of KU Leuven, Belgium. In 2014 he was a visiting professor at Stanford University. He received his PhD in Control and Optimization, M.S. in Applied Mathematics and M.Eng. in Chemical Engineering from the National Technical University of Athens in 2010, 2005 and 2003, respectively. After his PhD he held postdoc positions at the University of Trento and IMT Lucca, Italy, where he became an assistant professor in 2012. His current research interests lie in the intersection of optimization, control and learning. In particular he is interested in the theory and algorithms for structured nonconvex optimization as well as learning-based, model predictive control with a wide range of applications including autonomous vehicles, machine learning and signal processing. He is the co-recipient of the 2020 best paper award in International Journal of Circuit Theory \& Applications \end{IEEEbiography}
\begin{appendix}
\ifJournal

\subsection{Conic representations of divergence-based ambiguity sets}\label{sec:conic-reps}
In this section, we briefly present conic representations for the divergences included in \Cref{tab:overview}. 
\rerevision{}{
These reformulations, which may not be unique, 
follow from rather straightforward manipulations, 
but are included here for completeness.
}
In the following, $\hat{p}$ represents a given empirical distribution and $r$ represents a given radius. $p$ is the candidate distribution over which the expectation 
is maximized. 
Below, we describe the set 
\(
    \{ p \in \simplex_{\nModes} \mid \D{}(\hat{p}, p) \leq r \}
\)
for different divergences $\D{}$. Note that 
in all cases, we have the constraint $p \in \simplex_{\nModes}$, which is represented by linear equality and inequality constraints.

   \subsubsection*{Total variation}
   \[ 
    \begin{aligned}
      &\nrm{p - \hat{p}}_{1} \leq r
      \iff
        \exists \nu \in \Re^{\nModes}: \begin{cases}
        -\nu_i \leq p_i - \hat{p}_i \leq \nu_i, \; \forall i \in \natseq{1}{\nModes}\\
        \sum_{i=1}^{\nModes} \nu_i \leq r.
      \end{cases}
    \end{aligned}
    \]
    \subsubsection*{Kullback-Leibler} \[
        \begin{aligned}
        & \tsum_{i=1}^{\nModes} \hat{p}_i \log \tfrac{\hat{p}_i}{p_i} \leq r \\
        \iff & \exists \nu \in \Re^{\nModes}: \begin{cases}
            \hat{p}_i \log \tfrac{\hat{p}_i}{p_i} \leq \nu_i & \forall i \in \natseq{1}{\nModes}\\ 
            \sum_{i=1}^{\nModes} \nu_i \leq r,
        \end{cases}\\
        \iff&\exists \nu \in \Re^{\nModes}:
        \begin{cases}
                (-\nu_i, \hat{p}_i, p_i) \in \cone_{\exp} &\forall i \in \natseq{1}{\nModes}\\ 
                \sum_{i=1}^{\nModes} \nu_i \leq r,
        \end{cases}\\
        \end{aligned}
    \]
    where 
    \( 
        \cone_{\exp} = \{ (x_1,x_2, x_3) :  x_1 \leq x_2 \log \left(
            \tfrac{x_3}{x_2}
        \right)
        , x_2 > 0 \} 
    \)
    denotes the exponential cone, the dual cone of which is 
    \( 
        \cone^{*}_{\exp} \dfn \{ (u,v,w) : u < 0, w > 0, - u \log (-\tfrac{u}{w}) + u -v \leq 0\}.
    \)
    \subsubsection*{Jensen-Shannon}    
    \[ 
    \begin{aligned}
      \D[J](\hat{p}, p) &\dfn \tfrac{1}{2}
      \Big(
        \DKL\big(
            p,\tfrac{1}{2}(\hat{p} + p)
        \big) 
            + 
        \DKL\big(
            \hat{p}, \tfrac{1}{2}(\hat{p} + p)
            \big)
        \Big) \leq r \\
        \iff&
        \sum_{i=1}^d p_i \log \left( \frac{2 p_i}{p_i + \hat{p}_i} \right) + \hat{p}_i \log \left( \frac{2 \hat{p}_i}{p_i + \hat{p}_i} \right)\leq 2r \\
        \iff& \exists x, y \in \Re^{\nModes}: \begin{cases}
            \sum_{i} x_i + y_i \leq 2 r\\
            (-x_i, p_i, \tfrac{1}{2}(p_i + \hat{p}_i)) \in \cone_{\exp}\\
            (-y_i, \hat{p}_i, \tfrac{1}{2}(p_i + \hat{p}_i)) \in \cone_{\exp}.
        \end{cases}
    \end{aligned}
    \]
    \subsubsection*{Hellinger} 
    Using the fact that 
    \[ 
    \begin{aligned}
        \D[H]^2(\hat{p}, p) &= \tsum_{i=1}^{d} (\sqrt{p_i} - \sqrt{\hat{p}_i})^2 = 2 (1 - \tsum_{i=1}^d \sqrt{p_i \hat{p}_i}), 
    \end{aligned}
    \] 
    we have 
    \(
        \D[H]^2(\hat{p}, p) \leq r 
    \iff
        \hat{r} \dfn 1 - \frac{r}{2}   
        {}\leq{}
        \sum_{i=1}^d \sqrt{p_i \hat{p}_i}
    \), and thus 
    \[ 
    \begin{aligned}
    \D[H]^2(\hat{p}, p) \leq r 
        \iff&\exists \nu \in \Re^d : 
    \begin{cases}
        \hat{r} \leq \tsum_{i=1}^d \sqrt{\hat{p}_i} \nu_i \\ 
        \nu_i^2 \leq p_i, & \forall i \in \natseq{1}{\nModes}.
    \end{cases}
    \end{aligned}
    \]
    The constraint $\nu_i^2 \leq p_i$ can be reformulated as 
    \[ 
        \left \|
              (2 \nu_i, 
              p_i - 1)
        \right \|_2 \leq p_i + 1 \iff (p_i+1, 2\nu_i, p_i-1) \in \cone_{\operatorname{SO}}, 
    \]
    with 
    \(
    \cone_{\operatorname{SO}} \dfn \{ x = (t, y) \in \Re^{1+n} \mid t \geq \nrm{y}_2 \}
    \) the (self-dual) second order (or quadratic) cone. 
    \subsubsection*{Wasserstein} 
    Assuming that $\W$ is a metric space with distance metric 
    $\dist: \W \times \W \to \Re_+$, then we define 
    the fixed-support $q$-Wasserstein distance, for $q > 0$ as 
    \[ 
        \D[{W}][q](p, \hat{p}) \dfn \min_{\Pi \in \Re^{\nModes\times \nModes}_+} \Big\{ 
            \big( \tsum_{i,j \in \W} \elem{\Pi}{i}{j}\elem{\Wker}{i}{j}^q \big)^{\nicefrac{1}{q}} \Big | \begin{array}{r}
                \Pi \1_\nModes = p \\
                \Pi^\top \1_{\nModes} = \hat{p}
            \end{array} 
        \Big\},
    \]
    where $\Wker \dfn (\dist(i,j))_{i,j \in \W} \in \Re_+^{\nModes \times \nModes}$ is the \textit{distance kernel} induced by $\dist$. Since $\D[W][q](p, \hat{p}) \leq r$ is equivalent to $\D[W][q](p, \hat{p})^q \leq \tilde{r} \dfn r^q$, the $q$-Wasserstein case can be reduced to 
    the $1$-Wasserstein case with distance kernel $\tilde \Wker \dfn \Wker^q$. Thus, we may drop the subscript $q$ without loss of generality. From the definition, we immediately obtain the representation  
    \[ 
    \begin{aligned}
          &\D[{W}][1](p, \hat{p}) \leq r
          \iff \exists \Pi \geq 0 :
           \begin{cases}
                \tsum_{i,j \in \W} \elem{\Pi}{i}{j}\elem{\tilde{\Wker}}{i}{j} \leq \tilde r,\\
                \Pi \1_\nModes = p, \; 
                \Pi^\top \1_{\nModes} = \hat{p}, 
           \end{cases}
    \end{aligned}
    \]
    consisting only of linear constraints. 

\subsection{Conic reformulation over scenario trees} \label{sec:tractability}

\rerevision{}{
By definition, a conic risk measure $\rho$ is given as the optimal value of a standard \ac{CP}. Under strong duality, which holds if the \ac{CP} is strictly feasible \cite[Prop. 2.1]{shapiro_DualityTheoryConic_2001}, its epigraph $\epi \rho \dfn \{ (G,\gamma) \in \Re^{\nModes + 1} \mid \gamma \geq \rho[G]\}$ can be characterized as \cite{sopasakis_risk-averse_2019c}
\begin{equation} \label{eq:risk-epigraph}
    \ifSingleColumn
    \epi \rho = \{(G,\gamma) \in \Re^{\nModes + 1} \mid \exists y: \trans{E}y=G, \trans{F}y=0, y \geqc{\cone^*} 0, \gamma \geq \trans{b}y \}.\else 
    \epi \rho = \left\{
        (G,\gamma) \in \Re^{\nModes + 1} \sep \begin{matrix}\exists y: \trans{E}y=G, \trans{F}y=0,\\
         y \in \cone^*, \gamma \geq \trans{b}y \end{matrix} 
        \right\}.
    \fi
\end{equation}
\revision{
Since the \ac{TV} ambiguity set defined in \Cref{sec:learning} as well as the ambiguity set inducing the average value-at-risk are polyhedra, they are conic representable (taking the nonnegative orthant as the cone $\cone$).
Similarly, the \ac{KL} ambiguity set, and similarly the entropic value-at-risk \cite{ahmadi-javid_EntropicValueatRiskNew_2012} are known 
to be conic representable \cite{sopasakis_risk-averse_2019c,parys_DataDecisionsDistributionally_2020}.
}
{
Aside from the ambiguity sets described in Appendix~\ref{sec:conic-reps},
}
it is not difficult to show that the worst-case average value-at-risk \eqref{eq:ambiguous-chance-constraint} over a conic representable ambiguity set also 
defines a conic risk measure:
\begin{proposition} \label{prop:conic-robust-avar}
    Let $\amb = \{ p \in \simplex_{\nModes} \mid \exists \nu: \bar E p + \bar F \nu \leqc{\cone} \bar b \}$ be a conic-representable ambiguity set. Then, the risk measure
    $\bar \rho = \max_{p \in \amb} \AVAR_{\alpha}^{p}$ is a conic risk measure.
\end{proposition}
\begin{proof}
    For any reference distribution $p\in \simplex_\nModes$,
    the ambiguity set $\amb_{\AVAR}$ inducing $\AVAR_{\alpha}^{p}$ can be written in the form 
    \eqref{eq:conic-representation} with 
    $E=\trans{\smallmat{\1_{\nModes}&-\1_{\nModes}&\alpha I&-I}}$,
    $F=0$, $\cone=\Re^{2(\nModes + 1)}_{+}$ the nonnegative orthant,
    and
    $b = \trans{\smallmat{1&-1&\trans{p}&0}}$ 
    (which is of the form $b=b' + Bp$) \cite{sopasakis_risk-averse_2019c}.
    Writing out the definition of $\max_{p \in \amb} \AVAR_{\alpha}^{p}$ and rearranging terms yields 
    \(
      \max_{p \in \amb} \AVAR_{\alpha}^{p}[z] 
      = \max_{\mu} \left\{ \trans{\mu}z \sep \exists \nu: \smallmat{E\\0}\mu + \smallmat{-B & 0\\\bar E & \bar F} \nu \leqc{\Re^{2(\nModes + 1)}_{+} \times \cone}
      \smallmat{b' \\ \bar b} \right\},
    \)
    which is exactly of the form \eqref{eq:conic-representation}.
\end{proof}
Thus, if for all $(\lrn, \md, \conf) \in \lrnset \times \W \times [0,1]$, $\amb_{\conf}(\lrn,\md)$ is conic representable, then $\lrisk{\lrn}{\md}{\conf}$ and $\lriskc{\lrn}{\md}{\conf}{\alpha}$ are conic risk measures.
Therefore, problem \eqref{eq:risk-averse-OCP} can be cast to 
a finite dimensional optimization problem, as we now describe. 
}

\revision{}{
Since $\W$ is a finite set, the possible realizations of $\seq{\md}{0}{\hor}$ can be enumerated and represented on a \emph{scenario tree}.
}
\ifJournal
A scenario tree with horizon $N$
represents the natural filtration of $(\Omega, \F, \prob)$ induced by $\seq{\md}{0}{\hor}$ \cite{pflug_MultistageStochasticOptimization_2014}.
Any adapted stochastic process $(z_t)$ can be represented on such a scenario tree. We denote the value of $z_{t}$ corresponding to a 
node $\iota$ in the tree as $\nodevar{z}{\iota}$.

The set of nodes in the tree are partitioned into time steps or \emph{stages}. The set of nodes at a stage $k$ is denoted by $\nodes{k}$,
and similarly, for $k_0 < k_1 \in \natseq{0}{\hor}$,   
$\nodes{[k_0,k_1]} = \bigcup_{k=k_0}^{k_1} \nodes{k}$.
For a given node $\iota \in \nodes{t}$, $t \in \natseq{0}{\hor-1}$, we call a node $\iota_\pplus \in \nodes{t+1}$ that can be reached from $\iota$ in one step a \emph{child} node, denoted $\iota_{\pplus} \in \child{\iota}$.
An $\hor$-step policy $\pi$ can thus be identified with 
a collection of control actions 
\(
    \useq =\{
        u^\iota \mid \iota \in \nodes{[0, \hor-1]} 
\}.
\)
It therefore suffices to optimize over a finite number of decision variables rather than infinite-dimensional control laws.

\begin{proposition}[\rerevision{Tractable}{Finite-dimensional} reformulation] \label{prop:tractable}
    Given an initial state $(x, \md)$ and parameters $(\nodevar{\lrn}{\iota}, \nodevar{\beta}{\iota})_{\iota \in \nodes{[0,\hor-1]}}$,
    consider an $N$-stage scenario tree and the corresponding optimal control problem
    \begin{subequations} \label{eq:scenario-tree-reformulation}
    \begin{align} 
        &\minimize_{\xi, \tau, \xseq, \useq}\; \nodevar{\tau}{0} + \nodevar{\xi}{0}
        \label{eq:scentree-cost}\\
        &\stt{}\;
            \nodevar{x}{0}=x, \nodevar{\md}{0} = \md, \nodevar{x}{\iota_\pplus} = f(\nodevar{x}{\iota}, \nodevar{u}{\iota}, \nodevar{\md}{\iota_\pplus}),
            \label{eq:scentree-dynamics}  
            \\
        &\phantom{\stt{}\;{}} \ell(\nodevar{x}{\iota}, \nodevar{u}{\iota},\nodevar{\md}{\iota}) \leq \nodevar{\tau}{\iota},
        \label{eq:scentree-stage-cost-bounds}  \\
        &\phantom{\stt{}\;{}} \Vf(\nodevar{x}{\iota_\hor}, \nodevar{\md}{\iota_\hor}) \leq \nodevar{\xi}{\iota_{\hor}}+\nodevar{\tau}{\iota_{\hor}},
        \label{eq:scentree-term-cost-bounds}\\
        &\phantom{\stt{}\;{}}(\nodevar{\tau}{\iota_\pplus}+\nodevar{\xi}{\iota_\pplus}, \nodevar{\xi}{\iota}) \in \epi \lrisk{\nodevar{\lrn}{\iota}}{\nodevar{\md}{\iota}}{\nodevar{\conf}{\iota}},
        \label{eq:scentree-epigraph-cost}\\
        &\phantom{\stt{}\;{}} \big( g(\nodevar{x}{\iota}, \nodevar{u}{\iota}, \nodevar{\md}{\iota}, \nodevar{\md}{\iota_\pplus}), 0 \big) \in \epi \lriskc{\nodevar{\lrn}{\iota}}{\nodevar{\md}{\iota}}{\nodevar{\bar \conf}{\iota}}{\nodevar{\hat{\alpha}}{\iota}},
        \label{eq:scentree-epigraph-constraints}\\
        &\phantom{\stt{}\;{}} (\nodevar{x}{\iota_\hor}, \nodevar{\lrn}{\iota_\hor}, \nodevar{\confb}{\iota_\hor}, \nodevar{w}{\iota_\hor}) \in \DR{\Xf},
        \label{eq:scentree-terminal-constraints}
    \end{align}
    \end{subequations}
    for $\iota \in \nodes{[0,\hor-1]}$, $\iota_{\pplus} \in \child{\iota}$, and $\iota_{\hor} \in \nodes{\hor}$, where $\xseq = (\nodevar{x}{\iota})_{\iota \in \nodes{[0,\hor]}}$ and $\useq$ as defined above.
    If the parameters $(\nodevar{s}{\iota}, \nodevar{\beta}{\iota})$ satisfy for all $\iota \in \nodes{[0,\hor-1]}$ that 
    \begin{equation} \label{eq:params-dynamics}
        \nodevar{\lrn}{\iota_\pplus} = \learner(
            \nodevar{\lrn}{\iota}, \nodevar{\confb}{\iota}, \nodevar{\md}{\iota}, \nodevar{\md}{\iota_\pplus})
        \text{ and }
        \nodevar{\confb}{\iota_\pplus} = \confdyn(
            \nodevar{\confb}{\iota}),
    \end{equation}
    then the optimal cost of \eqref{eq:scenario-tree-reformulation} is equal to $\DR{V}_\hor(z)$.
\end{proposition}
\begin{proof}
    The claim is a direct application of the results in \cite{sopasakis_risk-averse_2019c}.
\end{proof}
If 
\begin{inlinelist*}
    \item the costs $\ell(\argdot, \argdot, \md)$,
$\Vf(\argdot, \md)$, the constraint mappings $g(\argdot,\argdot, \md,\mdnxt)$ and terminal set $\DR{\Xf}$ are convex
    \item the ambiguity sets $\amb_{\nodevar{\conf}{\iota}}(\nodevar{\lrn}{\iota}, \nodevar{\md}{\iota})$ are 
    conic representable
    \item the dynamics $f(\argdot, \argdot, \md)$ are affine
\end{inlinelist*}
for all $\md \in \W$, then it follows from \Cref{prop:conic-robust-avar} that both $\lrisk{\nodevar{\lrn}{\iota}}{\nodevar{\md}{\iota}}{\nodevar{\conf}{\iota}}$ and $\lriskc{\nodevar{\lrn}{\iota}}{\nodevar{\md}{\iota}}{\nodevar{\bar \conf}{\iota}}{\nodevar{\hat{\alpha}}{\iota}}$ are conic risk measures and thus 
\eqref{eq:scenario-tree-reformulation} can be reduced to a convex conic program. See \Cref{sec:numerical} for a numerical 
illustration, as well as \cite{schuurmans_LearningBasedRiskAverseModel_2020} for a case study in 
a slightly simplified setting.
Note that since the learner and confidence dynamics $\learner$ and $\confdyn$ are eliminated before solving the optimization problem, they need not be affine for the problem to remain convex. For nonlinear dynamics $f(\argdot, \argdot, \mdnxt)$, the problem is no longer convex but can in practice still be solved effectively with standard NLP solvers.

\subsection{Technical Lemma}

\begin{lem}[Infimum convergence] \label{lem:epi-convergence}
    Consider a sequence of proper, \ac{lsc} functions $\cost^{(t)} : \Re^{n} \rightarrow \barre$, $t \in \N$ and a proper, \ac{lsc}, level-bounded function $\cost : \Re^{n} \rightarrow \barre$. 
    Suppose that
    \begin{conditions}
        \item \label{cond:eventual-upper} (Eventual upper bound) there exists a $T \in \N$, such that for all $t>T$, and for all $u$, $\cost^{(t)}(u) \geq \cost(u)$;
       \item \label{cond:pwise} (Pointwise convergence) $\cost^{(t)} \pto \cost$. That is, for all $u$, $\lim_t \cost^{(t)}(u) = \cost(u)$.
    \end{conditions}
    Then, $\lim_{t} \inf_{u} \cost^{(t)}(u) = \inf_{u} \cost(u)$.
\end{lem}

\begin{proof}
    By \ref{cond:eventual-upper} it follows that for any sequence $u_t \to \bar{u}$, 
    \[
        \liminf_{t} \cost^{(t)}(u_t) = \liminf_{\stackrel{u \to \bar{u}}{t \to \infty}}\cost^{(t)}(u) 
    \geq \liminf_{u \to \bar{u}} \cost(u) \geq \cost(\bar{u}),
    \]
    where the first inequality follows from \Cref{cond:eventual-upper},
    and the second inequality follows from lower semicontinuity of $\cost$. 
    Moreover, fixing $(u_t)_{t\in\N}$ to be the constant sequence $u_t = \bar{u}$, it follows from \ref{cond:pwise} that $\limsup_t \cost^{(t)}(u_t) = \lim_t \cost^{(t)}(\bar{u}) \leq \cost(\bar{u})$. Invoking \cite[Prop. 7.2]{rockafellar_VariationalAnalysis_1998}, we conclude that $\cost^{(t)} \eto \cost$, i.e., $\cost^{(t)}$ epi-converges to $\cost$. 
    Secondly, from \Cref{cond:eventual-upper} and the level-boundedness of 
    $\cost$, it follows that $(\cost^{(t)})_{t\in \N}$ is eventually level-bounded \cite[Ex. 7.32]{rockafellar_VariationalAnalysis_1998}.
    The claim then follows from \cite[Thm. 7.33]{rockafellar_VariationalAnalysis_1998}.
\end{proof}
\fi 
\subsection{Deferred proofs}

\begin{appendixproof}{lem:DR-MSS}
    Let $(z_t)_{t\in\N}=(x_t, \lrn_t, \confb_t, \md_t)_{t\in\N}$ denote 
    the stochastic process satisfying dynamics \eqref{eq:closed-loop}, for some initial state $z_0 \in \dom \cost$. For ease of notation, let us define $\cost_t \dfn \cost(z_t), t \in \N$. 
    Due to nonnegativity of $V$, 
    \[
        \begin{aligned}
            \E\left[\tsum_{t=0}^{k-1} c \nrm{x_t}^2\right]
            &\leq \E\left[V_k + \tsum_{t=0}^{k-1} c \nrm{x_t}^2\right]\\
            & =\E\left[V_k - V_0 + \tsum_{t=0}^{k-1} c \nrm{x_t}^2\right] + V_0,
        \end{aligned}
    \]
    where the second equality follows from the fact that $\cost_0$ is deterministic. By linearity of the expectation, we can in turn write
    \[
    \begin{aligned}
        \E\big[V_k{-}V_0{+}c \tsum_{t=0}^{k-1} \nrm{x_t}^2 \big] 
                &= \E\left[\tsum_{t=0}^{k-1} V_{t+1}{-}V_{t} {+} c \nrm{x_t}^2 \right]\\  
                &= \tsum_{t=0}^{k-1} \E \left[ V_{t+1} {-} V_t{+} c \nrm{x_t}^2 \right].
    \end{aligned}
    \]
    Therefore,
    \begin{equation} \label{eq:step-proof}
        \begin{aligned}
            \E\big[c \tsum_{t=0}^{k-1} \nrm{x_t}^2\big]&{-}V_0
            \leq
            \tsum_{t=0}^{k-1} \E \left[V_{t+1}{-}V_t \right]{+}c \E \left[ \nrm{x_t}^2 \right].
        \end{aligned}
    \end{equation}
    Recall that $\conf_t$ denotes the coordinate of $\confb_t$ corresponding to the risk measures in the cost function \eqref{eq:cost-function-risk}. Defining the event
    \(
        E_t \dfn \{
                \omega \in \Omega \mid
                    \row{\transmat}{\md_t(\omega)} \in
                    \amb_{\conf_{t}}(\lrn_t(\omega), \md_t(\omega))
            \},
    \)
    and its complement $\lnot E_t = \Omega \setminus E_t$, we can use the law of total expectation to write 
    \begin{multline*}
        \E\left[ V_{t+1} - V_t \right]
        = \E\left[ V_{t+1} - V_{t} \mid E_t \right] {\prob[E_t]}\ifSingleColumn\else\\\fi
        +\E\left[ V_{t+1} - V_{t} \mid \lnot E_t \right] {\prob[\lnot E_t]}.
    \end{multline*}
    By condition \eqref{eq:high-confidence}, $\prob[\lnot E_t] < \conf_t$. From \Cref{cond:RPI,cond:boundedness}, it follows that $z_t \in \dom \cost$, $\forall t \in \natseq{0}{k}$ and that there exists a $\dV \geq 0$ such that $V(z) \leq \dV$, for all $z \in \dom V$. Therefore, $\E[V_{t+1} - V_t \mid \lnot \E_t] \leq \dV$.
    Finally, by \Cref{cond:lyap-decrease}, 
    \(
        \E\left[ V_{t+1}- V_{t} \mid E_t \right] 
        \leq \E[-c \nrm{x_t}^2 \mid E_t ].
    \) Thus,
    \extrastep{
        This follows from explicitly writing out the expectation 
        \[ 
            \begin{aligned}
                \E[V_{t+1} - V_t] &= \E[\smashoverbracket{\E[V_{t+1} \mid V_t] }{\leq V_t - c \nrm{x}^2} ] - \E[V_t] \\
                &\leq \E[V_t - c \nrm{x}^2] - \E[V_t] \\
                &= - c \nrm{x}^2 
            \end{aligned}
        \]
    }
    \begin{align*}
        \E\left[ V_{t+1} - V_t \right] 
        &\leq \E\left[-c \nrm{x_t}^2 \mid E_t\right] \prob[E_t] + \dV \conf_t.
    \end{align*}
    This allows us to simplify expression \eqref{eq:step-proof} as
    \begin{align*}
        &\E\left[c \tsum_{t=0}^{k-1} \nrm{x_t}^2\right] -V_0\\
        &\leq \tsum_{t=0}^{k-1}-c \E\left[\nrm{x_t}^2 \mid E_t\right] \prob[E_t] + \dV \conf_t + c \E \left[ \nrm{x_t}^2 \right]\\  
        &\leq \tsum_{t=0}^{k-1} -c \E\left[\nrm{x_t}^2 \mid E_t\right] \prob[E_t] + \dV \conf_t \\
        &\qquad
         + c \E \left[ \nrm{x_t}^2 \mid E_t \right] \prob[E_t] 
         +c \E \left[ \nrm{x_t}^2 \mid \lnot E_t \right] \prob[\lnot E_t]\\  
        &= \tsum_{t=0}^{k-1} \dV \conf_t +c \E \left[ \nrm{x_t}^2 \mid \lnot E_t \right] \prob[\lnot E_t]\\
        &\leq \tsum_{t=0}^{k-1} \conf_t (\dV +  c \E\left[\nrm{x_t}^2 \mid \lnot E_t \right] ).
    \end{align*}
    Since \revision{$\dom \cost$}{$\dom \cost(\argdot, \lrn_t, \confb_t, \md_t)$} was assumed to be compact and to contain the origin, there exists an 
    $r \geq 0$ such that $\nrm{x_t}^{2} \leq r,\, \forall t \in \natseq{0}{k}$. Therefore,   
    \begin{align*}
        \E\left[\tsum_{t=0}^{k-1} \nrm{x_t}^2\right]
        &\leq \tfrac{V_0}{c} + \left(\tfrac{\dV}{c} + r \right) \tsum_{t=0}^{k-1} \conf_t,
    \end{align*}
    which remains finite as $k \to \infty$, since $(\conf_t)_{t \in \N}$ is summable. Thus, necessarily $\lim_{t\to \infty} \E[\nrm{x_t}^2] = 0$.
\end{appendixproof}

\begin{appendixproof}{thm:MPC-stability}
    First, note that using the monotonicity of coherent risk measures \cite[Sec. 6.3, (R2)]{shapiro2009lectures}, a straightforward inductive argument allows us to show that 
    under \Cref{cond:TVfleqVf},
    \begin{equation} \label{eq:monotonicity-T}
        \T \DR{V}_\hor \leq \DR{V}_\hor, \quad \forall \hor \in \N.
    \end{equation}
    Since $\bar{\mathcal{Z}} \subseteq \dom \DR{\cost}_\hor$, recall that by definition \eqref{eq:definition-DP}, we have for any $z = (x,\lrn, \confb, \md) \in \bar{\mathcal{Z}}$ that
    \begin{equation*}
        \hat{\cost}_N(z) = \ell(x,\DR{\law}_\hor(z), \md) 
        + \lrisk{\md}{\lrn}{\conf}\big[
            \hat{\cost}_{N-1}
            \big(
                \fc(z,\mdnxt), 
                \mdnxt
            \big)
        \big],   
    \end{equation*}
    where $\conf$ denotes the component of $\confb$ corresponding to the cost. 
    Therefore, we may write   
    \begin{align*}
        &\lrisk{\md}{\lrn}{\conf}\left[
            \hat{\cost}_\hor(\fc(z,\mdnxt), \mdnxt)
        \right] - 
        \hat{\cost}_\hor(z)  \\ 
        &=\lrisk{\md}{\lrn}{\conf}\left[
            \hat{\cost}_\hor(\fc(z,\mdnxt), \mdnxt)
        \right] 
        - \ell(x,\DR{\law}_\hor(z), \md) 
        \\ &\; - \lrisk{\md}{\lrn}{\conf}
        \big[
        \hat{\cost}_{N-1}
        \big(
            \fc(z,\mdnxt), \mdnxt
        \big)
        \big]
        \leq - \ell (x, \DR{\law}_\hor(z), \md) \leq - c \nrm{x}^2,
    \end{align*}
    where the first inequality follows by \eqref{eq:monotonicity-T} and monotonicity of coherent risk measures. The second inequality follows from \Cref{cond:stage-cost-bound}.
    Combined with \Cref{cond:locally-bounded}, this implies that 
    $\cost: z \mapsto \DR{\cost}_\hor(z) + \delta_{\bar{\mathcal{Z}}}(z)$ satisfies the conditions of \Cref{lem:DR-MSS} and the assertion follows. 
\end{appendixproof}

\rerevision{}{
For the following, it will be convenient to define 
$\Vtt{t}$ and $\Vtil$ as
\begin{equation} \label{eq:def-q-func}
    \begin{aligned}
    &\Vtt{t}(x,u,\md) \dfn \ell(x,u,\md)
                + \lrisk{\lrn_t}{\md}{\conf_{t}}[\Vht{t+1}[\hor-1](f(x, u, \mdnxt), \mdnxt)], \\
    &\Vtil(x,u,\md) \dfn \ell(x,u,\md)
    + \E_{\row{\transmat}{\md}}[\cost_{\hor-1}(f(x, u, \mdnxt), \mdnxt) {\mid} x, \md], 
    \end{aligned}
\end{equation}
and let 
\( 
    \Uht{t}(x, \md) \dfn \DR{\Ufeas}(x, \lrn_t, \confb_t, \md),
\)
so we may write
\begin{subequations}
    \begin{align}\label{eq:definition-vt-from-q-func}
        \Vht{t}(x,\md) &= \inf_{u \in \Uht{t}(x, \md)} \Vtt{t}(x,u,\md)
    \text{ and }\\
        \cost_{\hor}(x,\md) &= \inf_{u \in \Ufeas(x,\md)} \Vtil(x,u,\md).
    \end{align}
\end{subequations}

\begin{appendixproof}{thm:out-of-sample}

    We will show \eqref{eq:upper-approximation-hp} by induction on $\hor$.
    For $\hor=0$, we have $\Vht{t}[0] \equiv \Vf \equiv \cost_0$, thus, 
    \eqref{eq:upper-approximation-hp} holds with 
    $\probt[t][0] = 0$,
    $\forall t \in \N$, and the claim holds trivially.
    For the induction step, we define the events
    \begin{subequations}
        \begin{align}
            \label{eq:proof-def-evtCov}
            \evtCov &\dfn \{ 
            \row{\transmat}{\md} 
                \in
            \cap_{i=1}^{\nbeta}
                \amb_{\beta_{t,i}}(\lrn_t, \md), 
                \; \forall \md \in \W
            \}, \\ 
            \evtCost &\dfn \{ 
            \Vht{t} (x,\md) \geq \cost_N(x,\md), \; \forall (x, \md) \in \Re^{\ns} \times \W\}, 
        \end{align}
    \end{subequations}
    for $\hor \in \N, t \in \N$.
    The induction hypothesis now reads  
    \begin{equation} \label{eq:proof-costbound-induction-hypothesis}
        \prob[\evtCost[t][\hor-1]] \geq \probt[t][\hor-1] = \nModes \tsum_{k=0}^{N-1} \nrm{\confb_{t+k}}_1   ,\quad \forall t \in \N,
    \end{equation}
    and our goal is to show that this implies that $\prob[\evtCost] \geq \probt[t][\hor]$, $\forall t \in \N$.
    
    Given the occurrence of event $\evtCost[t+1][\hor-1]$, 
    the monotonicity of risk measures\cite[Sec. 6.3, (R2)]{shapiro2009lectures}
    ensures that 
            \(
            \Vtt{t}(x,u,\md) \geq \ell(x,u,\md)
                        + \lrisk{\lrn_t}{\md}{\conf_{t}}[\cost_{\hor-1}(f(x, u, \mdnxt), \mdnxt)],
            \)
    for all $(x, u, \md) \in \Re^{\ns} \times \Re^{\na} \times \W$, and $t \in \N$.
    Furthermore, conditional on event $\evtCov$, \eqref{eq:risk-upper-bound} implies
    that 
    \(
            \lrisk{\lrn_t}{\md}{\conf_t} \geq \E_{\row{\transmat}{\md}} \text{ and }
            \lriskc{\lrn_t}{\md}{\bar\conf_t}{\DR{\alpha}_t} \geq \AVAR^{\row{\transmat}{\md}}_{\alpha}
    ,\)
    uniformly.
    Combining this fact with 
    \eqref{eq:chance-constraint} and \eqref{eq:DR-constraints},
    we obtain the implication
    \[ 
            \begin{aligned}
                    \evtCost[t+1][\hor-1], 
                    \evtCov[t]
                \implies 
                \begin{cases}
                    \Vtt{t}(x,\md,u) \geq \Vtil(x,\md,u),\\
                    \Uht{t}(x, \md) \subseteq \Ufeas(x, \md),\\
                \end{cases}
            \end{aligned}
    \]
    $\forall (x, u, \md) \in \Re^{\ns} \times \Re^{\na} \times \W$, 
    and hence,
    \begin{equation*}
        \begin{aligned} 
            \Vht{t}(x, \md) &= \min_{u \in \Uht{t}(x,\md)} \Vtt{t}(x,u,\md)
                \geq \min_{u \in \Uht{t}(x,\md)} \Vtil(x,u,\md)\\
                &\geq \min_{u \in \Ufeas(x,\md)} \Vtil(x,u,\md) = \cost_{N}(x,\md), 
        \end{aligned}
    \end{equation*}
    which describes exactly the event $\evtCost[t][\hor]$. 
    Thus, we have shown that 
    $\prob\big[\evtCost[t][\hor]] \geq \prob[\evtCov, \evtCost[t+1][\hor-1]\big]$.
    By the union bound, we now obtain
    \begin{equation} \label{eq:proof-induction-step-result}
        \begin{aligned}
            \prob[\evtCost]
            &\geq 
                \prob\big[\evtCov, \evtCost[t+1][\hor-1]\big]
            \geq 
                1 - \big(\prob[\lnot \evtCov] + \prob[\lnot \evtCost[t+1][\hor-1]]\big) \\
            &\geq
                1 - \big(\nModes  \nrm{\confb_t}_1 + \probt[t+1][\hor-1]),
        \end{aligned}
    \end{equation}
    where in the final inequality, $\prob[\lnot \evtCost[t+1][\hor-1]]$
    was bounded using the 
    induction hypothesis \eqref{eq:proof-costbound-induction-hypothesis}
    and $\prob[\lnot \evtCov]$ was replaced by another application of the union bound:
    \begin{align}
            \notag
            \prob[\lnot \evtCov] &= \prob \big[
                \exists w \in \W: 
                \row{\transmat}{\md} 
                    \notin
                \cap_{i=1}^{\nbeta} \amb_{\conf_{t,i}}(\lrn_t, \md)
            \big] \\ 
            \notag
            &\leq 
                \tsum_{w \in \W} 
                \tsum_{i = 1}^{\nbeta} 
                \prob [
                \row{\transmat}{\md} 
                    \notin
                \amb_{\beta_{t,i}}(\lrn_t, \md) 
                ]\\
            \label{eq:proof-bound-coverage}
            &\leq \nModes \tsum_{i=1}^{\nbeta}\conf_{t,i} = \nModes \nrm{\confb_t}_{1}.
    \end{align}
    Thus, substituting the expression for $\probt[t+1][\hor-1]$ from the induction hypothesis 
    \eqref{eq:proof-costbound-induction-hypothesis} into the result 
    \eqref{eq:proof-induction-step-result}, we 
    obtain that \eqref{eq:upper-approximation-hp} holds with 
    \[  
        \begin{aligned}
          \probt &\dfn d \nrm{\confb_t}_1 + \probt[t+1][\hor-1] = d \tsum_{k=0}^{\hor} \nrm{\confb_{t+k}}_1,
        \end{aligned}
    \]
    which establishes \ref{state:general}.
    Under the conditions of \ref{state:concentric},
    namely that \eqref{eq:concentric-ambiguity} holds, it follows from 
    definition \eqref{eq:proof-def-evtCov} that 
    \begin{equation} \label{eq:proof-replacement}
         \prob \big[
            \lnot \evtCov
    \big] = 
    \prob \big[ \row{\transmat}{\md} 
    \notin \amb_{\beta_t^\star}(\lrn_t, \md) \big] \leq \nModes \beta_t^\star, 
    \end{equation}
    with $\beta^\star_t \dfn \max_{i \in \natseq{1}{\nbeta}} \{ \conf_{t,i} \} = \nrm{\confb_{t}}_{\infty}$. 
    \Cref{state:concentric}
    is then established by the same inductive argument,
    replacing the expression for $\probt[t][\hor-1]$ in \eqref{eq:proof-costbound-induction-hypothesis},
    and replacing \eqref{eq:proof-bound-coverage} with \eqref{eq:proof-replacement}.
\end{appendixproof}
}

\begin{appendixproof}{thm:consistency-hard-constraints}
    By \Cref{asm:assumption-consistency}, 
    we have for $\hor = 0$ that $\Vht{t}[0] \equiv \cost_{0} \equiv \bar{\Vf}$ and there is nothing to prove.
    The general case, $\hor > 0$, is proved by induction.
    Assume that \eqref{eq:convergence} holds for some $\hor \geq 0$.
    We will now demonstrate that this implies that it also holds for 
    $\hor+1$.
    \rerevision{
    Under \Cref{assum:confidences}, the Borel-Cantelli lemma \cite[Thm. 4.3]{billingsley_ProbabilityMeasure_1995}
    guarantees that with probability 1, there exists a finite $T_\hor \in \N$,
    such that for all $t > T_\hor$, $\row{\transmat}{\md} \in \amb_{\idx{\confb_t}{i}}(s_t,\md)$,
    for all $\md \in \W$ and $i \in \natseq{1}{\nbeta}$,
    and consequently $\lrisk{\lrn_t}{\md}{\conf_t} \geq \E_{\row{\transmat}{\md}}$, uniformly.}{}
    To this end, we will show that the sequence 
    $(\Vtt{t}[\hor + 1](x, \argdot, \md))_{t \in \N}$ and the function $\Vtil[\hor+1](x,\argdot,\md)$, satisfy the conditions of \Cref{lem:epi-convergence}.
    Under \Cref{assum:regularity}, and using \cite[Thm. 3.31]{rockafellar_VariationalAnalysis_1998}, it follows from
    \cite[Prop. 2]{sopasakis_RiskaverseModelPredictive_2019} that $\Vtil[\hor]$ and $\Vtt{t}[\hor+1]$,
    are proper, \ac{lsc}, and level-bounded in $u$ locally uniformly in $x$,
    for all $\md \in \W$.

    Let us introduce the shorthand for the worst-case conditional distribution $p^{\star}_t(u) = (\idx{p^\star_{t}}{\mdnxt}(u))_{\mdnxt \in \W}$:  
    \[
        p^{\star}_{t}(u) \dfn \argmax_{p \in \amb_{\conf_{t}}(\lrn_t, \md)} \sum_{\mdnxt \in \W} \idx{p}{\mdnxt} \Vht{t+1}(f(x,u,\mdnxt),\mdnxt),
    \]
    where we have omitted the dependence on the constant $x$ and $\md$.
    \rerevision{Under \Cref{assum:confidences,}}{
    Combining \Cref{cor:asymptotic-performance-bound} with \Cref{assum:confidences}, 
    }
    the Borel-Cantelli lemma \cite[Thm. 4.3]{billingsley_ProbabilityMeasure_1995}
    guarantees that w.p. 1, there exists a finite $T_\hor \in \N$,
    such that for all $t > T_\hor$, $\row{\transmat}{\md} \in \amb_{\idx{\confb_t}{i}}(s_t,\md)$,
    for all $\md \in \W$ and $i \in \natseq{1}{\nbeta}$,
    \rerevision{}{and furthermore, $\Vht{t} \geq \cost_\hor$, which implies 
    that $\Vtt{t}[\hor+1] \geq \Vtil[\hor+1]$}.
    Moreover, by the induction hypothesis (i.e., \eqref{eq:convergence} holds for $\hor$), 
    there exists for every $\epsilon > 0$, a $T_\epsilon \geq T_{\hor}$, such that for all $t > T_\epsilon$,
    \begin{align} 
        \nonumber
            &
            \Vtt{t}[\hor+1](x,u,\md) - \Vtil[\hor+1](x,u,\md)
            \\\nonumber
            &= \sum_{\mdnxt \in \W} p^{\star}_{t,\mdnxt}(u) \Vht{t+1}[\hor](f(x,u,\mdnxt),\mdnxt) - \elem{\transmat}{\md}{\mdnxt} \cost_{\hor}(f(x,u,\mdnxt),\mdnxt)\\\nonumber
            &\stackrel{\eqref{eq:convergence}}{\leq} \sum_{\mdnxt \in \W} p^{\star}_{t,\mdnxt}(u) (\cost_\hor(f(x,u,\mdnxt),\mdnxt) + \epsilon) - \elem{\transmat}{\md}{\mdnxt} \cost_{\hor}(f(x,u,\mdnxt),\mdnxt)\\\nonumber
            &= \tsum_{\mdnxt \in \W} (p^{\star}_{t,\mdnxt}(u) - \elem{\transmat}{\md}{\mdnxt})V_\hor(f(x,u,\mdnxt),\mdnxt) + p^{\star}_{t,\mdnxt}(u) \epsilon\\
            \extrastep{&\leq \sum_{\mdnxt \in \W} |p^{\star}_{t,\mdnxt}(u) - \elem{\transmat}{\md}{\mdnxt}) | V_\hor(f(x,u,\mdnxt),\mdnxt) + p^{\star}_{t,\mdnxt}(u)\\}
            &\leq \tsum_{\mdnxt \in \W} \ambdia_t V_\hor(f(x,u,\mdnxt),\mdnxt) + \epsilon,\label{eq:lim}
    \end{align}
    where the final inequality is due to \Cref{assum:amb-decrease} and the fact that for all $t > T_\hor$, $\row{\transmat}{\md} \in \amb_{\conf_{t}}(\md, \lrn_t)$.
    As $\ambdia_t \to 0$, the first term in \eqref{eq:lim} can be made arbitrarily small by increasing $t$,
    provided that $\cost_{\hor}(f(x,u,\mdnxt),\mdnxt) < \infty$, for all $\mdnxt \in \W$, 
    hence establishing pointwise convergence $\Vtt{t}[\hor+1] \pto \Vtil[\hor+1]$
    whenever $\dom \cost_{\hor}$ is \ac{RCI} for \eqref{eq:system-dynamics}, 
    which in turn holds if $\Xf$ is \ac{RCI} by \Cref{prop:recursive-feasibility}. 
    The sequence $(\Vtt{t}[\hor + 1](x, \argdot, \md))_{t\in \N}$
    and the function $\Vtil[\hor+1](x,\argdot,\md)$ thus satisfy the conditions of \Cref{lem:epi-convergence},
    which establishes \eqref{eq:convergence} for $\hor +1$.
\end{appendixproof}

\end{appendix}
\end{document}